\documentclass[aos,MSNbibl]{imsart}
\RequirePackage{amsthm,amsmath,amsfonts,amssymb}
\RequirePackage[numbers]{natbib}
\RequirePackage[colorlinks,citecolor=blue,urlcolor=blue]{hyperref}
\providecommand{\arxivurl}[1]{\href{https://arxiv.org/abs/#1}{arXiv:#1}}
\providecommand{\bnote}[1]{#1}
\RequirePackage{graphicx}
\RequirePackage{float}

\startlocaldefs
\newcommand{\rrvert}{\vert}
\newcommand{\rrVert}{\Vert}
\newcommand{\llvert}{\vert}
\newcommand{\llVert}{\Vert}
\theoremstyle{plain}
\newtheorem{theorem}{Theorem}[section]
\newtheorem{lemma}[theorem]{Lemma}
\newtheorem{proposition}[theorem]{Proposition}
\theoremstyle{definition}
\newtheorem{condition}[theorem]{Condition}
\newtheorem{remark}[theorem]{Remark}
\usepackage{xcolor}

\def\emptyset{\varnothing}

\endlocaldefs

\begin{document}
\begin{frontmatter}

\title{Inferring diffusivity from killed diffusion}
\runtitle{Inference for killed diffusion}
%

\begin{aug}
%
%
\author{\fnms{Richard}~\snm{Nickl}\ead[label=e1]{nickl@maths.cam.ac.uk}}
\author{\fnms{Fanny}~\snm{Seizilles}\ead[label=e2]{fps25@cam.ac.uk}}
%
%
\address{Department of Pure Mathematics and Mathematical Statistics,
University of Cambridge\printead[presep={,\ }]{e1,e2}}
\end{aug}

\received{\smonth{3} \syear{2025}}
\revised{\smonth{12} \syear{2025}}

\begin{abstract}
We consider diffusion of independent molecules in an insulated Euclidean
domain with unknown diffusivity parameter. At a random time and position,
the molecules may bind and stop diffusing in dependence of a given ``binding
potential.'' The binding process can be modeled by an additive random functional
corresponding to the canonical construction of a ``killed'' diffusion Markov
process. We study the problem of conducting inference on the infinite-dimensional
diffusion parameter from a histogram plot of the ``killing'' positions
of the process. We show first that these positions follow a Poisson point
process whose intensity measure is determined by the solution of a Schr\"odinger-type
equation. The inference problem can then be recast as a nonlinear inverse
problem for this partial differential equation, which we show to be consistently
solvable in a Bayesian way under natural conditions on the initial state
of the system, provided the binding potential is not too ``aggressive.''
In the proofs, we obtain novel posterior contraction rate results for high-dimensional
Poisson count data that are of independent interest. A~numerical illustration
of the estimator by MCMC methods is also provided.
\end{abstract}


\end{frontmatter}
\setcounter{tocdepth}{2}
\tableofcontents
\section{Introduction}
\label{sec1}

\subsection{Motivation}
\label{sec1.1}

In a variety of imaging tasks in biochemistry (e.g.,
\cite{heltberg_physical_2021,heckert_recovering_2022,basu_live-cell_2021}),
one observes molecules diffusing in a ``nucleus'' $\Omega $---a region
that contains genetic material and which is enclosed by a membrane. Inside
such nuclei, certain proteins (such as chromatin) may be concentrated in
subcompartments, called ``foci.'' Diffusing molecules can be seen to bind
(``get trapped'') at random times, to then stop diffusing (or to visibly
diffuse much slower). The binding probability is higher in regions of high
protein concentration, but diffusion is still visible inside such foci,
and it is a scientifically relevant question whether diffusive behavior
of unbound molecules inside of foci is substantially different from outside.
The presence of binding events makes it difficult to discern such effects
with high-resolution microscope measurements, and the main ideas of this
article are concerned with the task of inferring the diffusivity parameter
$D$ of the underlying model from such data. To do this, we propose a Markovian
``binding'' model, where a motion of the molecule evolves according to
the basic diffusion equation (\ref{eq:diffuso}) below and where binding
is modeled by a ``killing'' time $S$, with a binding potential $q$ that
is given and which models the (observable) density of chromatin inside
of $\Omega $. We will show that inference on $D$ is possible from measuring
only the relative proportions of binding positions, which avoids more complex
tracking measurements for individually diffusing molecules.

\subsection{Observation model and main results}
\label{sec1.2}

Diffusion of particles in a bounded open set
$\Omega \subset \mathbb{R}^{d}$ (called a ``domain'') is mathematically
described by a stochastic differential equation (SDE) with reflection at
the boundary $\partial \Omega $: we consider the Markov process
$(X_{t}: t \ge 0)$ starting at initial distribution
$X_{0} \sim \phi $ and then evolving according to the SDE,
%
\begin{equation}
dX_{t}=\nabla D(X_{t})\,dt +\sqrt{2D(X_{t})}
\,dW _{t}+\nu (X_{t})\,dL _{t}, \quad t>0,
\label{eq:diffuso} 
\end{equation}
where $D : \Omega \to [D_{\mathrm{min}},\infty )$, $D_{\mathrm{min}}>0$, is a real-valued diffusion
parameter with Lipschitz gradient vector field $\nabla D$, and
$(W_{t})$ is a $d$-dimensional Brownian motion. Moreover, $(L_{t})$ is
a continuous nondecreasing process, with nonzero increments only when
$X_{t}$ is at $\partial \Omega $, and
$\nu (x), x \in \partial \Omega $, is the inward pointing normal vector.
To avoid technicalities, we assume that $\Omega $ is convex with a smooth
boundary---the existence of a process $(X_{t})$ with almost surely continuous
sample paths is then shown in
\cite{tanaka_stochastic_1979,lions_stochastic_1984}. One can model a further
drift term $\nabla U(X_{t})$ in the above SDE (see after (\ref{postD})
for discussion), but the main challenges already emerge when
$\nabla U=0$.

We will introduce (in fact, recall) a Markovian model for such binding
events based on the notion of a ``killed diffusion process'' (see
\cite{williams_markov_2000,bass_diffusions_1998,chung_brownian_1995,bass_stochastic_2011})
associated to a potential $q : \Omega \to [0,\infty )$. Larger values of
$q(x)$ indicate a higher probability of the process to bind when diffusing
near $x \in \Omega $. It induces an additive functional
$\int _{0}^{t} q(X_{s})\,ds $, for example, if $q$ equals the indicator
$1_{A}$ of a subset $A$ of $\Omega $, this measures the occupation time
of the $(X_{t})$ process in $A$ up until $t$. Then let $Y$ be a (standard)
exponential random variable independent of $(X_{t})$. The
\textit{binding time} $S$ is the first time the cumulative binding integral
of the potential along the path exceeds $Y$; that is,~$S$ is defined as
%
\begin{equation}
S=\inf \biggl\{t:
\int _{0}^{t} q(X_{s})\,ds > Y
\biggr\}. \label{eq:binding-time} 
\end{equation}
By properties of exponential random variables, this process is ``memory-less''
in the sense that not having bound by time $t$ does not affect the distribution
of binding events at later times. The killed diffusion process is defined
as
\begin{equation*}
\tilde{X_{t}} = %
\begin{cases}
X_{t} , \quad t < S,
\\
\dag , \quad t \ge S; \end{cases} %
\end{equation*}
so that $X_{t}$ is ``removed to the cemetery'' $\{\dagger \}$ at binding
time $S$. The ``binding'' location is given by $X_{S}$ encoding the terminal
position of the particle at time $S$ (see Figure \ref{fig1}). This captures the essence of the
biological mechanism (but is still a simplification as it does not model
the fact that molecules can also unbind, and may in fact not stop after
binding but just move much slower).

\begin{figure}
    \includegraphics[width=0.5\linewidth]{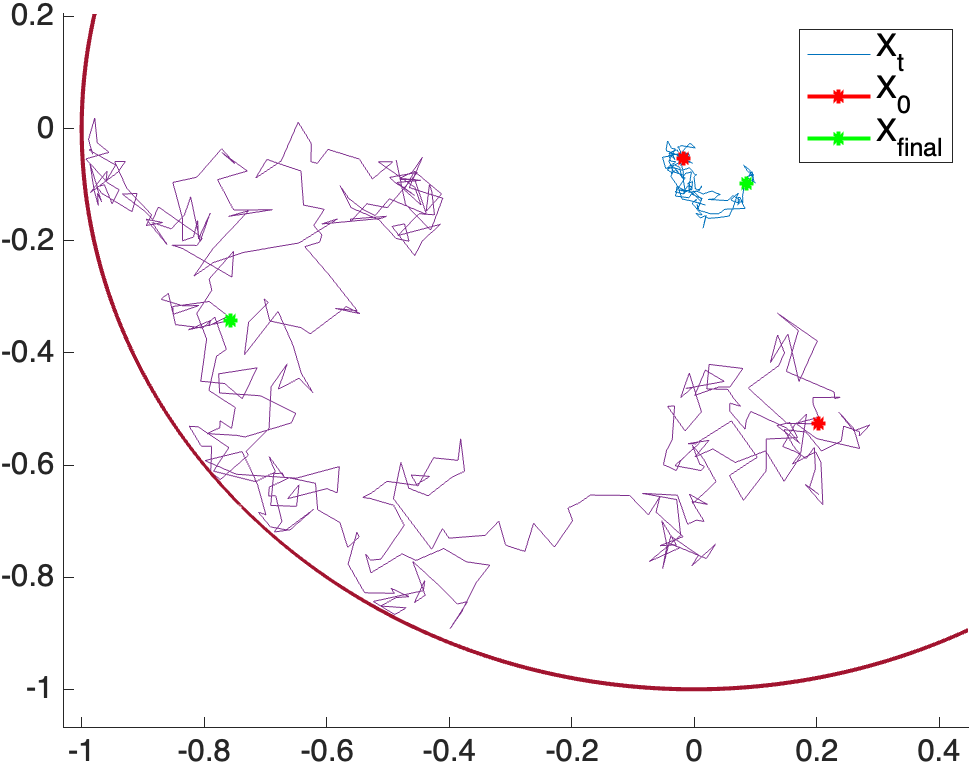}
\caption{Diffusion paths for 2 molecules. $X_{\mathrm{final}}$ denotes the terminal
position $X_{S}$ of the molecule after binding.}
\label{fig1}
\end{figure}

In the above mentioned applications from biochemistry (and elsewhere),
the aim is to make statistical inference on the functional parameter
$D$ from observing several molecules diffusing independently according
to (\ref{eq:diffuso}) in the presence of ``binding'' events. The equilibrium
distribution of the Markov process (\ref{eq:diffuso}) is uniform for all
$D$, and hence the occupation times spent by the process $X_{t}$ in particular
areas of $\Omega $ before binding are not informative for this task. If
one can track the trajectories of the particles, then methods such as those
from
\cite{nickl_consistent_2023,hoffmann_nonparametric_2024,giordano_statistical_2025}
(see also related contributions \cite{GHR04,NS17} for $d=1$ and for ``continuous''
data in \cite{NR20,GR22}) can be used to infer $D$ from the observed dynamics.
This generally requires a high-resolution imaging technique that can take
measurements for an extended period, but once a binding process is superimposed
(or in any case) it can become difficult or expensive to accurately collect
such data; many single particle tracking methods only allow for short path
measurements \cite{heckert_recovering_2022}.

A potentially much simpler observational signature is to take a dissection
of $\Omega $ into finitely many disjoint subsets
$B_{1}, \dots , B_{K}$ and to measure the relative frequency of binding
events $\{X_{S} \in B_{i}\}$ in each bin $B_{i}$.
\textit{The main contribution of this article is to demonstrate that merely
observing sufficiently dense histograms of the binding locations, rather
than the trajectories of diffusion, is enough to consistently identify
$D$.} This is shown to be true as long as (a) the initial state
$\phi $ of the system satisfies the mild identifiability Condition~\ref{ident} (of being either close to equilibrium, or appropriately prepared
by the experimenter), and (b) if the binding potential $q$ is not too ``aggressive''
(i.e., small enough relative to the parameter of the exponential variable
$Y$  {in a suitable sense}), but still strictly positive on  {$\Omega$}, so that all molecules
eventually bind at a finite time with probability one (see Proposition~\ref{prop:binding-time-finite}). If $n_{\mathrm{mol}}$ molecules
$X^{m}_{t}$, $m=1, \dots , n_{\mathrm{mol}}$, diffuse independently, denote by
%
\begin{equation}
\label{Ncount} N(A) = \sum_{m=1}^{n_{\mathrm{mol}}}
1\bigl\{X_{S}^{m} \in A\bigr\},\quad A \subset \Omega
~\text{(measurable)}, 
\end{equation}
the count of the number of molecules having bound in the subregion
$A$ of $\Omega $ (regardless of the point in time when binding has occurred).
It is physically plausible to assume that the counts for two disjoint subsets
$A$, $A'$ are statistically independent. In this case, a consistent model
requires $n_{\mathrm{mol}}$ to be randomly chosen from a Poisson distribution since
it is a consequence of R\'enyi's theorem that such ``completely random''
point processes are necessarily Poisson processes (see Theorem~6.12 in
\cite{last_lectures_2017} and also the discussion in Section~\ref{conppp}).
Our first main result---whose details and proof will be given in Section~\ref{biostory} under Conditions \ref{qcond} and \ref{Dphicond}---is the
following.

\begin{theorem}
\label{thm:poisson_process2}
For $n \in \mathbb N$, let $n_{\mathrm{mol}} \sim \operatorname{Poisson} (n)$ be the total number
of molecules diffusing independently according to (\ref{eq:diffuso}), each
with initial condition $X_{0} \sim \phi $ and binding at random time
$S$ with bounded potential $q \ge 0$ that is strictly positive on an open
subset $\Omega _{00}$ of $\Omega $. Then $N$ from (\ref{Ncount}) is a Poisson
point process with intensity measure $\Lambda $ defined by its Lebesgue
density $\lambda (x)=n q(x)u_{D,q}(x)$, $x \in \Omega $, where
$u=u_{D,q}$ solves the stationary Schr\"odinger equation
$ \nabla \cdot (D \nabla u) - qu = - \phi $ on $\Omega $ with Neumann boundary
conditions.
\end{theorem}

The connection between the intensity $\lambda $ and an elliptic partial
differential equation (PDE) revealed by the previous theorem allows to
recast the problem within the paradigm of Bayesian nonlinear inverse problems
\cite{stuart_inverse_2010,nickl_bayesian_2023} and to adapt techniques
from
\cite{monard_consistent_2021,nickl_bayesian_2023,nickl_consistent_2023}
to the present setting. We establish analytical properties of the nonlinear
PDE solution map $D \mapsto u_{D,q}$, including its injectivity under certain
hypotheses (discussed in detail after Theorem~\ref{yetagain}), to show
that this inverse problem can be solved in principle. We further obtain
in Section~\ref{pcont} a novel posterior contraction theorem for Poisson
count data, which is of independent interest for the theory of nonlinear
inverse problems, and from which we deduce the following.
%
\begin{theorem}%
\label{showoff}
Suppose $q$ is known, belongs to the Sobolev space  $H^{\eta}(\Omega)$ for some integer $\eta>1+d/2$, is strictly positive on $\Omega $, and 
sufficiently small. Assume moreover that the initial
condition $X_{0} \sim \phi $ satisfies the identifiability Condition~\ref{ident}. Then there exists an estimator $\hat D_{n}$ (arising from
the posterior mean of a Gaussian process prior) based on observations of
the Poisson point process from the previous theorem such that any positive
$D_{0} \in H^{\alpha}(\Omega )$, $\alpha >2+d$ that is constant near
$\partial \Omega $ can be recovered as $n \to \infty $ at convergence rate
\begin{equation*}
\llVert \hat D_{n} - D_{0} \rrVert _{L^{2}}
= O_{P_{D_{0}}}\bigl(n^{-\beta}\bigr)\quad \text{for some } \beta
>0.
\end{equation*}
\end{theorem}

See Theorem~\ref{thm:inference-diffusivity} for full details. The posterior
mean can be computed by standard MCMC and numerical PDE methods; see Figure~\ref{fig:intro} and after (\ref{postD}) for more discussion.

\begin{figure}
\includegraphics[width=\linewidth]{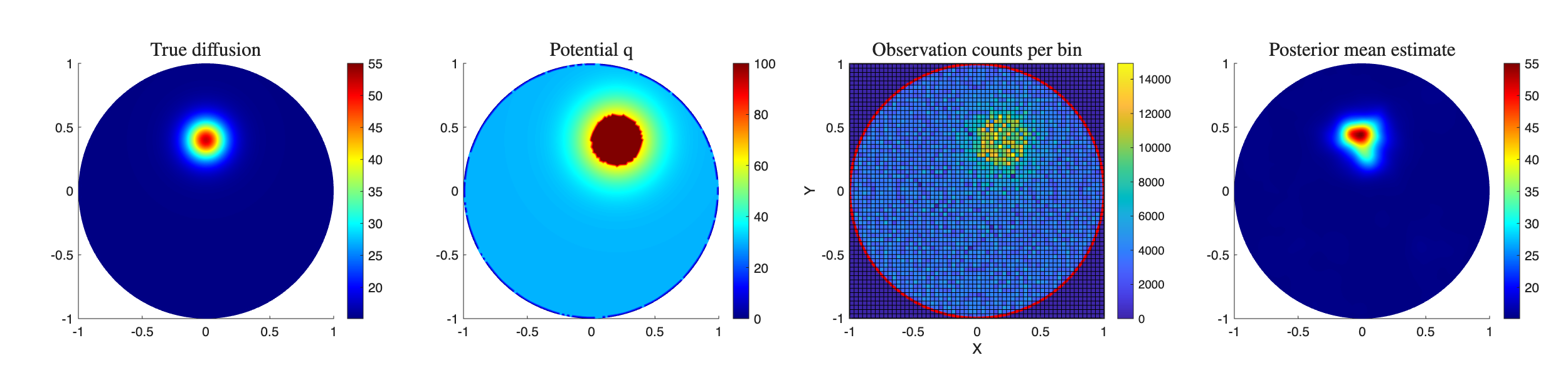}
\caption{Example of density reconstruction through MCMC, for
$n=10^{7}$. The true diffusivity function (1st) and killing potential (2nd)
generate the random observation counts (3rd) on each of the $K=2911$
bins. The posterior mean computed via MCMC is displayed on the right. For
more details about the implementation, see Section~\ref{subsec:numerics}.}
\label{fig:intro}
\end{figure}

\subsection{Basic notation}
\label{sec1.3}

The $L^{p}$, H\"older, $L^{2}$- and $L^{p}$-Sobolev spaces over a bounded
open set $\Omega $ in $\mathbb{R}^{d}$ with smooth boundary are denoted
by
$L^{p}(\Omega )$, $C^{a}(\Omega )$, $H^{a}(\Omega )$, $W^{a,p}(\Omega )$, respectively.
We denote by $\|\cdot \|_{B}$ the norm on the resulting Banach space
$B$, and $\langle \cdot , \cdot \rangle _{L^{2}}$ denotes the inner product
on $L^{2}$. By $\|H\|_{\infty}$, we always denote the supremum norm over
the domain of the map $H$, and $dx$ denotes Lebesgue measure on
$\Omega $, while $\nabla $, $\nabla \cdot $, $\Delta $ denote the gradient,
divergence and Laplace operator, respectively. We also denote by
$\mathbb{R}^{+}=[0,\infty )$.

We recall the well-known multiplication inequalities
%
\begin{equation}
\label{multi} \llVert u v \rrVert _{H^{a}} \lesssim \llVert u
\rrVert _{B^{a}} \llVert v \rrVert _{H^{a}},\quad u \in
B^{a}, v \in H^{a}, 
\end{equation}
where $B^{a}=C^{a}$ or, if $a>d/2$, $B^{a}=H^{a}$. We also make frequent
use of the standard interpolation inequalities for Sobolev spaces. First,
we have for all $u \in W^{a,p}$ that
%
\begin{equation}
\label{interpol} \llVert u \rrVert _{W^{j,p}} \lesssim \llVert u \rrVert
_{W^{a,p}}^{j/a} \llVert u \rrVert _{L^{p}}^{(a-j)/a},\quad 0
\le j \le a, 1\le p<\infty ; 
\end{equation}
see page135 in \cite{adams_sobolev_2003}. The following inequality will
also be useful (cf.~page~139 in \cite{adams_sobolev_2003}):
%
\begin{equation}
\label{interpol2} \llVert u \rrVert _{L^{2}} \lesssim \llVert u \rrVert
^{1-\theta}_{W^{a,1}} \llVert u \rrVert ^{\theta}_{L^{1}},
\quad u \in W^{a,1} 
\end{equation}
as long as $a>d/2$ and with $\theta = (a-d/2)/a$.

\section{Posterior consistency for high-dimensional Poisson count data}
\label{pcont}

In order to prove Theorem~\ref{showoff}, we need a general posterior contraction
theorem for Poisson count data with a finite but possibly large number
$K$ of ``bins'' $B_{1}, \dots, B_{K}$. This is the appropriate measurement
setting for various physical experiments in inverse problems (beyond the
context of this paper, also, e.g., with $X$-ray imaging,
\cite{monard_efficient_2019,monard_consistent_2021}). While there are
several contributions to posterior consistency with Poisson process data
(see
\cite{belitser_rate-optimal_2015,kirichenko_optimality_2015,donnet_posterior_2017,gugushvili_fast_2020,giordano_nonparametric_2023},
Section~10.4.4 in \cite{ghosal_fundamentals_2017}, as well as references
therein), we are not aware of a result that provides what we require. A
main challenge arises from the fact that the relevant data is just independent
but not i.i.d., that the sample size is random, and that the log-likelihood
ratios appearing with the Poisson family are not bounded, so that the machinery
from \cite{ghosal_fundamentals_2017} via Hellinger testing does not apply
in a straightforward way. We develop here some techniques based on the
alternative approach of \cite{gine_rates_2011} to establish tests by sharp
concentration properties of appropriate estimators---these are of independent
interest and can be found in Section~\ref{tprf}.

\subsection{Measurement setup and information inequalities}
\label{measset}

We consider a measurable space $(\mathbb O, \mathcal B)$ and a collection
$\mathcal L$ of finite measures
$\Lambda : \mathcal B \to [0,\infty )$. Let $B_{1}, \dots , B_{K}$ be a
measurable partition $\mathbb O = \bigcup_{i=1}^{K} B_{i}$. We are given
observations drawn independently across the ``bins'' $B_{i}$ from a Poisson
distribution
%
\begin{equation}
\label{pcount} Y_{i} \sim \operatorname{Poisson}\bigl(n \Lambda
(B_{i})\bigr), \quad i=1, \dots , K. 
\end{equation}
We can think of $n \Lambda (\mathbb O)$ as the ``average number of observations,''
and we assume for simplicity that $n$ is known. This models the setting
where we measure $K$ batches of count data arising from $n$ realizations
of a Poisson point process over $\mathbb O$ with intensity measure
$\Lambda $. We consider a high-dimensional setup where
$K=K_{n} \to \infty $ with $n \to \infty $, so that we may recover finer
properties of $\Lambda $ (such as its density w.r.t.~some dominating measure);
see Section~\ref{denrate}.

To proceed, consider $K$ realizations $Y_{1}=y_{1},\ldots $, $Y_{K}=y_{K}$ of
(\ref{pcount}). The probability density of each observation $Y_{i}$ is
given by
\begin{equation*}
p_{i,\Lambda , n}(y_{i})=e^{-n\Lambda (B_{i})} \frac{(n\Lambda (B_{i}))^{y_{i}}}{y_{i}!},
\quad y_{i} \in \mathbb N \cup \{0\} \equiv \mathbb
N_{0},
\end{equation*}
and by independence the resulting likelihood function is
%
\begin{equation}
\label{lik} p_{\Lambda}(y_{1}, \dots ,
y_{K}) \equiv p^{K}_{\Lambda , n}
(y_{1},\ldots,y_{K})= \prod
_{i=1}^{K} p_{i,\Lambda , n}(y_{i}).
\end{equation}
Denote by $P_{\Lambda}=P_{\Lambda ,n}^{K}$ and
$\mathbb E_{\Lambda }= \mathbb E_{n,\Lambda}^{K}$ the probability measure
and expectation operator with respect to the density
$p^{K}_{\Lambda ,n}$, respectively.

For two finite measures $\Lambda $, $\Lambda _{0}$ on $\mathbb O$, we define
a (semi)metric in sequence space as
%
\begin{equation}
\label{seqdist} \llVert \Lambda - \Lambda _{0} \rrVert
_{\ell _{1}} = \sum_{i=1}^{K}
\bigl\llvert \Lambda (B_{i})-\Lambda _{0}(B_{i})
\bigr\rrvert 
\end{equation}
as well as the following ``information divergences''
%
\begin{equation}
\label{seqdiv} \mathcal D^{2}_{2,K} (\Lambda ,\Lambda
_{0}) = \sum_{i=1}^{K}
\frac{ \llvert \Lambda (B_{i})-\Lambda _{0}(B_{i}) \rrvert ^{2}}{\Lambda _{0}(B_{i})},\qquad \mathcal D_{\infty , K}(\Lambda ,\Lambda
_{0}) = \max
_{1\le i\le K} \frac{ \llvert \Lambda (B_{i})-\Lambda _{0}(B_{i}) \rrvert }{\Lambda _{0}(B_{i})}, 
\end{equation}
which are understood to equal $\infty $ if $\Lambda _{0}(B_{i})=0$ for
some $i$ for which $\Lambda (B_{i}) \neq 0$ while $0/0=0$ by convention.
We now give a ``local'' bound for the underlying information distance.

\begin{lemma}%
\label{klbd}
Let $K \in \mathbb N$ and let $\Lambda $, $\Lambda _{0}$ be finite measures
on $\mathbb O$ such that
$\mathcal D_{\infty , K}(\Lambda , \Lambda _{0})<1/2$. Then we have
\begin{equation*}
\mathbb{E}_{\Lambda _{0}} \biggl(\log \frac{p_{\Lambda _{0},n}^{K}}{p_{\Lambda ,n}^{K}}(Y_{1},
\ldots,Y_{K}) \biggr) \leq 2 n \mathcal D^{2}_{2,K}(
\Lambda , \Lambda _{0}).
\end{equation*}
\end{lemma}

We can then obtain a lower bound on certain integrated likelihood ratios
(``evidence'' lower bound) that will be essential below. For this, we consider
a collection $\mathcal L$ of finite measures $\Lambda $ over
$(\mathbb O, \mathcal B)$ and equip $\mathcal L$ with a $\sigma $-field
$\mathcal S$ ensuring Borel measurability of the maps
%
\begin{equation}
\label{measur} (\Lambda , y) \to p_{\Lambda}(y_{1}, \dots
, y_{K})\quad \text{from } \mathcal L \times \mathbb
N^{K}_{0} \to \mathbb{R}. 
\end{equation}

\begin{lemma}%
\label{lemma:prelim-stats}
Fix $\Lambda _{0} \in \mathcal L$ s.t. $\Lambda _{0}(B_{i})>0$ for all
$i$. Let $\nu $ be a probability measure on
\begin{equation*}
A_{\epsilon}:= \bigl\{ \Lambda \in \mathcal L: \mathcal
D_{2,K}^{2}( \Lambda , \Lambda _{0}) \leq
\epsilon ^{2}, \mathcal D_{\infty , K}( \Lambda , \Lambda
_{0}) \leq 1/2 \bigr\} \in \mathcal S,\quad \epsilon >0.
\end{equation*}
Then for every $c>0$ there exists a constant $L_{c}>0$ such that for all
$n$, $\epsilon $, $K$, we have
\begin{equation*}
P_{\Lambda _{0}} \biggl(
\int _{A_{\epsilon}} \frac{p_{\Lambda}}{p_{\Lambda _{0}}}(Y_{1},
\ldots,Y_{K})\,d\nu (\Lambda ) \leq e^{-(2+c)n\epsilon ^{2}} \biggr) \leq
2e^{-L_{c} n\epsilon ^{2}}.
\end{equation*}
\end{lemma}
The proofs of these lemmas are not difficult and can be found in Section \ref{sec:supp_stats}.

\subsection{The basic contraction theorem}
\label{sec2.2}

Consider a prior $\Pi $ defined on the measurable sets $\mathcal S$ (from
before (\ref{measur})) of the collection $\mathcal L$ of ``intensity''
measures $\Lambda $ over $\mathbb O$. The posterior distribution arising
from data (\ref{pcount}) with likelihood (\ref{lik}) then arises from Bayes'
formula and is given by
%
\begin{equation}
\label{post} \Pi (A|Y_{1}, \dots , Y_{K}) =
\frac{\int _{A} p^{K}_{\Lambda ,n} (Y_{1}, \dots , Y_{K})
\,d\Pi  (\Lambda )}{\int _{\mathcal L} p^{K}_{\Lambda , n} (Y_{1}, \dots , Y_{K}) \,d\Pi  (\Lambda )},
\quad A \in \mathcal S. 
\end{equation}
Given the preparations from Section~\ref{measset}, the proof (to be found
in Section \ref{sec:supp_stats}) of the next theorem follows the standard pattern
\cite{ghosal_fundamentals_2017}, assuming the existence of sufficiently
good ``tests,'' which we will construct in Theorem~\ref{thm:concentration-ineq} below, using the concentration of measure
approach of \cite{gine_rates_2011} adapted to high-dimensional Poisson
count data.

\begin{theorem}%
\label{contract}
Let $K=K_{n} \in \mathbb N$ and $\varepsilon _{n}>0$ be sequences such
that
\begin{equation*}
\frac{K_{n}}{n} \to 0,\qquad  \varepsilon _{n} \le \sqrt{
\frac{K_{n}}{n}}, \qquad \varepsilon _{n} \to 0, \qquad  \sqrt{n}
\varepsilon _{n} \to \infty .
\end{equation*}
Suppose a sequence of prior probability measures $\Pi =\Pi _{n}$ defined
on $\mathcal L$ satisfies, for some $C>0$, some
$\Lambda _{0} \in \mathcal L$, and all $n$,
%
\begin{equation}
\label{KLball} \Pi \bigl(\Lambda : \mathcal D_{2,K}^2(\Lambda ,\Lambda
_{0}) \leq \varepsilon _{n}^{2}, \mathcal
D_{\infty ,K}(\Lambda , \Lambda _{0})\leq 1/2 \bigr) \geq
e^{-C
n\varepsilon _{n}^{2}}. 
\end{equation}
For fixed $\Lambda _{\mathrm{max}}<\infty $, define
$\mathcal{R}=\{\Lambda \in \mathcal L : \Lambda (\mathbb O) \le
\Lambda _{\mathrm{max}}\}$ and assume further that there exist measurable sets
$\Theta _{n} \subset \mathcal{L}$ such that, for all $n$ large enough and
some $L>0$,
%
\begin{equation}
\label{excess} \Pi \bigl(\mathcal L \setminus (\Theta _{n} \cap
\mathcal R)\bigr) \leq Le^{-(C+4)n
\varepsilon _{n}^{2}}. 
\end{equation}
Then the posterior distribution $\Pi (\cdot | Y_{1},\ldots,Y_{K})$ from (\ref{post})
arising from data (\ref{pcount}) contracts about the ground truth intensity
measure $\Lambda _{0}$ in the $\ell _{1}$ metric; specifically, we have
for some $k>0$ that
\begin{equation*}
\Pi \bigl(\Lambda \in \Theta _{n} \cap \mathcal R: \llVert \Lambda
-\Lambda _{0} \rrVert _{
\ell _{1}} < M\sqrt{K_{n}/n}
| Y_{1},\ldots,Y_{K} \bigr) = 1 - O_{P_{
\Lambda _{0}}}
\bigl(e^{-k n \varepsilon _{n}^{2}}\bigr)
\end{equation*}
whenever $M$ is large enough depending on $C$, $\Lambda _{\mathrm{max}}$, $k$.
\end{theorem}

In typical applications of Bayesian nonparametrics,
$K \simeq n \varepsilon _{n}^{2}$ will be a natural choice that matches
the number of bins with the effective smoothness / dimension of the prior
for the unknown intensity $\Lambda $. However, below, we shall employ priors
that arise as pushforwards of PDE solution maps and, therefore, naturally
oversmooth. In this case, the contraction rate obtained from ``Poisson
bin count'' data will be dominated by the number of bins used, hence the
previous theorem is stated in this more flexible form.

The uniform bound on $\Lambda (\mathbb O)$ can be replaced by a condition
that the posterior concentrates on uniformly bounded intensities with frequentist
probability approaching one.

\subsection{Contraction rates for intensity densities}
\label{denrate}

The $\ell _{1}$-metric on the bin counts is a natural metric for regression
problems arising from (\ref{pcount}), but does not identify the full intensity
measure for fixed $K$. However, as $K \to \infty $, if the measures
$\Lambda $ have sufficiently regular densities $\lambda $ with respect
to a joint dominating measure, we can recover these consistently as
$K \to \infty $. We now discuss this for the case relevant below where
$\mathbb O$ is a bounded subset of $\mathbb{R}^{d}$ of positive Lebesgue
measure $dx$ and the intensities have Lipschitz densities.

Denote the $L^{2}(dx)$-orthonormal functions arising from the measurable
partition of $\mathbb O$ by $e_{i} = 1_{B_{i}}/\sqrt{|B_{i}|}$ where
$|B|$ denotes the volume of $B$. The $L^{2}(dx)$ projection of
$g \in L^{2}$ onto $E_{K}=\operatorname{span}\{e_{i}: i \le K\}$ is then
\begin{equation*}
P_{E_{K}} g = \sum_{i \le K} \langle
e_{i}, g\rangle _{L^{2}} e_{i} = \sum
_{i=1}^{K} \frac{1}{ \llvert B_{i} \rrvert }
\int _{B_{i}}g(z)\,dz 1_{B_{i}}
\end{equation*}
and this makes sense for any finite signed measure $\lambda $ if we replace
$\int _{B_{i}} \lambda $ by $\Lambda (B_{i})$. For such $\lambda $, we
have
\begin{align*}
\llVert P_{E_{K}}\lambda \rrVert _{L^{1}} =
\int _{\mathbb O} \Biggl\llvert \sum_{i=1}^{K}
\frac{1}{ \llvert B_{i} \rrvert } \Lambda (B_{i}) 1_{B_{i}}(x) \Biggr
\rrvert \,dx \le \sum_{i=1}^{K} \bigl
\llvert \Lambda (B_{i}) \bigr\rrvert = \llVert \Lambda \rrVert
_{\ell _{1}}.
\end{align*}

Now given $x \in \mathbb O$, denote by $B_{i}(x)$ the unique element of
the partition containing $x$. Then if $\Lambda $ has a Lebesgue density
$\lambda $ that is $B$-Lipschitz, we obtain
\begin{equation*}
\bigl\llvert P_{E_{K}}\lambda (x) - \lambda (x) \bigr\rrvert = \biggl
\llvert \frac{1}{ \llvert B_{i}(x) \rrvert }
\int _{B_{i}(x)} \bigl(\lambda (z)-\lambda (x)\bigr)\,dz \biggr
\rrvert \le B \sup_{z \in B_{i}(x)} \llvert z-x \rrvert
_{
\mathbb{R}^{d}} \le c(B,d) h
\end{equation*}
if each bin $B_{i}$ has Euclidean diameter at most $h$. This forces
$K \simeq h^{-d}$ in Euclidean domains by standard volumetric arguments,
for example, page 373 in \cite{gine_mathematical_2016}. The previous bound
can be $dx$-integrated to give
%
\begin{equation}
\label{bias} \llVert P_{E_{K}}\lambda - \lambda \rrVert
_{L^{1}} \le c_{B,\mathbb O} h, 
\end{equation}
which controls the $L^{1}$-approximation error from the ``bin count'' histogram
estimator. Using slightly refined approximation theory (as in Chapter~4.3
in \cite{gine_mathematical_2016}), the assumption of $B$-Lipschitzness
can be weakened to $\lambda $ having a bounded
$B^{1}_{1\infty}(\mathbb O)$-norm, if the latter Besov space is appropriately
defined on $\mathbb O$.

\begin{theorem}
\label{thm2.4}
In the setting of Theorem~\ref{contract}, suppose that $\mathbb O$ is a
bounded subset of $\mathbb{R}^{d}$ and that the bins
$(B_{i})_{i=1}^{K}$ have diameter bounded by a constant multiple of
$K^{-1/d}$. Assume further that
%
\begin{equation}
\label{Lipbd} \Theta _{n} \subseteq \{\text{the density }\lambda
\text{ of } \Lambda \text{ exists and is } B\text{-Lipschitz} \} 
\end{equation}
for all $n$ and some $B>0$. Then if the density $\lambda _{0}$ of the true
intensity measure $\Lambda _{0}$ exists and is $B$-Lipschitz on
$\mathbb O$, we have as $n \to \infty $ that
%
\begin{equation}
\label{densrat} \Pi \bigl(\lambda : \llVert \lambda - \lambda _{0}
\rrVert _{L^{1}} > M \bigl(\sqrt{K/n} + K^{-1/d}
\bigr)|Y_{1}, \dots , Y_{K}\bigr) = O_{P_{\Lambda _{0}}}
\bigl(e^{-k n \varepsilon _{n}^{2}}\bigr) 
\end{equation}
 for some large enough constant $M$ depending on
$B$, $d$, $\mathbb O$, $C$, $\Lambda _{\mathrm{max}}$ and some $k>0$.
\end{theorem}
\begin{proof}
The result follows from Theorem~\ref{contract}, (\ref{bias}) and
\begin{align*}
\llVert \lambda - \lambda _{0} \rrVert _{L^{1}}
& \le \llVert \Lambda -\Lambda _{0} \rrVert _{\ell _{1}}
+ \bigl\llVert P_{E_{K}}(\lambda - \lambda _{0}) - (
\lambda - \lambda _{0}) \bigr\rrVert _{L^{1}}.
\end{align*}
\end{proof}

For high-dimensional Poisson regression problems, the $L^{1}$-structure
is useful to deploy appropriate concentration of measure techniques in
the proof of Theorem~\ref{thm:concentration-ineq}. When we know that the
$W^{\beta ,1}$ norms of $\lambda -\lambda _{0}$ are uniformly bounded,
then we can interpolate (\ref{interpol2}) and obtain a contraction rate
for $\|\lambda -\lambda _{0}\|_{L^{2}}$ as well.

The approach based on histograms presented here is naturally confined to
obtain bounds of the order $O(K^{-1/d})$, which is sufficient to prove
consistency at ``algebraic'' rates $n^{-\beta}$ for some $\beta >0$, but
which prohibits to attain ``fast'' convergence rates for intensity measures
$\Lambda $ whose densities are smoother than Lipschitz. But note that for
Lipschitz densities $\lambda $, the rate of convergence obtained in (\ref{densrat})
as well as in Theorem~\ref{thm:concentration-ineq} is minimax optimal (as
can be proved similar to Chapter~6.3 in \cite{gine_mathematical_2016},
using Lemma~\ref{klbd} above and also Exercise~6.3.2 there). Note also
that the assumption of disjoint support of the bins is crucial to retain
independence of the $K$ samples in (\ref{pcount}), and obtaining faster
rates will require rather different techniques based on general functionals
of Poisson point processes that we do not wish to develop here.

\section{Derivation of the observation model and proof of Theorem~\ref{thm:poisson_process2}}
\label{biostory}

The purpose of this section is to derive the Poisson process measurement
model from the axioms sketched above, and to prove Theorem~\ref{thm:poisson_process2} under natural regularity Conditions
\ref{qcond} and \ref{Dphicond} (to be introduced in the next subsection).
We recall the killed process $\tilde X_{t}$ from the introduction with
``binding'' time $S$ from (\ref{eq:binding-time}). Some well-known properties
of this process are explained, for example, in Section III.18 in
\cite{williams_markov_2000}: On the enlarged state space
$\Omega \cup \{\dagger \}$ where all measurable functions $f$ are extended
simply by setting $f(\dagger )=0$, one shows that $\tilde X_{t}$ is a Markov
process with infinitesimal generator equal to the Schr\"odinger operator
$\mathcal L_{D,q}$ from (\ref{eq:schrodinger-operator}). The law of
$\tilde X_{t}$ on $\Omega $ models the ``subprobability density'' of the
surviving particles at time $t$, given by the solution
$v_{D,q}(t,\cdot )$ of the PDE (\ref{schrodtime}). The first fact we prove
is that when $q$ is strictly positive on a set of positive Lebesgue measure,
the survival time of a particle is finite almost surely.

\begin{proposition}%
\label{prop:binding-time-finite}
Suppose $q$, $D$, $\phi $ satisfy Conditions \ref{qcond} and
\ref{Dphicond}. Let $S$ be the binding time from (\ref{eq:binding-time}).
Then $P(S<\infty )=1$.
\end{proposition}
\begin{proof}
For any $m \in \mathbb N$, we use (\ref{earlier}) and the Cauchy--Schwarz
inequality to bound
\begin{align*}
P(S>m) &\le P\bigl(\tilde X_{m} \notin \{\dagger \}\bigr) = P(\tilde
X_{m} \in \Omega )=
\int _{\Omega }v_{D,q}(m,x) \,dx
\\
&\le \llvert \Omega \rrvert ^{1/2} \bigl\llVert v_{D,q}(m,
\cdot ) \bigr\rrVert _{L^{2}} \lesssim e^{-cm}
\end{align*}
from which we deduce that $\sum_{m >1}P(S>m)<\infty $, and hence by the
Borel--Cantelli lemma we have
\begin{equation*}
P\bigl(\{S>m\} ~\text{infinitely often}\bigr)=0,
\end{equation*}
so the result follows.
\end{proof}

In the killed process, $\tilde X_{S}=\dagger $ is ``not visible'' any longer
at time $S$, whereas we observe the actual binding location $X_{S}$ of
the molecule. This is a well-defined random variable since $S$ is a stopping
time for the process $(X_{t})$ (e.g., Sections~3.3 and 22 in
\cite{bass_stochastic_2011}). Since all particles will bind at some random
point in time, intuitively, the distribution of their binding locations
should equal the time-average $\int _{0}^{\infty }v(t,\cdot )\,dt $ of the
density of their positions, multiplied by the intensity $q$ at the relevant
position.

\subsection{Probability density of the binding location $\times $ time}
\label{sec3.1}

In this section, $P=P^{(X,Y)}$ refers to the cylindrical probability measure
describing the product law of $(X_{t})$ on the path space
$C([0,\infty ))$ (see, e.g., Sections~24 or 39 in
\cite{bass_stochastic_2011}) and of the exponential random variable
$Y$ on $\mathbb{R}^{+}$. Define a random variable $Z$ taking values in
$\hat \Omega := \Omega \times [0,\infty ]$ by
\begin{equation*}
Z= (X_{S}, S)
\end{equation*}
such that $Z$ jointly records the location and time of the binding event
and is a measurable map for the Borel $\sigma $-algebra
$\mathcal{B}_{\hat \Omega} := \mathcal{B}_{\Omega} \otimes
\mathcal{B}_{[0,\infty ]}$, where
$\mathcal B_{[0,\infty ]} = \{B \cup \{\infty \}: B \in \mathcal B_{
\mathbb{R}^{+}}\}$. The law $\Lambda _{Z}=\operatorname{Law}(Z)$ can be characterized
as follows on a generating family of rectangles of measurable sets.

\begin{proposition}
\label{prop:density-rectangles}
Suppose $q$, $D$, $\phi $ satisfy Conditions \ref{qcond} and
\ref{Dphicond}. The probability measure $\Lambda _{Z}$ has a density function
such that for every rectangle
$A = \prod_{i=1}^{d} [a_{i}, b_{i}] \subset \Omega $,
$a_{i} <b_{i} \in \mathbb{R}$, and every interval
$(t_{1},t_{2}], 0<t_{1}< t_{2}<\infty $,
%
\begin{equation}
P\bigl(Z\in A \times (t_{1},t_{2}]\bigr) =
\int _{t_{1}}^{t_{2}}
\int _{A} \lambda (t, x) \,dx \,dt, \label{eq:proba-macro}
\end{equation}
where
%
\begin{equation}
\label{Lamde} \lambda (t,x)=q(x)v_{D,q}(t,x),\quad t>0, x \in
\Omega , 
\end{equation}
for $v_{D,q}$ the solution of the parabolic Schr\"{o}dinger equation (\ref{schrodtime})
with initial condition $\phi $.
\end{proposition}

\begin{proof}
Note first that from the exponential decay estimate (\ref{earlier}) and
since $\|q\|_{\infty}<\infty $, we know that the density $\lambda $ is
Lebesgue integrable on $[0, \infty ) \times \Omega $; hence, the right-hand
side~of (\ref{eq:proba-macro}) is well-defined. We start with the following
key lemma.

\begin{lemma}
\label{lemma:infinitesimal}
Recall the setting of Proposition~\ref{prop:density-rectangles}. For any
$A =\prod_{i=1}^{d} [a_{i}, b_{i}] \subset \Omega $ and any
$t \geq 0$,
%
\begin{equation}
\label{eq:asymptotic} P\bigl(Z \in A \times (t, t+h] \bigr) = h
\int _{A} \lambda (t, x)\,dx +o(h) \quad \text{as }h
\to 0. 
\end{equation}
\end{lemma}
\begin{proof}
By the memory-less property of the exponential random variable $Y$, if
$\tilde X_{t}$ binds at time $S \in (t, t + h]$, it has not bound before
time $t$ and we can restart the clock and the particle will bind at the
first subsequent time $t +\sigma $ for which
$\int _{t}^{t+\sigma}q(X_{s})\,ds > Y$. We also know that if
$\tilde X_{s}$ binds in $A$ at time $t+\sigma $, the variables
$\tilde X_{s}$ and $X_{s}$ coincide for all $s <t+\sigma $ (since the process
$\tilde X_{s}$ cannot return from the cemetery $\{\dagger \}$), and by
continuity of sample paths of the diffusion,
\begin{equation*}
X_{t+\sigma} = \lim_{r \to t+\sigma} X_{r} = \lim
_{r \to t+\sigma} \tilde X_{r}
\end{equation*}
must hold almost surely in this case.

Thus, let us write the probability
%
\begin{align}
\begin{aligned}
&P\bigl(Z \in A \times (t,t+h]\bigr)
\\
&\quad = P \biggl(
\int _{t}^{t+\sigma} q(X_{s})\,ds > Y
\text{ and } \lim_{r \to t +\sigma}\tilde X_{r} \in A
\text{ for some } 0<\sigma \leq h \biggr)
\\
&\quad = P \biggl({\bigcup_{0<\sigma \leq h}B_{\sigma }\cap
A_{\sigma}}\biggr) \equiv p_{h},
\end{aligned}
 \label{eq:proba-lemma1}
\end{align}
where for $y>0$, we have defined measurable subsets
\begin{equation*}
B_{\sigma }= \biggl\{ \omega ,y :
\int _{t}^{t+\sigma} q\bigl(\omega (s)\bigr)\,ds > y
\biggr\} \quad \text{and} \quad A_{\sigma}= \bigl\{\omega : \omega
(t+\sigma ) \in A \bigr\}\times \mathbb{R^{+}}.
\end{equation*}
Here, $\omega = \omega (s) \in \Omega $, $0<s \leq t+\sigma $ are points
in the path space of continuous trajectories of $\tilde X_{s}$ before binding,
extended by continuity to $s=t+\sigma $. We prove below that
%
\begin{equation}
\label{limsupinf} \limsup_{h \to 0} \frac{p_{h}}{h} \le
\int _{A} \lambda (t,x)\,dx \le \liminf
_{h \to 0} \frac{p_{h}}{h}, 
\end{equation}
from which the lemma follows. In what follows, let us write
$P^{\Omega}$ for the law of the process $(\tilde X_{s})$ with paths in
$\Omega $, $P^{Y}$ for the law of $Y$, and recall that the marginal law
$\operatorname{Law}(\tilde X_{t})$ has Lebesgue density $v_{D,q}(t,\cdot )$ on
$\Omega $ for every $t$ fixed.

\boxed{\text{Lower bound}} We clearly have
$B_{h} \cap A_{h} \subset \bigcup_{0<\sigma \leq h} B_{\sigma} \cap A_{
\sigma}$, and hence
\begin{equation*}
p_{h} \ge P(B_{h} \cap A_{h})=E^{\Omega}
\bigl[1_{A_{h}}(\omega )P^{Y}(Y \in B_{h}|
\omega )\bigr].
\end{equation*}
Therefore, by boundedness and continuity of $q$ combined with the fundamental
theorem of calculus and the dominated convergence theorem, we have
\begin{align*}
p_{h} &\ge
\int {1}_{A_{h}}(\omega ) P^{Y} \biggl(
\int _{t}^{t+h} q\bigl( \omega (s)\bigr)\,ds > Y
\biggr) \,dP ^{\Omega}(\omega )
\\
&=
\int {1}_{A_{h}}(\omega ) \bigl(1-e^{-\int _{t}^{t+h}q(\omega (s))\,ds } \bigr)\,dP
^{\Omega}(\omega )
\\
&= h
\int {1}_{A_{h}}(\omega ) \frac{1}{h}
\int _{t}^{t+h} q\bigl(\omega (s)\bigr)\,ds \,dP
^{\Omega}(\omega ) + o(h)
\\
&=h
\int {1}_{A_{h}}(\omega )q\bigl(\omega (t)\bigr)\,dP
^{\Omega}(\omega ) + o(h).
\end{align*}
This implies
\begin{equation*}
\liminf_{h} \frac{p_{h}}{h} \ge \liminf
_{h}
\int {1}_{A_{h}}( \omega )q\bigl(\omega (t)\bigr)\,dP
^{\Omega}(\omega ) \ge
\int _{\liminf _{h} A_{h}}q\bigl( \omega (t)\bigr)\,dP ^{\Omega}(
\omega ),
\end{equation*}
where we have used Fatou's lemma in the second inequality. Since
$\omega $ is continuous, we have
\begin{equation*}
\liminf_{h} A_{h} = \Biggl\{\omega (t+h)
\in \prod_{i=1}^{d}[a_{i},
b_{i}]~ \text{ eventually} \Biggr\} \supset \bigl\{\omega
_{i}(t) \in (a_{i}, b_{i}) ~\forall i=1,
\dots , d \bigr\} \equiv A^{o}
\end{equation*}
so the last integral is lower bounded by
\begin{equation*}
\int _{A^{o}} q\bigl(\omega (t)\bigr) \,dP ^{\Omega}(
\omega ) =
\int _{\prod _{i=1}^{d}(a_{i},
b_{i})} q(z)v_{D,q}(t,z)\,dz =
\int _{A} q(z) v_{D,q}(t,z)\,dz
\end{equation*}
since the boundary of the rectangle $A$ has Lebesgue measure zero.

\boxed{\text{Upper bound}} Since $q(\dagger ) =0 \le q$, we can regard
$B_{\sigma }\subset B_{h}$, so that
$\bigcup_{0<\sigma \leq h} B_{\sigma} \cap A_{\sigma }\subset B_{h}
\cap \bigcup_{0<\sigma \leq h} A_{h}$, and note that
%
\begin{equation}
P^{Y} \biggl(
\int _{t}^{t+h} q\bigl(\omega (s)\bigr)\,ds > Y
\biggr) = h q\bigl(\omega (t)\bigr) + o(h) \label{eq:upper-bound} 
\end{equation}
follows just as above. Hence, the probability $p_{h}$ in question is upper
bounded by
\begin{equation*}
h
\int 1_{\bigcup_{0\leq \sigma \leq h}A_{\sigma}}(\omega ) q\bigl(\omega (t)\bigr) \,dP
^{\Omega}(\omega ) + o(h).
\end{equation*}
Now letting $h=1/n$ for $n\in \mathbb{N}$, we see that
\begin{equation*}
\bigcup_{0\leq \sigma \leq 1/n} A_{\sigma} \downarrow
\bigcap_{n} \bigcup_{0\leq \sigma \leq 1/n}
A_{\sigma} = \limsup_{n} A_{n} =
\bigl\{ \omega (t+1/n) \in A \text{ for infinitely many $n$}\bigr\},
\end{equation*}
and since $\omega $ is continuous, we have
\begin{equation*}
\bigl\{\omega _{i}(t+1/n) \in [a_{i},
b_{i}] \ \forall i \text{ for infinitely many $n$}\bigr\} \subseteq \bigl\{\omega
_{i}(t)\in [a_{i}, b_{i}]~ \forall i\bigr
\} = \bar A.
\end{equation*}

Using that $q \ge 0$ and the monotone convergence theorem, we deduce from
what precedes that
\begin{equation*}
\limsup_{h} \frac{p_{h}}{h} \le
\int _{\limsup A_{n}}q\bigl(\omega (t)\bigr)\,dP ^{
\Omega}(
\omega ) \le
\int _{\bar A} q\bigl(\omega (t)\bigr)\,dP ^{\Omega}(
\omega ) =
\int _{A}q(z)v_{D,q}(t,z)\,dz ,
\end{equation*}
so that the proof of (\ref{limsupinf}), and hence of the lemma is complete.
\end{proof}

Now fix any $A = \prod_{i=1}^{d} [a_{i}, b_{i}] \subset \Omega $ and define
$S_{A}$ as the binding time in $A$, that is, $S_{A}=S$ if
$X_{S} \in A$ and $S_{A}=\infty $ otherwise. Write
$F_{S_{A}}(t)= P(S_{A} \le t)$, $t>0$, for its distribution function and
let $t$, $h$ be arbitrary. Then as in (\ref{eq:proba-lemma1}) and above (\ref{eq:upper-bound}),
using also $1-e^{-x} \le x$ for all $x$, we can estimate
\begin{align*}
F_{S_{A}}(t+h)-F_{S_{A}}(t)&= P\bigl(Z \in A \times (t, t+h]\bigr) \le
P(B_{h})
\\
&=
\int
\int {1}_{B_{h}}(\omega ,y)\,dP ^{Y}(y)\,dP
^{W}(\omega )
\\
&=
\int \bigl(1-e^{-\int _{t}^{t+h} q(\omega (s))\,ds }\bigr)\,dP ^{\Omega}(\omega ) \le
h \llVert q \rrVert _{\infty}.
\end{align*}

For any finite set of disjoint intervals
$(a_{1},b_{1}),\ldots, (a_{N},b_{N})$ lying in $[t_{1},t_{2}]$, we see that
\begin{equation*}
\sum_{j=1}^{N} (b_{j}-a_{j})<
\varepsilon / \llVert q \rrVert _{\infty }\quad \implies \quad \sum
_{j=1}^{N} \bigl\llvert F_{S_{A}}(b_{j})-F_{S_{A}}(a_{j})
\bigr\rrvert < \llVert q \rrVert _{\infty} \sum
_{j=1}^{N} (b_{j}-a_{j})
< \varepsilon ,
\end{equation*}
so the map $t \mapsto F_{S_{A}}(t)$ is absolutely continuous on
$[t_{1},t_{2}]$. By Theorem~3.35 in \cite{folland_real_1999},
$F_{S_{A}}$ therefore admits an integrable density $f_{S_{A}}$ with respect
to the Lebesgue measure such that
$P(Z \in A \times \tau )=\int _{\tau }f_{S_{A}}(s)\,ds $ for every subinterval
$\tau \subseteq [t_{1},t_{2}]$ and $A$ fixed. Moreover, by Lebesgue's differentiation
theorem and Lemma~\ref{lemma:infinitesimal},
\begin{align*}
f_{S_{A}}(t)-
\int _{A}\lambda (t,x)\,dx &= \lim_{h \to 0}
\frac{1}{h} \biggl(
\int _{t}^{t+h} f_{S_{A}}(s) \,ds - h
\int _{A} \lambda (t,x)\,dx \biggr)
\\
&= \lim_{h \to 0} \frac{1}{h} (P(Z \in
\biggl((t,t+h],A) - h
\int _{A}\lambda (t,x)\,dx \biggr) = 0
\end{align*}
for almost every $t \in (t_{1}, t_{2}]$, so that $f_{S_{A}}$ and
$\int _{A}\lambda (\cdot ,x)\,dx $ coincide almost everywhere. The proposition
is proved.
\end{proof}

We can extend Proposition~\ref{prop:density-rectangles} to the
$\sigma $-field
$\mathcal{B}_{\hat \Omega} = \mathcal{B}_{\Omega} \otimes \mathcal{B}_{[0,
\infty ]}$. Denote by $\Lambda $ the finite measure induced by the application
\begin{align*}
(B,\tau ) & \mapsto \Lambda (B \times \tau )=
\int _{B}
\int _{\tau } \lambda (t, x)\,dt\, dx ,\quad (B, \tau ) \in
\mathcal B_{\Omega }\otimes \mathcal B_{\mathbb{R}^{+}}.
\end{align*}
By Proposition~\ref{prop:density-rectangles}, the law $\Lambda _{Z}$ of
$Z$ and $\Lambda $ coincide on the rectangles $A \times \tau $ where
$\tau =(t_{1},t_{2}]$ and
$A = \prod_{i=1}^{d} [a_{i}, b_{i}] \subset \Omega $ or
$A=\emptyset $ and, therefore, by standard uniqueness arguments (e.g.,
Theorem~1.14 in \cite{folland_real_1999}) they coincide on the Borel
$\sigma $-field
$\mathcal B_{\Omega }\otimes \mathcal B_{\mathbb{R}^{+}}$ generated by
them. We can extend $\Lambda $ to $\mathcal B_{\hat \Omega}$ by setting
$\Lambda (B \times \tau \cup \{\infty \}) = \Lambda (B \times \tau )$,
in particular $\Lambda (\Omega \times \{\infty \})=0$, which is consistent
with
\begin{align*}
\Lambda \bigl(\Omega \times \mathbb{R}^{+}\bigr) &= \lim
_{T \to \infty}\Lambda \bigl( \Omega \times [0,T)\bigr) = \lim
_{T \to \infty}\Lambda _{Z}\bigl(\Omega \times [0,T
)\bigr)
\\
& = \Lambda _{Z}\bigl(\Omega \times \mathbb{R}^{+}
\bigr) = P(S< \infty )=1
\end{align*}
in view of Proposition~\ref{prop:binding-time-finite} and countable additivity
of measures. In particular, $\lambda $ is a probability density on
$\Omega \times [0,\infty )$, that is,
%
\begin{equation}
\label{probden}
\int _{\Omega} q(x)
\int _{\mathbb{R}_{+}} v_{D,q}(t,x)\,dt \,dx = \Lambda
\bigl(\Omega \times \mathbb{R}^{+}\bigr)=1. 
\end{equation}

If we record only the binding position (without reporting the time of binding),
then the probability to bind in a particular measurable subset $A$ of
$\Omega $, at some time, is given by
%
\begin{align}
\label{bindform}
\begin{aligned}
P( X_{S} \in A, S<\infty ) &= \Lambda
_{Z}\bigl(A \times \mathbb{R}^{+}\bigr) =
\int _{A}
\int _{0}^{\infty }q(x) v_{D,q}(t,x) \,dt
\,dx
\\
&=
\int _{A} q(x) u_{D,q}(x) \,dx \equiv \Lambda
_{D}(A),
\end{aligned}
\end{align}
where the second identity follows from (\ref{avsch}), and where
$u_{D,q}$ is the solution to the steady- state Schr\"odinger equation
%
\begin{align}
\label{ellip}
\begin{aligned}
\nabla \cdot (D \nabla u) - qu &= -\phi \quad \text{on }
\Omega, 
\\
\frac{\partial u}{\partial \nu} &=0 \quad \text{ on }\partial \Omega ,
\end{aligned}
\end{align}
with $\phi $ modeling the initial condition of the system.

\subsection{Construction of the Poisson point process}
\label{conppp}

The measure $\Lambda _{D}$ from (\ref{bindform}) will serve as the intensity
measure for the point process we construct now. We take independently diffusing
molecules $\tilde X_{s}^{i}$, $i=1, \dots $, and observe their terminal positions
$X_{S}^{i}$. We then record the count of binding locations through the
number $N(A)$ of molecules having bound in a measurable subset $A$ of
$\Omega $. Concretely, we take independent and identically distributed
random variables $T_{1}, T_{2}, \ldots$\,,
\begin{equation*}
T_{i}(A)= %
\begin{cases} 1 &
\text{if molecule $i$ has bound in $A$},
\\
0 & \text{otherwise} \end{cases} %
\end{equation*}
such that $P(T_{i}(A)=1)=\Lambda _{D}(A)$, and obtain a mixed binomial
process
\begin{equation*}
N(A) = \sum_{i=1}^{n_{\mathrm{mol}}}
T_{i}(A),\quad A \text{ measurable},
\end{equation*}
where $n_{\mathrm{mol}} \in \mathbb{N}$ is the total number of diffusing molecules.
If we assume that the binding positions in disjoint subsets of
$\Omega $ are statistically independent, then such a ``completely random''
point process is necessarily a Poisson process; see Theorem~6.12 in
\cite{last_lectures_2017}. The preceding mixed binomial process is a Poisson
process precisely when $n_{\mathrm{mol}}$ is drawn from a Poisson distribution (see
Propositions 3.5 and 3.8 in \cite{last_lectures_2017}). To model an effective
sample size of $n$, we therefore assume $n_{\mathrm{mol}} \sim \operatorname{Poisson}(n)$ for some
$n \in \mathbb{N}$. One then shows as in Section~3.2 of
\cite{last_lectures_2017} that $N$ is indeed a Poisson point process with
intensity $n\lambda _{D} = n qu_{D,q}$ on $\Omega $. The expected number
$\mathbb{E}[N(A)]$ of bound molecules in every Borel set $A$ of
$\Omega $ is precisely given by Campbell's formula (see Proposition~2.7
in \cite{last_lectures_2017}):
%
\begin{equation}
\mathbb{E}\bigl[N(A)\bigr] =
\int _{A} n q(x) u_{D,q}(x)\,dx .
\label{eq27}
\end{equation}

\section{Inference on the diffusion parameter from killed diffusion}
\label{sec4}

We formulate appropriate hypotheses on $D$, $q$, $\phi $ that allow to combine
the results from the previous sections. Since boundary reflection dominates
the diffusive dynamics of (\ref{eq:diffuso}) near $\partial \Omega $, it
is convenient to consider  {$q$ constant near $\partial \Omega$.} 
Along with that, we assume
that $D$ is known to be constant ($=1/2$, say) near
$\partial \Omega $ and bounded away from zero (by $1/4$, for notational
convenience). We also assume that the initial condition is strictly positive
and satisfies a basic smoothness and boundary compatibility condition.

\begin{condition}
\label{qcond}
The potential $q  {:\Omega \to [0,\infty)}$ lies in $H^{\eta}(\Omega ), \eta >1+d/2$, with
$\|q\|_{\infty }\le Q$ for some $Q>0$. Suppose further that
$\inf_{x \in \Omega _{00}}q(x) \ge q_{\mathrm{min}}>0$ on some open set
$\Omega _{00} \subset \Omega $ such that
$\operatorname{dist}(\Omega _{00}, \partial \Omega )>0$ and that  {$q$ equals the constant $\overline q \ge 0$} outside a compact subset $\Omega _{0} \subset \Omega $ that contains
$\Omega _{00} \subset \Omega _{0}$.
\end{condition}
%
\begin{condition}
\label{Dphicond}
The diffusivity $D$ lies in $H^{\alpha}(\Omega )$ for some
$\alpha >2+d/2$, equals $1/2$ outside of $\Omega _{00}$ and satisfies
$\inf_{x \in \Omega}D(x)>1/4$. Moreover, $\phi $ is a probability density
that satisfies $\inf_{x \in \Omega} \phi (x) \ge \phi _{\mathrm{min}}$ for some
$\phi _{\mathrm{min}}>0$ and $\phi $ either (i) lies in $H^{b}(\Omega )$ for some
$b$, $d/2 < b \le \min (\alpha +1, \eta +2)$, and is constant outside of
$\Omega _{00}$ or (ii) lies in $H^{2}_{\nu}(\Omega )$ from (\ref{h2nu})
below, in which case we set $b=2$, $d<4$, in what follows.
\end{condition}

The hypotheses ensure that $\phi $ belongs to the spectrally defined Sobolev
space $\tilde H_{D,q}^{b}$ from (\ref{sobspec}) below for
\textit{any} such $D$, $q$ and $b$ in view of Proposition~\ref{evergreen} (for
appropriate $\alpha $, $\eta )$. Lemma~\ref{regpde}, (\ref{fkacellip}) and
the Sobolev imbedding then imply that the unique solution $u_{D,q}$ to
(\ref{ellip}) is non-negative and lies in
$\tilde H^{b+2} \subset H^{b+2} \subset C^{2}$. Since also
$0 \le q \in H^{\eta }\subset C^{1}$, this ensures that the density
%
\begin{align}
\label{lD} \frac{d\Lambda _{D}}{dx} &= \lambda _{D} = q
u_{D,q} 
\end{align}
is Lipschitz and nonnegative on $\Omega $. We can therefore consider the
Poisson point process $N$ over $\Omega $ constructed in Section~\ref{conppp} with intensity measure $n \Lambda _{D}$ where
$\Lambda _{D}$ is the (probability) measure with Lebesgue density
$\lambda _{D}$. From what precedes $N$ will generate samples lying in
$\Omega _{0}$ almost surely. Considering that there are on average
$n$ particles diffusing, and in light of (\ref{densrat}), we then consider
%
\begin{equation}
\label{knd} K = K_{n,d} \simeq n^{d/(2+d)} 
\end{equation}
many bins $B_{i}$ that partition $\Omega _{00}$ in a way such that each
bin has Euclidean diameter bounded by a multiple of $K^{-1/d}$. We measure
the count of the killed particles in each bin:
%
\begin{equation}
\label{eq:observations} Y_{i} := N(B_{i}), \quad 1\leq i \leq
K, 
\end{equation}
which for these choices of $K, \Lambda =\Lambda _{D}$ and
$\mathbb O = \Omega _{00}$ fits into the general measurement setting (\ref{pcount}).
We denote the resulting law $P_{\Lambda _{D}}$ of the observation vector
by $P_{D}$ in this section.

\subsection{Construction of the prior and posterior}
\label{sec4.1}

Our goal is to make inference on $D$, and hence we construct a prior directly
for this parameter, which induces a pushforward prior for the intensity
measure $\Lambda _{D}$ via (\ref{lD}). We start with a Gaussian random
field $(w(x):x \in \Omega )$ that induces a probability measure
$\operatorname{Law}(w)$ on the (separable) Banach space
$H^{a}(\Omega ) \subset C^{1}(\Omega )$ with reproducing kernel Hilbert
space (RKHS) equal to $H^{\alpha}(\Omega )$, where
\begin{equation*}
\alpha >2+d, \alpha -d/2>a>2+d/2.
\end{equation*}
For the existence of such a process, see Theorem B.1.3 in
\cite{nickl_bayesian_2023}. We then define the rescaled random function
%
\begin{equation}
\label{prior0} W(x) = \frac{1}{n^{d/(4\alpha +2d)}}\zeta (x) w(x),\quad x \in \Omega ,
\end{equation}
where $\zeta $ is a smooth cutoff function that is zero outside of
$\Omega _{0}$ and identically $\zeta =1$ on $\Omega _{00}$. The prior probability
measure for $D$ is then the law $\Pi =\Pi _{n}=\operatorname{Law}(D)$ on $H^{a}$ of the
random function given by
%
\begin{equation}
\label{prior} D(x)=\frac{1}{4}+\frac{1}{4}e^{W(x)},
\quad x \in \Omega , 
\end{equation}
which equals $D=1/2$ near the boundary $\partial \Omega $, specifically
on $\Omega \setminus \Omega _{0}$. For given $q$, this induces a prior
probability measure $\Pi _{\Lambda}$ supported in $H^{a}$ for the intensity
measures $\Lambda _{D}$ from (\ref{lD}) simply via the image measure theorem
(and since the map $D \mapsto u_{D,q}$ is appropriately measurable in view
of Lemmata \ref{fwdregu} and \ref{regpde} below). This prior and then also
the resulting posterior measures concentrate precisely on the nonlinear
manifold of intensity functions that arise from observing a killed diffusion
process for some $D$ (and $q$), which will be crucial later to solve the
nonlinear inverse problem via Theorem~\ref{yetagain}. Note that the posterior
measure from (\ref{post}) based on data (\ref{eq:observations}) is then
given by
%
\begin{equation}
\label{postD} \Pi (A|Y_{1}, \dots , Y_{K}) =
\frac{\int _{A} p^{K}_{\Lambda _{D},n} (Y_{1}, \dots , Y_{K}) \,d\Pi
(D)}{\int _{H^{a}} p^{K}_{\Lambda _{D}, n} (Y_{1}, \dots , Y_{K}) \,d\Pi  (D)},
\quad A \text{ a Borel set in } H^{a}(\Omega ).
\end{equation}
Despite the lack of convexity of the negative log likelihood function
$D \mapsto -\log p_{\Lambda _{D}, n}^{K}$, following ideas in
\cite{stuart_inverse_2010}, approximate computation of this posterior measure
is possible by MCMC methods, with each MCMC iterate requiring the numerical
evaluation of $p_{\Lambda _{D}, n}^{K}$, and hence of the elliptic PDE
(\ref{ellip}). For instance, one can use the pCN, ULA or MALA algorithm;
see \cite{cotter_mcmc_2013} and also Section~1.2.4 in
\cite{nickl_bayesian_2023} as well as references therein. Computational
guarantees for such algorithms can be given following ideas in
\cite{nickl_polynomial-time_2022}; see also Chapter~5 in
\cite{nickl_bayesian_2023}. In concrete applications, one may wish to model
a further drift $\nabla U(X_{t})\,dt $ in (\ref{eq:diffuso}), which corresponds
to subtracting a further vector field $\nabla U \cdot \nabla u $ on the
left-hand side~in (\ref{ellip}), and hence changes little in the implementation
if $U$ is known or determined independently; see also the introduction
of \cite{nickl_consistent_2023} for further discussion of this aspect.

\subsection{Numerical illustration}
\label{subsec:numerics}

To illustrate the theory, we perform numerical experiments where
$\Omega $ is the unit disk ($d=2$) represented by a triangular mesh with
$9985$ nodes. We take $n=10^{7}$ as the expected number of molecules, and
$K=2911$ the number of bins. We select parameters $D$, $q$ such that
$q$ equals $100$ on a small circular ``focus,'' and $30$ outside (with
a smooth sigmoid transition in between), and such that in the link function
(\ref{prior}) we replace the additive constant $1/4$ by $15$. These values
were chosen for convenience to amplify the effects of both diffusion and
binding in the graphical representation of our results---a more comprehensive
numerical investigation of the algorithm proposed here will be undertaken
in future work.

\begin{figure}[b]
     \includegraphics[width=\linewidth]{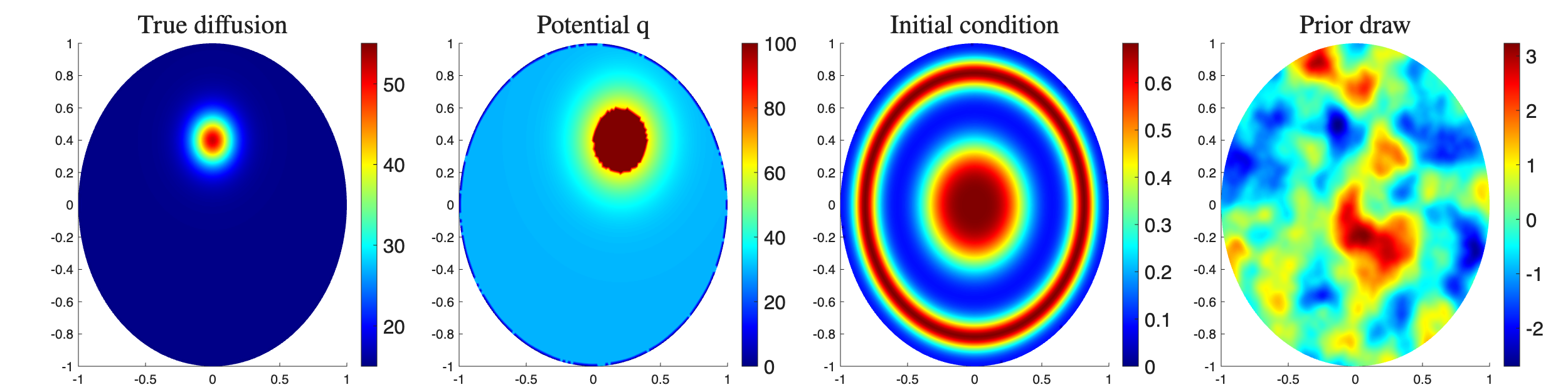} 
     \caption{Problem setting.}
\label{fig:setting}
\end{figure}

\subsubsection{Synthetic data}
\label{sec4.2.1}

We take the initial condition $\phi $ and ground truth $D_{0}$ such that
$\phi (x,y) \propto \exp (\cos (3\pi (x^{2} + y^{2})))$,
$W_{0}(x,y)=5e^{-10x^{2} -10(y-0.4)^{2}}$ for $(x,y)$ in the interior of
the mesh and zero elsewhere. The corresponding solution
$u_{D_{0},q}$, is computed on the mesh nodes via FEM (Finite Elements Method),
and can be extended to the whole domain via standard interpolation, as
needed to compute the integrals
$\Lambda _{D_{0}}(B_{i}) =\int _{B_{i}}qu_{D_{0},q}$, $i=1, \dots , K$. We
then draw the random observation count $Y_{i}=N(B_{i})$ for each bin as
$\operatorname{Poisson}(n\Lambda _{D_{0}}(B_{i}))$. For the sake of speed, the mass matrix
provided by the FEM is used to compute the integrals over the bins.

\subsubsection{MCMC}
\label{sec4.2.2}

Following a standard implementation from \cite{giordano2023bayesian}, we
run the preconditioned Crank--Nicolson (pCN) scheme for $15{,}000$ steps,
on the posterior density (\ref{postD}) arising from a Mat\'{e}rn Gaussian
process prior of smoothness parameter $\alpha =2.5$ and length scale
$\ell = 0.15$. At each step, the function $W$ is represented by the
$9{,}985$-dimensional vector of its values on the mesh nodes, and we again
use a standard interpolation algorithm to define it on the whole domain
when solving the PDE. We initialize the chain uniformly at 0. After running
the chain, we thin it by retaining every 200th sample after removal
of the burn-in (initial $5{,}000$ iterations), which we then use to compute
the posterior mean $\bar W$ giving rise to
$\bar D = 15+\frac{1}{4}e^{\bar W}$.
Figure \ref{fig:setting} displays the setting (true diffusivity $D_0$, potential $q$, and initial condition $\phi$ along with a prior draw) for our experiments. Figure \ref{fig:intensities} shows the resulting posterior mean $\bar D$, with the PDE solutions associated to $\bar D$ and $D_0$, respectively. Figure \ref{fig:counts} compares the observed bin intensities with the reconstruction obtained from the Bayesian estimator.

\begin{figure}
     \includegraphics[width=\linewidth]{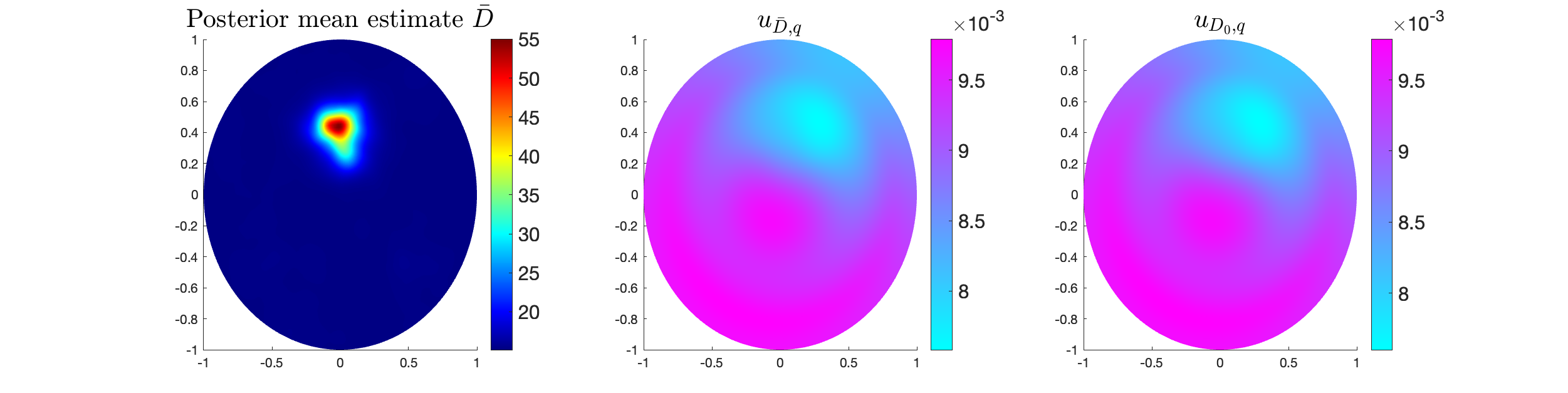}
     \caption{Posterior mean estimate $\bar D$ for the diffusivity function (left),
along with the corresponding PDE solution $u_{\bar D, q}$ (center), and
the true solution $u_{D_{0},q}$ (right) for comparison.}
\label{fig:intensities}
\end{figure}
%
\begin{figure}[h]
     \includegraphics[width=\linewidth]{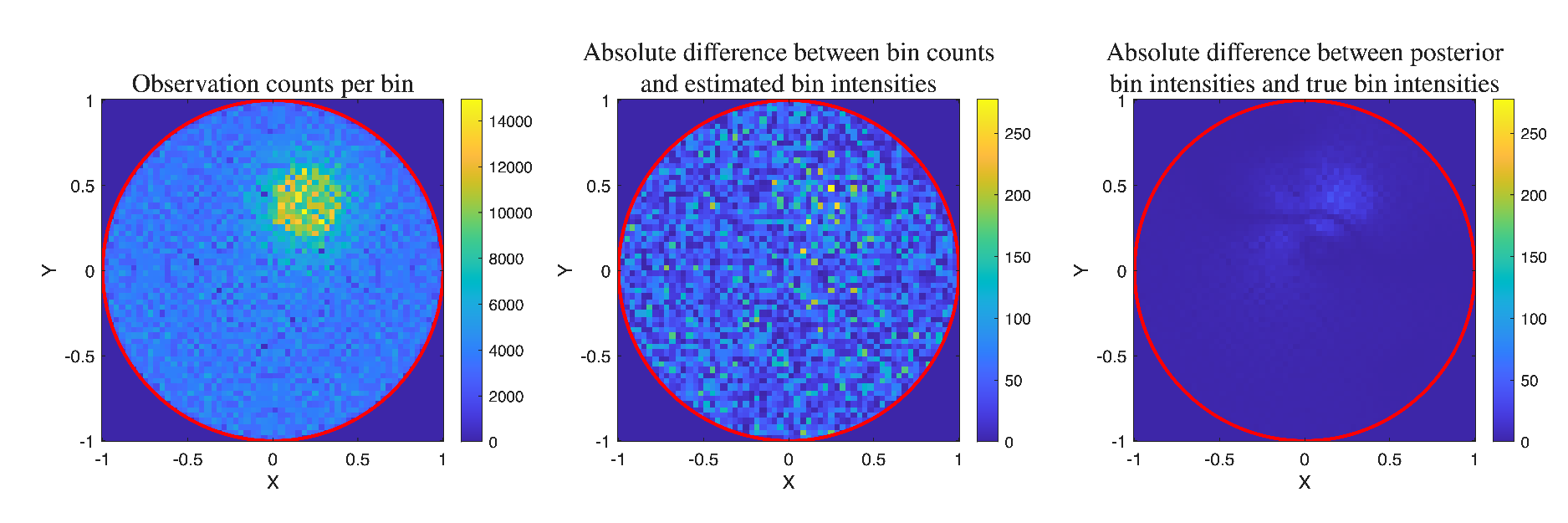} 
     \caption{The second and third displays compare respectively the ``fit''
of the histogram bin count $N(B_{i})$ and the estimated posterior intensities
$n\Lambda _{\bar D}(B_{i})$ with the ground truth
$n\Lambda _{D_{0}}(B_{i})$. The numerically improved recovery also of the
intensity $\Lambda _{D_{0}}$ by the Bayesian method when compared to the
naive bin count is clearly visible in these images.}
\label{fig:counts}
\end{figure}

\subsection{Posterior consistency theorems}
\label{sec4.3}

We now adapt techniques from Bayesian nonlinear inverse problems
\cite{monard_consistent_2021,nickl_bayesian_2023,nickl_consistent_2023}
to the present situation to derive some posterior consistency results.
While the probabilistic structure of the prior for $D$ is well understood
via Gaussian process theory, the prior for the intensity measures
$\Lambda _{D}$ is more complex as it involves the nonlinear PDE solution
map $D \mapsto u_{D,q}$. To verify the hypotheses of Theorem~\ref{contract}, we will require the following lemmas.

\begin{lemma}%
\label{fwdregu}
Suppose $D$, $q$, $\phi $ satisfy Conditions \ref{qcond} and
\ref{Dphicond}. Let the intensity measure $\Lambda _{D}$ be given as in
(\ref{lD}) with Lebesgue density $\lambda _{D}$. Then we have
\begin{equation*}
\inf_{x \in \Omega _{00}}\lambda _{D}(x) \ge
\frac{q_{\mathrm{min}}\phi _{\mathrm{min}}}{Q}>0.
\end{equation*}
Next, consider two diffusivities $D_{1}, D_{2} \ge 1/4$ s.t.~$D_{1}=D_{2}$
on $\Omega \setminus \Omega _{0}$ and $\|D_{i}\|_{H^{a}} \le B$ for some
$B>0$ and $a>2+d/2$. Then there exist finite positive constants
$c_{B,q_{\mathrm{min}}, Q}$, $c_{B,Q}$ such that for every $a'>1+d/2$,
\begin{align*}
\llVert \lambda _{D_{1}} - \lambda _{D_{2}} \rrVert
_{L^{2}} &\leq c_{B,q_{\mathrm{min}}, Q} \llVert D_{1}-D_{2}
\rrVert _{L^{2}},%
\\
\llVert \lambda _{D_{1}} - \lambda _{D_{2}} \rrVert
_{\infty} &\le c_{B,Q} \llVert D_{1}-D_{2}
\rrVert _{H^{a'}}.%
\end{align*}
\end{lemma}

\begin{proof}
We use PDE techniques and notation from Section~\ref{schrodspec}. The first
claim is Lemma~\ref{techlem}(a) with $\varphi =\phi $. Next, suppose
$u_{D_{i}}, i=1,2$, solve the elliptic equation (\ref{ellip}) for two distinct
diffusivities $D_{i}$ s.t.~$\|D_{i}\|_{H^{a}} \le B$. Set
$\bar u = u_{D_{1}} - u_{D_{2}}$ and notice that
$\|u_{D_{i}}\|_{C^{2}} \lesssim \|u_{D_{i}}\|_{\tilde H^{b+2}} \le c_{B}<
\infty $ in view of Lemma~\ref{regpde} for appropriate $b>d/2$, Proposition~\ref{evergreen}, the Sobolev imbedding and our hypotheses which ensure
$\phi \in \tilde H^{b}_{D_{i},q}$. Then we have
\begin{equation*}
\mathcal L_{D_{1}, q}(\bar u) = -\phi - \mathcal L_{D_{2}, q}
u_{D_{2}} + \nabla \cdot (D_{2}-D_{1}) \nabla
u_{D_{2}}=\nabla \cdot (D_{2}-D_{1}) \nabla
u_{D_{2}}
\end{equation*}
so that $\bar u$ equals $\mathcal L_{D_{1},q}^{-1}[\varphi ]$ with
$\varphi = \nabla \cdot (D_{2}-D_{1}) \nabla u_{D_{2}}$, which lies in
$\tilde H_{D_{1}, q}^{b}, b>d/2$, by what precedes, Proposition~\ref{evergreen}, and since $D_{1}-D_{2}$ vanishes near
$\partial \Omega $. We thus obtain from the second inequality in part (a)
of Lemma~\ref{techlem} that
%
\begin{equation}
\llVert u_{D_{1}} - u_{D_{2}} \rrVert _{\infty} \le
\bigl\llVert \nabla \cdot \bigl((D_{1}-D_{2}) \nabla
u_{D_{2}}\bigr) \bigr\rrVert _{\infty }\lesssim \llVert
D_{1}-D_{2} \rrVert _{C^{1}} \le \llVert
D_{1}-D_{2} \rrVert _{H^{a'}}, \label{eq34}
\end{equation}
which implies the desired $L^{\infty}$-estimate since
$\lambda _{D} = q u_{D,q}$ for all $D$. For the $L^{2}-L^{2}$-estimate,
we notice that $D_{1}-D_{2}$ has compact support in $\Omega $, and so we
see from the divergence theorem, again Lemma~\ref{techlem} combined with
the previous bound for $\|u_{D_{2}}\|_{C^{1}}$ and the Cauchy--Schwarz
inequality that
\begin{align*}
\llVert u_{D_{1}} - u_{D_{2}} \rrVert _{L^{2}} &
\lesssim \bigl\|\nabla \cdot (D_{1}-D_{2}) \nabla
u_{D_{2}}\bigr\|_{\tilde H^{-1}}
\\
&= \sup_{ \llVert \psi  \rrVert _{\tilde H^{1}}\le 1} \biggl\llvert
\int _{\Omega }(D_{1}-D_{2}) \nabla \psi
\cdot \nabla u_{D_{2}} \biggr\rrvert \lesssim \llVert
D_{1}-D_{2} \rrVert _{L^{2}},
\end{align*}
completing the proof.
\end{proof}


\begin{lemma}%
\label{gpsmall}
Consider the Gaussian process $W$ from (\ref{prior0}) with
$\alpha >2+d$ and let $W_{0} \in H^{\alpha}$ be such that $W_{0}=0$ outside
of $\Omega _{00}$. Let
$\bar \varepsilon _{n} = n^{-\alpha /(2\alpha +d)}$. Then for every fixed
$\epsilon >0$, $a<\alpha -d/2$, there exists a positive constant
$C=C(\epsilon , \alpha , a, \Omega , \zeta , \Omega _{00})>0$ such that
for all $n$,
%
\begin{equation}
\label{gpest} \Pr \bigl(W: \llVert W-W_{0} \rrVert
_{L^{2}} < \bar \varepsilon _{n}, \llVert
W-W_{0} \rrVert _{H^{a}}< \epsilon \bigr) \ge
e^{-Cn \bar \varepsilon _{n}^{2}}. 
\end{equation}
\end{lemma}
\begin{proof}
Note that the rescaling factor in (\ref{prior0}) equals
$1/\sqrt{n} \bar \varepsilon _{n}$. Since $\zeta ^{-1}$ is smooth on the
support of $W_{0}$, using also Exercise~2.6.5 in
\cite{gine_mathematical_2016} and (\ref{multi}), the RKHS-norm of the ``shift''
$W_{0}$ can be bounded as
\begin{equation*}
\sqrt n \bar\varepsilon _{n} \llVert W_{0}/\zeta
\rrVert _{H^{\alpha}} \lesssim \sqrt n \bar \varepsilon _{n}
\llVert W_{0} \rrVert _{H^{\alpha}}.
\end{equation*}
Therefore, using Corollary~2.6.18 in \cite{gine_mathematical_2016}, the
Gaussian correlation inequality (e.g., Theorem~B.1.2 in
\cite{nickl_bayesian_2023}), and again (\ref{multi}), we can lower bound
the probability in question by
\begin{equation*}
e^{-n\bar \varepsilon _{n}^{2}  \llVert W_{0} \rrVert _{H^{\alpha}}^{2}}\Pr\bigl(w: \llVert w \rrVert _{L^{2}} < \sqrt n \bar
\varepsilon ^{2}_{n} / \llVert \zeta \rrVert
_{\infty}\bigr) \Pr\bigl( \llVert w \rrVert _{H^{a}}< \epsilon
\sqrt n \bar \varepsilon _{n} / \llVert \zeta \rrVert
_{C^{a}}\bigr).
\end{equation*}
Since the law $\operatorname{Law}(w)$ is supported on the separable Banach space
$H^{a}$, the second probability is bounded from below by a positive constant
independent of $n$ in view of Anderson's lemma (Theorem~2.4.5 in
\cite{gine_mathematical_2016}) and for
$\sqrt{n}\bar \varepsilon _{n} \ge 1$ (in fact the probability converges
to one as $n \to \infty $). For the first probability, the required lower
bound follows by the same computation as the one on page 34 in
\cite{nickl_bayesian_2023}, so that the result follows.
\end{proof}

Collecting these facts, we can now show the following.
%
\begin{theorem}%
\label{fwdthm}
Suppose that $q$, $D_{0}$, $\phi $ satisfy Conditions \ref{qcond} and
\ref{Dphicond} with $\alpha >2+d$, respectively. The prior
$\Pi _{\Lambda}$ for the intensities $\Lambda _{D}$ from (\ref{lD}) as
constructed after (\ref{prior}) satisfies conditions (\ref{KLball}) and
(\ref{excess}) with $\Lambda _{D_{0}}$, $\Lambda _{\mathrm{max}}=1$,
$\varepsilon _{n} = \bar k n^{-\alpha /(2\alpha +d)}$ for some
$\bar k>0$, and with regularisation set
%
\begin{equation}
\label{regset} \Theta _{n} = \bigl\{D: \llVert D \rrVert
_{H^{a}} \le B\bigr\}, 
\end{equation}
for any $B$ large enough and every $a<\alpha -d/2$. As a consequence we
obtain, for all $M$ large enough, some $k>0$ and $K=K_{n}$ as in (\ref{knd})
that
%
\begin{equation}
\label{fwdcontract} \Pi \bigl(D \in \Theta _{n}: \llVert \lambda
_{D} - \lambda _{D_{0}} \rrVert _{L^{1}(
\Omega _{00})} \le M
n^{-1/(2+d)}|Y_{1}, \dots , Y_{K} \bigr) = 1
-O_{P_{D_{0}}}\bigl(e^{-k
n \varepsilon _{n}^{2}}\bigr) 
\end{equation}
\end{theorem}
\begin{proof}
To the ground truth $D_{0}$ corresponds the potential
$W_{0} = \log (4 D_{0}-1)$, which lies in $H^{\alpha}$ (by the chain rule
and since $D_{0}>1/4$ does), and which equals $W_{0}=\log 1=0$ outside
of $\Omega _{00}$, hence satisfies the hypotheses of Lemma~\ref{gpsmall}. To verify (\ref{KLball}), we work on the event in (\ref{gpest}),
so that in particular $\|W\|_{H^{a}} +\|W_{0}\|_{H^{a}}$ is bounded by
a fixed constant, which also bounds the $\|\cdot \|_{\infty}$-norms by
the Sobolev imbedding with $a>d/2$. By Lemma~\ref{fwdregu} and (\ref{lD}),
the true intensity function $\lambda _{D_{0}}$ is uniformly lower bounded
on $\Omega _{00}$, and hence on any $B_{i}$, by
$\Lambda _{\mathrm{min}}=q_{\mathrm{min}}\phi _{\mathrm{min}}/Q>0$. We then have from Jensen's inequality
\begin{align*}
\sum_{i=1}^{K} \frac{ \llvert \Lambda _{D}(B_{i}) - \Lambda _{D_{0}}(B_{i}) \rrvert ^{2}}{\Lambda _{D_{0}}(B_{i})}
&\leq \Lambda _{\mathrm{min}}^{-1} \sum
_{i=1}^{K} \frac{1}{ \llvert B_{i} \rrvert } \biggl(
\int _{B_{i}}(\lambda _{D} - \lambda
_{D_{0}}) (x)\,dx \biggr)^{2}
\\
&= \Lambda _{\mathrm{min}}^{-1} \sum
_{i=1}^{K} \biggl(
\int _{B_{i}}( \lambda _{D} - \lambda
_{D_{0}}) \frac{dx}{ \llvert B_{i} \rrvert } \biggr)^{2} \llvert
B_{i} \rrvert
\\
&\leq \Lambda _{\mathrm{min}}^{-1} \sum
_{i=1}^{K}
\int _{B_{i}}( \lambda _{D} - \lambda
_{D_{0}})^{2} \,dx
\\
&= \Lambda _{\mathrm{min}}^{-1} \llVert \lambda
_{D} - \lambda _{D_{0}} \rrVert _{L^{2}(
\Omega _{00})}^{2}
\leq C_{0} \llVert D-D_{0} \rrVert ^{2}_{L^{2}}
\le C_{1} \llVert W-W_{0} \rrVert ^{2}_{L^{2}},
\end{align*}
using also Lemma~\ref{fwdregu}, and for a fixed positive constant
$C_{1}>0$. Similarly, from Lemma~\ref{fwdregu},
\begin{align*}
\frac{ \llvert \Lambda _{D}(B_{i}) - \Lambda _{D_{0}}(B_{i}) \rrvert }{\Lambda _{D_{0}}(B_{i})} &\leq \Lambda _{\mathrm{min}}^{-1}
\frac{1}{ \llvert B_{i} \rrvert } \biggl\llvert
\int _{B_{i}}( \lambda _{D} - \lambda
_{D_{0}}) (x)\,dx \biggr\rrvert \le \Lambda _{\mathrm{min}}^{-1}
\llVert \lambda _{D} - \lambda _{D_{0}} \rrVert
_{\infty}
\\
&\leq c \llVert D-D_{0} \rrVert _{H^{a'}} \le
c' \llVert W-W_{0} \rrVert _{H^{a'}}
\end{align*}
for some $c'>0$. The condition (\ref{KLball}) then follows from Lemma~\ref{gpsmall} with $\varepsilon _{n}$ a constant multiple of
$\bar \varepsilon _{n}$ as well as $a'=a$ and $\epsilon >0$ small enough.
The condition (\ref{excess}) is verified for the rescaled prior and our
$\Theta _{n}$ just as on page 33 in \cite{nickl_bayesian_2023}. We also
note the uniform upper bound on
$\Lambda (\Omega _{00})\le 1=\Lambda _{\mathrm{max}}$ in view (\ref{probden}), (\ref{bindform}).
Therefore, noting also that our choices of $K_{n}$, $\varepsilon _{n}$ are
as required, we can invoke Theorem~\ref{contract} to obtain the contraction
rate $O(\sqrt{K/n})$ for the
$\|\Lambda _{D} - \Lambda _{D_{0}}\|_{\ell _{1}}$ distance over the
$K$ bins $B_{i}$ partitioning $\Omega _{00} = \mathbb O$. Using again Lemma~\ref{regpde}, Conditions \ref{qcond}, \ref{Dphicond} and the Sobolev imbedding
imply that $\lambda _{D} = qu_{D}$ is uniformly bounded in
$\tilde H^{2+b} \subset C^{1}(\Omega )$ for
$D \in \Theta _{n} \cup \{D_{0}\}$. This implies that these densities admit
a uniform Lipschitz constant $B$, and so the final conclusion for the
$L^{1}$-norms follows from (\ref{densrat}).
\end{proof}

The last step is to take advantage of the stability estimate Theorem~\ref{yetagain} to show that the posterior inference for $D$ is also valid,
proving Theorem~\ref{showoff}. To focus on the main statistical ideas we
show this for ``small potentials'' $q$ and $d \le 3$; see Section~\ref{idstab} for discussion.

\begin{theorem}
\label{thm:inference-diffusivity}
In the setting of Theorem~\ref{fwdthm}, assume further that $\phi $ satisfies
the identifiability Condition~\ref{ident}, that $d \le 3$ and  {that $q$ satisfies Condition (\ref{sergio})} for $\delta $ small enough as in Theorem~\ref{yetagain}. Then if $K=K_{n}$ is as in (\ref{knd}), we have
%
\begin{equation}
\label{invcontract} \Pi \bigl(D: \llVert D - D_{0} \rrVert
_{L^{2}} \le M n^{-\beta}|Y_{1}, \dots ,
Y_{K} \bigr) = 1 -O_{P_{D_{0}}}\bigl(e^{-k n \varepsilon _{n}^{2}}\bigr)
\end{equation}
for some $k>0$ and
$\beta = \frac{1}{2+d} \frac{b+2-d/2}{b+2}\frac{b}{b+2}$, where $b$ is
as in Condition~\ref{Dphicond}. Moreover, if
$\bar{D}= (1+e^{\bar W})/4$ for posterior mean
$\bar W =E^{\Pi}[W|Y_{1}, \dots , Y_{K}]$, then we also have
\begin{equation*}
\llVert \bar{D} - D_{0} \rrVert _{L^{2}} =
O_{P_{D_{0}}} \bigl(n^{-\beta}\bigr).
\end{equation*}
\end{theorem}

\begin{proof}
For $D \in \Theta _{n}$ from (\ref{regset}) and
$D_{0} \in H^{\alpha}$, we know from Lemma~\ref{regpde} and our hypothesis
on $\phi $ that $u_{D,q}-u_{D_{0},q}$ is bounded in
$\tilde H^{b+2} \subset H^{b+2} \subset W^{b+2,1}$. We then combine (\ref{fwdcontract})
with Theorem~\ref{yetagain} for $D_{1}=D_{0}$, $D_{2}=D$ and the interpolation
inequalities (\ref{interpol}), (\ref{interpol2}) to obtain
\begin{align*}
\llVert W-W_{0} \rrVert _{L^{2}} &\lesssim \llVert
 D-D_{0} \llVert _{L^{2}} \lesssim
\llVert u_{D,q} - u_{D_{0},q} \rrVert _{H^{2}(\Omega _{00})}
\\
&\lesssim \llVert u_{D,q} - u_{D_{0},q} \llVert _{L^{1}(\Omega _{00})}^{c(b,d)} \lesssim \llVert \lambda
_{D}-\lambda _{D_{0}} \rrVert _{L^{1}(\Omega _{00})}^{c(b,d)}
\\
&\lesssim n^{-\frac{1}{2+d} \frac{b+2-d/2}{b+2}\frac{b}{b+2}}
\end{align*}
from which the first claim of the theorem follows. The convergence of the
posterior mean can then be proved as in Theorem~2.3.2 in
\cite{nickl_bayesian_2023}, and is left to the reader.
\end{proof}

\begin{remark}
\label{rem4.7}
We have not tried to optimize the exponent $\beta $ in the rate
$n^{-\beta}$ obtained for recovery of $D$ in the previous theorem. While
some improvement is conceivably possible by replacing Sobolev imbeddings
and energy estimates in the proof by Schauder-type regularity results,
this is beyond the scope of the present paper. We emphasize that the key
analytical result Theorem~\ref{yetagain} below is based on a ``Darcy''-type
stability estimate, which likely precludes the possibility to obtain minimax
optimal convergence rates in this problem.
\end{remark}

\section{Analytical results from PDE theory}
\label{sec5}

\subsection{Schr\"odinger operators and their spectrum}
\label{schrodspec}

We now recall some analytical results for Schr\"odinger-type equations
that will allow to derive some estimates for solutions of the PDE (\ref{ellip})
that are essential in our proofs. Define the Schr\"odinger operator
%
\begin{equation}
\label{eq:schrodinger-operator} \mathcal L_{D,q}v \equiv \nabla \cdot (D \nabla v)
- qv 
\end{equation}
acting on $v$ belonging to the usual Sobolev space
%
\begin{equation}
\label{h2nu} H^{2}_{\nu }= \biggl\{ u \in
H^{2}(\Omega ): \frac{\partial u}{\partial \nu} =0 \text{ on } \partial
\Omega \biggr\} 
\end{equation}
equipped with Neumann boundary conditions. In this subsection, we generally
assume Conditions \ref{qcond} and \ref{Dphicond} even though weaker assumptions
would be possible, too. Adapting standard arguments from spectral theory
for the Laplace equation with Neumann boundary conditions in Chapter~5.7
in \cite{taylor_partial_2010} to the presence of $D$, $q$ (with
$D=1/2$, 
 {$q$ constant}
near $\partial \Omega $), one shows that the Schr\"odinger
operator $\mathcal L_{D,q}$ has $L^{2}(\Omega )$-orthonormal eigenpairs
\begin{equation*}
(e_{j,D,q}, -\lambda _{j,D,q}) \in H^{2}_{\nu}(
\Omega ) \times  {(-\infty, 0]} ,\quad  j \ge 1,
\end{equation*}
with eigenvalues satisfying
%
\begin{equation}
\label{weyl} \lambda _{j,D,q} \le \lambda _{j+1, D,q}
\simeq j^{2/d}, 
\end{equation}
where the constants implicit in $\simeq $ depend on upper bounds for
$\|q\|_{\infty}$, $\|D\|_{\infty}$ and on
$D_{\min}=\inf_{x\in \Omega}D(x)>0$. We now prove the following key Poincar\'e-type
inequality.
%
\begin{lemma}%
\label{qgap}
Suppose  {$q:\Omega \to [0,\infty)$ is such that} $q \ge q_{\mathrm{min}}>0$ on some subset $\Omega _{00}$ of $\Omega $ of
positive Lebesgue measure $|\Omega _{00}|>0$. Then we have the spectral
gap inequality
\begin{equation*}
\lambda _{1,D,q} \ge c\bigl(\Omega , \llvert \Omega _{00}
\rrvert , q_{\mathrm{min}}\bigr) >0.
\end{equation*}
\end{lemma}
\begin{proof}
By the variational characterisation of eigenvalues (Section~4.5 in
\cite{davies_spectral_1995}) and $D \ge 1/4$,
\begin{align*}
\lambda _{1,D,q} &= - \sup_{u \in H^{2}_{\nu},  \llVert u \rrVert _{L^{2}(\Omega )}=1} \langle \mathcal
L_{D,q}u, u \rangle _{L^{2}(\Omega )} \notag
\\
&= \inf_{u \in H^{2}_{\nu},  \llVert u \rrVert _{L^{2}(\Omega )}=1} \bigl[\langle D \nabla u, \nabla u\rangle
_{L^{2}(\Omega )} + \langle q u, u \rangle _{L^{2}(
\Omega )} \bigr]
\\
& \ge C_{q_{\mathrm{min}}}\inf_{u \in H^{2}_{\nu},  \llVert u \rrVert _{L^{2}(\Omega )}=1} \bigl[\langle \nabla u,
\nabla u\rangle _{L^{2}(\Omega )} + \llVert u \rrVert ^{2}_{L^{2}(
\Omega _{00})}
\bigr]
\end{align*}
for some $C_{q_{\min}}>0$. Suppose the last expression equals zero. By
the usual Poincar\'e inequality (page~292 in \cite{evans_partial_2010}),
this implies first that $u=\mathrm{const}$ and then that
$|\Omega _{00}|/|\Omega |=0$, a contradiction since $\Omega _{00}$ has
positive Lebesgue measure. The result follows.
\end{proof}

We deduce from Parseval's identity that for any $\varphi \in H^{2}$, we
have
\begin{equation*}
\mathcal L_{D,q}\varphi = -\sum_{j=1}^{\infty}
\lambda _{j,D,q} e_{j,D,q} \langle e_{j, D, q},
\varphi \rangle _{L^{2}},
\end{equation*}
with convergence at least in $L^{2}(\Omega )$, and that for any
$\varphi \in L^{2}$ the functions
\begin{equation*}
\mathcal L_{D,q}^{-1} [-\varphi ] = \sum
_{j=1}^{\infty }\lambda _{j,D,q}^{-1}
e_{j,D,q} \langle e_{j,D,q}, \varphi \rangle
_{L^{2}}
\end{equation*}
solve the PDE (\ref{ellip}) with right-hand side~$-\phi =-\varphi $ in
the $L^{2}$-sense. To see that these solutions satisfy Neumann boundary
conditions, let us define the associated Sobolev-type spaces
%
\begin{equation}
\label{sobspec} \tilde H^{b}(\Omega ) = \tilde
H^{b}_{D,q}(\Omega ) = \biggl\{f: \sum
_{j} \lambda _{j,D,q}^{b} \langle
f, e_{j,D,q} \rangle _{L^{2}}^{2} \equiv \llVert
f \rrVert _{\tilde H^{b}}^{2} <\infty \biggr\},\quad  b \ge -1.
\end{equation}

By Parseval's identity $\tilde H^{0}=L^{2}$ while
$\tilde H^{-1} = (\tilde H^{1})^{*}$ is interpreted by duality in the usual
way. For $b >0$, these are subspaces of $L^{2}$ with the following properties.


\begin{proposition}%
\label{evergreen}
\textup{(a)} For any $D \in C^{1}$, $q \in L^{\infty}$, we have that
$\tilde H^{2}_{D,q} = H^{2}_{\nu}$ and the graph norm
$\|\mathcal L_{D,q}f\|_{L^{2}}=\|f\|_{\tilde H^{2}_{D,q}}$ is Lipschitz-equivalent
to the $H^{2}$-norm on $H^{2}_{\nu}$, with equivalence constants  {depending in a Lipschitz way on $D_{\min}$, on upper bounds on $\|q\|_{\infty}$, $\|D\|_{C^{1}}$,  and on an upper bound $s_{gap}$ for 
$\max(1,\lambda^{-1}_{1,D,q})>0$}.

\textup{(b)} Let $\|D\|_{H^{a}} + \|q\|_{H^{\eta}} \le B$ for some
$a>2+d/2$, $\eta >1+d/2$ and let $|b| \le \mathrm{min}(a+1, \eta +2)$. Then
$\tilde H^{b} \subset H^{b}$ and
$\|u\|_{H^{b}} \le c(B,b, {s_{gap}}) \|u\|_{\tilde H^{b}}$ for all
$u \in \tilde H^{b}$. Conversely, if $u \in H^{b}$ is constant outside
of a compact set $K \subset \Omega $, then $u \in \tilde H^{b}$ and
$\|u\|_{\tilde H^{b}} \le c(B,b,K) \|u\|_{H^{b}}$.  {If $\int_\Omega u=0$ and $q$ is constant on $\Omega$, then the constants $c(B,b)$ in the preceding inequalities can be taken to be independent of $s_{gap}$.}
\end{proposition}
\begin{proof}
First, let $f \in \tilde H_{D,q}^{2}$ and define its $J$th partial sum
as
\begin{equation*}
f_{J}= \sum_{j \le J}
e_{j,D,q}\langle f, e_{J,D,q}\rangle _{L^{2}}, \quad J
\in \mathbb{N},
\end{equation*}
so that $f_{J} \in H^{2}_{\nu}$ since $e_{j} \in H^{2}_{\nu}$ and clearly
$f_{J} \to f$ in $\tilde H_{D,q}^{2}$. Using standard isomorphism theorems
for the Neumann--Laplacian (e.g., Theorem II.5.4 in
\cite{lions_non-homogeneous_1972}) and $D \ge 1/4$, we see for any
$u \in H^{2}_{\nu}$ that
\begin{align*}
\llVert u \rrVert _{H^{2}}
& \lesssim 
\llVert \Delta u \rrVert
_{L^{2}}  + \llVert u \rrVert_{L^{2}}  
\lesssim 
\llVert \mathcal L_{D,q}u \rrVert
_{L^{2}} + \llVert D \rrVert _{C^{1}} \llVert u \rrVert
_{H^{1}} + (1+ \llVert q \rrVert _{\infty }) \llVert u \rrVert
_{L^{2}}
\\
&\lesssim \llVert u \rrVert _{\tilde H^{2}_{D,q}} + \langle \mathcal
L_{D,q} u, u \rangle _{L^{2}}^{1/2} + (1+\llVert   q
\rrVert _{\infty }) \llVert u \rrVert _{L^{2}} \\
&
 {
\lesssim (1+ \lambda_{1,D,q}^{-1/2}+\lambda_{1,D,q}^{-1}) \llVert u\rrVert _{
\tilde H^{2}_{D,q}} \lesssim s_{gap} \llVert u\rrVert _{
\tilde H^{2}_{D,q}} }
\end{align*}
 {where we used Parseval's identity and (\ref{sobspec}).}

Applying this to $u=f_{J}-f_{J'}$, we see that $f_{J}$ is a Cauchy sequence
in $H^{2}$, and hence $f \in H^{2}$, in fact by the trace Theorem I.9.4
in \cite{lions_non-homogeneous_1972} we must have
$f \in H^{2}_{\nu}$. Conversely, let $f \in H^{2}_{\nu}$. Then by the divergence
theorem and Parseval's identity,
\begin{align*}
\llVert f \rrVert ^{2}_{\tilde H^{2}_{D,q}} &= \sum
_{j} \lambda _{j,D,q}^{2} \langle f,
e_{j,D,q}\rangle _{L^{2}}^{2} = \sum
_{j} \langle \mathcal L_{D,q}f,
e_{j,D,q}\rangle _{L^{2}}^{2}
\\
&= \llVert \mathcal
L_{D,q}f \rrVert ^{2}_{L^{2}} \lesssim \bigl(
\llVert D \rrVert _{C^{1}}+ \llVert q \rrVert _{\infty}\bigr)^2
\llVert f \rrVert _{H^{2}}^2,
\end{align*}
so $f \in \tilde H^{2}_{D,q}$ and part (a) of the proposition follows.
Part (b) is proved by induction similar to Proposition~2 in
\cite{nickl_consistent_2023} and left to the reader.
\end{proof}

\textbf{Parabolic and steady-state equation}\vspace*{3pt}

We now study the underlying time-dependent equation and recall its relationship
to the steady-state equation (\ref{ellip}). Specifically, consider solutions
$v=v_{D,q}$ to the parabolic PDE:
%
\begin{align}
\label{schrodtime}
\begin{aligned}
\frac{\partial}{\partial t} v(t, \cdot ) &= \mathcal
L_{D,q}v \quad \text{on } (0,T] \times \Omega, 
\\
v(0,\cdot ) &= \varphi \quad \text{on } \Omega,
\\
\frac{\partial v}{\partial \nu} &=0 \quad \text{ on } (0,T] \times \partial \Omega .
\end{aligned}
\end{align}
For any $\varphi \in L^{2}$ and $D$, $q$ as in Conditions \ref{qcond},
\ref{Dphicond}, unique solutions exist by standard PDE arguments (as, e.g.,
in Section~7.1.2 in \cite{evans_partial_2010}) and can be represented spectrally
as the convergent series in $L^{2}(\Omega )$ given by
%
\begin{equation}
v_{D,q}(t,\cdot ) = \sum_{j=1}^{\infty }e^{-t\lambda _{j,D,q}}
e_{j,D,q} \langle e_{j,D,q}, \varphi \rangle
_{L^{2}}. \label{eq44}
\end{equation}
One shows, using Proposition~\ref{evergreen}, that for
$\varphi \in \tilde H^{b}$, the solutions $v_{D,q}(t,\cdot )$ lie in\break
$C([0,T], \tilde H^{b}) \subset C([0,T], H^{b})$, where the parabolic spaces
$C([0,T],X)$ are defined as the bounded continuous maps from $[0,T]$ into
a Banach space $X$. We also have from Lemma~\ref{qgap}, (\ref{weyl}) and
Parseval's identity that
%
\begin{equation}
\label{earlier} \bigl\llVert v_{D,q}(t,\cdot ) \bigr\rrVert
_{L^{2}} \le e^{-ct} \llVert \varphi \rrVert
_{L^{2}}\quad \forall t>0, 
\end{equation}
for some $c>0$, which shows that the energy of the system dissipates as
$t\to \infty $ (corresponding to the ``killing of all particles'' in Proposition~\ref{prop:binding-time-finite}). Integrating the penultimate expression
in time and using
$\int _{0}^{s} e^{-t \lambda} \,dt  = -[e^{-s\lambda}-1]/\lambda $, $
\lambda >0$, we have after letting $s \to \infty $, and using the dominated
convergence theorem,
%
\begin{equation}
\label{avsch}
\int _{0}^{\infty }v_{D,q}(t,\cdot ) \,dt
= \sum_{j} \lambda _{j,D,q}^{-1}
e_{j,D,q} \langle e_{j,D,q}, \varphi \rangle
_{L^{2}} = \mathcal L_{D,q}^{-1}[- \varphi ],
\quad \varphi \in L^{2}. 
\end{equation}
Thus, the time average of the solutions to (\ref{schrodtime}) equal the
solutions $u_{D,q}$ to the steady-state equation (\ref{ellip}) with source
$-\phi = -\varphi $. From Proposition~\ref{evergreen} and the Sobolev imbedding,
we have that
$\varphi \in \tilde H_{D,q}^{b} \subset H^{b} \subset L^{\infty}$ for some
$b>d/2$ implies that the solutions $u_{D,q}$ lie in
$\tilde H_{D,q}^{2+b} \subset H^{2+b} \subset C^{2}$. Theorem~5.1 (or 5.2)
in \cite{F85} then allows us to represent the solution via the Feynman--Kac
formula
%
\begin{equation}
\label{fkacellip} u_{D,q}(x) =
\int _{0}^{\infty }v(t,x) \,dt = E^{x}
\int _{0}^{\infty } \varphi (X_{t}) \exp
\biggl\{-
\int _{0}^{t} q(X_{s})\,ds \biggr\}
\,dt ,\quad x \in \Omega , 
\end{equation}
where $X_{s}$ is the Markov process from (\ref{eq:diffuso}) started at
$X_{0}=x$. We can now prove the following regularity lemmas for solutions
to (\ref{ellip}).

\begin{lemma}
\label{techlem}
Assume Conditions \ref{qcond}, \ref{Dphicond} and that
$u_{D,q} = \mathcal L_{D,q}^{-1}(-\varphi )$ is the solution to (\ref{ellip})
with source $-\varphi =-\phi $.
\begin{longlist}[(b)]
\item[(a)] Suppose further $\varphi \in \tilde H_{D,q}^{b}$ for some $b>d/2$ and
that $\|q\|_{\infty }\le Q<\infty $. Then
\begin{equation*}
\inf_{x \in \Omega} u_{D,q}(x) \ge \frac{ \inf_{x \in \Omega}\varphi (x)}{Q}\quad \text{as well as }
\llVert u_{D,q} \rrVert _{\infty} \lesssim \llVert \varphi \rrVert
_{\infty}.
\end{equation*}
\item[(b)] We also have for any $\varphi \in L^{2}$ that
$\|u_{D,q}\|_{L^{2}} \le C_{q, \Omega _{00}, \Omega} \|\varphi \|_{
\tilde H^{-1}} \lesssim \|\varphi \|_{L^{2}}$.
\end{longlist}
\end{lemma}
\begin{proof}
For part (a), note that (\ref{fkacellip}) and the hypotheses on
$\varphi $, $q$ imply
\begin{equation*}
u(x) \ge \inf_{x \in \Omega}\varphi (x)
\int _{0}^{\infty }e^{-Qt}\,dt =
\frac{ \inf_{x \in \Omega}\varphi (x)}{Q},
\end{equation*}
%

For the second inequality, similarly

\begin{align}\label{eq:binding-time-tail}
 \|u_{D,q}\|_\infty 
& \;\le\; 
 \|\varphi\|_\infty\,\sup_{x\in\Omega}{E}^{x}\!\int_{0}^{\infty}
\exp\!\Big\{\!-\!\int_{0}^{t}q(X_s)\,\mathrm{d}s\Big\}\,\mathrm{d}t \nonumber\\
 &\;=\;  
 \|\varphi\|_\infty\,\sup_{x\in\Omega}{E}^{x}\!\int_{0}^{\infty}
\Pr\big(Y > \int_{0}^{t}q(X_s)\,\mathrm{d}s \big)\,\mathrm{d}t \nonumber\\
&\;\le\; 
 \|\varphi\|_\infty\,\sup_{x\in\Omega}{E}^{x}\!\int_{0}^{\infty}
\Pr(S>t)\,\mathrm{d}t
\end{align}

\noindent with $Y\sim \mathrm{Exp(1)}$ and $S$ as in (\ref{eq:binding-time}) for $X_0=x$.  
Let $\tilde p_t(x,y)=\sum_{j\ge1}e^{-t\lambda_{j,D,q}}e_{j,D,q}(x)e_{j,D,q}(y)$ be the transition kernel of (\ref{eq44}). 
By the Markov property, for $t\ge1$ the function $\tilde p_t(x,\cdot)$ is the solution of (\ref{schrodtime}) at time $t-1$ started from initial condition $\tilde p_1(x,\cdot)$. As in the proof of Proposition \ref{prop:binding-time-finite} we obtain
$$ \Pr(S>t)\le|\Omega|^{1/2}e^{-c(t-1)}\|\tilde p_1(x,\cdot)\|_{L^2},\qquad t\ge1 .$$
The Sobolev embedding, Proposition \ref{evergreen}b, Definition (\ref{sobspec}) and (\ref{weyl}) give
$$\|e_{j,D,q}\|_\infty\lesssim\|e_{j,D,q}\|_{H^b}\lesssim \|e_{j,D,q}\|_{\tilde H^b}=\lambda_{j,D,q}^{b/2}\lesssim j^{b/d},$$
and therefore
$$\sup_{x\in\Omega}\|\tilde p_1(x,\cdot)\|_{L^2}^2\le\sum_{j\ge1}e^{-2\lambda_{j,D,q}}\|e_{j,D,q}\|^2_\infty\lesssim\sum_{j\ge1}j^{2b/d}e^{-2c'j^{2/d}}=:\kappa^2<\infty,$$
We conclude that, uniformly in $x\in \Omega$,
$$\int_{0}^\infty \Pr(S>t)dt 
 \leq 1 + \int_{1}^\infty \Pr(S>t)dt \leq 1 + |\Omega |^{1/2}\kappa/c$$
which combined with (\ref{eq:binding-time-tail}) completes the proof of part (a).\medskip

For part (b), we use (the proof of) Lemma~\ref{qgap} and
sequence space duality to the effect that for any
$\psi \in \tilde H^{1}$ and testing $g=-\psi $,
\begin{equation*}
\llVert \mathcal L_{D,q} \psi \rrVert _{\tilde H^{-1}} = \sup
_{ \llVert g \rrVert _{\tilde H^{1}}
\le 1} \bigl\llvert \langle \mathcal L_{D,q}\psi
, g \rangle _{L^{2}} \bigr\rrvert \ge - \langle \mathcal
L_{D,q}\psi , \psi \rangle _{L^{2}} \ge c \llVert \psi
\rrVert _{L^{2}}.
\end{equation*}
Inserting $\psi = u = \mathcal L_{D,q}^{-1}\varphi \in \tilde H^{1}$ and
the elementary embedding $\tilde H^{1} \subset L^{2}$ imply (b).
\end{proof}

\begin{lemma}
\label{regpde}
In the setting of Lemma~\ref{techlem}, suppose that
$\varphi \in \tilde H_{D,q}^{b}$ for some
$b \le \min (a+1, \eta +2) -2$. Then for every $q \in H^\eta$ there exists a constant $c(B,q)$ such that
%
\begin{equation}
\sup_{ \llVert \varphi  \rrVert _{\tilde H^{b}} +  \llVert D \rrVert _{H^{a}}
\le B} \bigl\llVert \mathcal L_{D,q}^{-1}(-
\varphi ) \bigr\rrVert _{\tilde H_{D,q}^{b+2}} \le c(B,q)<\infty . \label{eq48}
\end{equation}
 {If $q$ is constant and $\int _{\Omega} \varphi =0$ then the result is true for a finite constant $c(B,Q)$ with $Q$ any upper bound for $q$.}
\end{lemma}
\begin{proof}
Within the spectrally defined $\tilde H^{\alpha}$ scale, by definition
of the $\tilde H^{b}$ norm we immediately obtain
\begin{equation*}
\bigl\llVert \mathcal L_{D,q}^{-1}
\varphi \bigr\rrVert _{\tilde H^{b+2}} =  \llVert
\varphi \rrVert _{\tilde H^{b}} \leq B < \infty .
\end{equation*}
The proof for  {$q=const$} is similar under the additional constraint on
$\varphi $; cf.~Section~3.1 in \cite{nickl_consistent_2023}  {for $q=0$, and extends to $q=const$ arguing as at the beginning of Step 2 of the proof of Theorem \ref{yetagain} since $\varphi$ is orthogonal to the eigenspace spanned by the constant eigenfunction.}
\end{proof}

\subsection{Identifiability and stability}
\label{idstab}

We will now show that the diffusivity $D$ can be identified from
$u_{D,q}$ on a fixed subset $\Omega _{00} \subset \Omega $ if $q$ is small
enough in 
 {an appropriate sense}
and assuming $d \le 3$. These hypotheses can
be weakened by the introduction of further technicalities (e.g., replacing
energy methods by Schauder estimates), but our focus here is not on obtaining
the most general stability estimate but on providing a concept proof that
inference on diffusivity from killed diffusion is possible under certain
hypotheses on the initial condition $\phi $. To ease notation, we take
$\Omega $ to have normalized Lebesgue measure equal to one, so that the
invariant measure of the process $(X_{t})$ from (\ref{eq:diffuso}) equals
$1$ identically on $\Omega $.

\begin{condition}%
\label{ident}
Let $d \le 3$. Assume that $|\Omega |=1$ and that for some strict subdomain
$\Omega _{00}$ of $\Omega $ and some $\epsilon >0$ the initial condition
$\phi $ of (\ref{eq:diffuso}) is a probability density that satisfies:
\begin{longlist}[2.]
\item[1.] either $\phi =1+\epsilon e_{1}$ where $e_{1} \in H^{2}_{\nu}$ is
the first nonconstant eigenfunction of the generator
$\mathcal L_{D_{0},0}$ of the diffusion process (\ref{eq:diffuso}), and
$e_{1}$ satisfies
$\inf_{x \in \Omega _{00}}|\nabla e_{1}(x)| \ge c>0$,
\item[2.] or $\phi \geq 1+ \epsilon $ on $\Omega _{00}$.
\end{longlist}
\end{condition}

Before we discuss this condition, let us state our main stability (inverse
continuity) estimate for the ``forward'' map $D \mapsto u_{D,q}$ of our
inverse problem.

\begin{theorem}
\label{yetagain}
Assume $D_{i}, i=1,2; q, \phi $ satisfy Conditions \ref{qcond},
\ref{Dphicond} with $\alpha \geq 2+d/2$,  {that $\inf_{x \in \Omega}q(x) \ge q_{\min}=\overline q>0$}, and moreover, that $\phi $ satisfies
Condition~\ref{ident} with $D_{0}=D_{1}$ in case (a). Denote by
$u_{D_{i},q}$ the corresponding solutions to (\ref{ellip}). Suppose further
that $\|D_{1}\|_{H^{\alpha}}+\|D_{2}\|_{H^{\alpha}} \le B$ for some
$B>0$. Then there exist positive constants $\delta $, $C$ depending on
$b$, $B$, $\phi $, $\epsilon $, $\Omega $, $\Omega _{00}$, $\alpha $ such that if  {
\begin{equation}\label{sergio}
q_{\min}+ \Big\|\frac{q-q_{\min}1_{\Omega}}{q_{\min}}\Big\|_{H^{2}(\Omega)}<\delta , 
\end{equation}} we have
%
\begin{equation}
\llVert D_{1}- D_{2} \rrVert _{L^{2}(\Omega )} \le C
\llVert u_{D_{1},q}-u_{D_{2},q} \rrVert _{H^{2}(
\Omega _{00})}.
\label{eq49}
\end{equation}
\end{theorem}

The proof relies on spectral theory for Schr\"odinger operators with Neumann
boundary conditions and a general stability lemma for a transport operator,
taken from \cite{nickl_consistent_2023} and inspired by earlier work in
\cite{nickl_convergence_2020,nickl_bayesian_2023}. The details can be found
in Section \ref{sec:stab-proof}.

\begin{remark}%
\label{rem5.7}
\normalfont  We next discuss some physical interpretations of our hypotheses.
Since $\phi $ is a probability density, Condition~\ref{ident}(b) requires
the initial state to be prepared such that more than average $(=1)$ particles
start diffusing in the subdomain $\Omega _{00}$ where we wish to identify
$D$. This can be thought of as ensuring that we will see sufficiently many
binding events in the region $\Omega _{00}$. The hypothesis (a) is somewhat
more subtle to interpret: It essentially requires the domain
$\Omega $ and operator
$\mathcal L_{D_{0},0}=\nabla \cdot (D_{0} \nabla )$ to satisfy the
\text{hotspots conjecture} (see Sections~2.2.2 and 3.7 in
\cite{nickl_consistent_2023} for concrete examples and relevant references).
We notice that $e_{1} \in H^{2}_{\nu }\subset L^{\infty}$ (cf.~Proposition~2 in \cite{nickl_consistent_2023}), so Condition~\ref{Dphicond} is satisfied
for $b=2$, $d<4$ and $\epsilon =\epsilon (\|e_{1}\|_{\infty})$ small enough.
Since the distribution $p_{t}(\psi )$ of $X_{t}$ from (\ref{eq:diffuso})
for any initial condition $\psi $ at time $t$ equals
$1 + \sum_{j \ge 1} e^{-t\lambda _{j}} e_{j} \langle e_{j}, \psi
\rangle _{L^{2}}$ (with $(e_{j}, -\lambda _{j})$ the eigenpairs of
$\mathcal L_{D_{0},0}$), the perturbative ($t \to \infty $) interpretation
of hypothesis (a) is that $\phi = 1 +\epsilon e_{1}$ is already close to
the invariant distribution of $(X_{t})$, and the first eigenfunction
$e_{1}=e_{1,D_{0}}$ is the last ``informative'' part of diffusion that
we can see before it reaches its constant equilibrium. (As
$\phi =e_{1}$ now depends on $D_{0}$, one needs to replace $\phi $ by an
empirical estimate in any MCMC implementation.)

That $q$ needs to be small enough can be related to the problem that too
large $q$ may kill diffusing particles too aggressively for them to visit
all areas of $\Omega _{00}$ before they ``bind.'' 
 {For instance, sufficiently small $q=q_{\min}$ constant on $\Omega$ is permitted, but also non-constant $q$ is allowed: for instance such that $q=q_{\min}1_{\Omega}+\delta q_{\min} \tilde q$, with $\tilde q \geq 0$ supported in $\Omega_0$ and such that $ \|\tilde q\|_{H^2}<1$, and $q_{\min}$ small enough.}
One could equivalently
calibrate the parameter of the exponential random variable $Y$ featuring
in the binding time. We do not know whether this hypothesis is necessary
or not. The numerical findings from Section~\ref{subsec:numerics} suggest
that the ratio $D/q$ is relevant for stable recovery, and moreover, that
recovery of $D$ is possible also after relaxing the hypothesis Condition~\ref{ident}.
\end{remark}

\section{Proof for Section~\ref{pcont}: A concentration inequality for high-dimensional Poisson regression}
\label{tprf}

For the construction of tests in the proof of Theorem~\ref{contract}, we
rely on a natural estimator $\hat \Lambda $ of $\Lambda $ with good concentration
properties---the normalized bin count, which gives the random vector in
$(\mathbb{R}^{+})^{K}$ with entries
%
\begin{equation}
\hat \Lambda _{i} = \frac{Y_{i}}{n},\quad i=1, \dots , K,
\label{eq:estimator}
\end{equation}
for data from (\ref{pcount}). Note that $\hat \Lambda $ does not satisfy
the ``PDE constraint'' (\ref{bindform}), hence cannot be used itself in
the context of our inverse problem since the stability estimate Theorem~\ref{yetagain} does not apply to it. Now consider the average
$\ell _{1}$-risk across the bins
%
\begin{equation}
\llVert \hat \Lambda - \Lambda \rrVert _{\ell _{1}} = \sum
_{i=1}^{K} \bigl\llvert \hat \Lambda
_{i} - \Lambda (B_{i}) \bigr\rrvert =
\frac{1}{n} \sum_{i=1}^{K}
\llvert Y_{i}- \mathbb{E}_{\Lambda}Y_{i} \rrvert
. \label{eq51}
\end{equation}
A basic first inequality is, using also the Cauchy--Schwarz inequality
twice,
\begin{equation*}
\mathbb{E}_{\Lambda} \llVert \hat \Lambda - \Lambda \rrVert
_{\ell _{1}} \le \frac{1}{n} \sum_{i \le K}
\sqrt{\operatorname{Var}(Y_{i})} = \frac{1}{n} \sum
_{i
\le K} \sqrt{n \Lambda (B_{i})} \leq
\frac{\sqrt K \sqrt{\Lambda (\mathbb O)}}{\sqrt {n}}.
\end{equation*}
This is upgraded to a probability inequality in the following theorem.
%
\begin{theorem}%
\label{thm:concentration-ineq}
If $K \le Cn$ for some $C>0$, then for any $c'>0$ we can find $t$ large
enough depending on $c'$, $C$ and on an upper bound
$\Lambda _{\mathrm{max}} \ge \Lambda (\mathbb O)$, such that we have for every
$K, n \in \mathbb N$,
%
\begin{equation}
P_{\Lambda} \bigl( \llVert \hat \Lambda -\Lambda \rrVert
_{\ell _{1}} \ge t \sqrt {K/n} \bigr) \le e^{-c'K}
\label{eq:concentration-ineq} . 
\end{equation}
\end{theorem}

\begin{proof}
We represent the $\ell _{1}$-norm of the vector
$\hat \Lambda - \Lambda $ on $\mathbb{R}^{K}$ by duality,
\begin{equation*}
\llVert \hat \Lambda - \Lambda \rrVert _{\ell _{1}} = \sup
_{v \in \mathbb{R}^{K}:
\max _{i} \llvert v_{i} \rrvert \le 1} \bigl\llvert v^{T} (\hat \Lambda -\Lambda
) \bigr\rrvert .
\end{equation*}
Let $B_{\infty}(1)$ be the unit ball of $\mathbb{R}^{K}$ for the
$\ell _{\infty}$ norm
$\|v\|_{\ell _{\infty}} = \max_{i \le K}|v_{i}|$, and let
$N(1/2):=N(B_{\infty}(1), \ell _{\infty},1/2)$ be the minimal covering
number of $B_{\infty}(1)$ by balls of $\ell _{\infty}$-radius $1/2$, with
centers $v_{(m)},~m=1,\ldots,N(1/2)$. By Proposition~4.3.34 in
\cite{gine_mathematical_2016}, we have $N(1/2) \le 6^{K}$. We deduce from
the triangle inequality for every fixed $v \in B_{\infty}(1)$ and some
$v_{(m)}$ that
\begin{align*}
\bigl\llvert v^{T} (\hat \Lambda -\Lambda ) \bigr\rrvert &\le \bigl
\llvert (v-v_{(m)})^{T}(\hat \Lambda - \Lambda ) \bigr
\rrvert + \bigl\llvert v_{(m)}^{T}(\hat \Lambda - \Lambda
) \bigr\rrvert \le \frac{1}{2} \llVert \hat \Lambda - \Lambda \rrVert
_{\ell _{1}} + \bigl\llvert v_{(m)}^{T}(\hat
\Lambda - \Lambda ) \bigr\rrvert.
\end{align*}
Taking suprema over $v \in B_{\infty}(1)$ and subtracting, we obtain
\begin{equation*}
\frac{1}{2} \llVert \hat \Lambda - \Lambda \rrVert _{\ell _{1}}
\le \max_{m=1,
\dots , N(1/2)} \bigl\llvert v_{(m)}^{T}(
\hat \Lambda - \Lambda ) \bigr\rrvert
\end{equation*}
so that the left-hand side~of (\ref{eq:concentration-ineq}) is bounded,
using also a union bound for probabilities, by\vspace*{6pt}
%
\begin{equation}
\label{tobound2} P_{\Lambda} \Bigl(\max_{m=1, \dots , N(1/2)}
\bigl\llvert v_{(m)}^{T}(\hat \Lambda - \Lambda ) \bigr
\rrvert \ge t \sqrt{K/n}/2 \Bigr) \le 6^{K} \sup_{v\in B_{\infty}(1)}P_{
\Lambda }
\bigl(n \bigl\llvert v^{T}(\hat \Lambda - \Lambda ) \bigr\rrvert \ge
t \sqrt{Kn}/2 \bigr). 
\end{equation}
We now need an appropriate tail bound for the random variables
\begin{equation*}
nv^{T}(\hat \Lambda - \Lambda ) = \sum
_{i=1}^{K} v_{i}(Y_{i}-
\mathbb{E}_{\Lambda} Y_{i}), \quad v \in B_{\infty}(1).
\end{equation*}
We consider each summand and $0<x<3$. First, if $0 \le v_{i} \le 1$, so
that $0<xv_{i} <3$, the moment generating function can be bounded as
\begin{align*}
\mathbb{E}_{\Lambda}\exp \bigl\{xv_{i}(Y_{i}-
\mathbb{E}_{\Lambda} Y_{i})\bigr\} &= \exp \bigl\{n \Lambda
(B_{i}) \bigl[e^{xv_{i}} -1 -xv_{i}\bigr] \bigr
\} \le \exp \biggl\{\frac{n \Lambda (B_{i}) (xv_{i})^{2}}{2(1- \llvert xv_{i} \rrvert /3)} \biggr\}
\\
&\le \exp \biggl\{\frac{n \Lambda (B_{i}) x^{2}}{2(1- \llvert x \rrvert /3)} \biggr\},
\end{align*}
using that the $Y_{i} \sim \operatorname{Poisson}(n \Lambda (B_{i}))$ as well as (3.10)
and (3.22) in \cite{gine_mathematical_2016}. Similarly, if
$-1<v_{i}<0$, so $0<-xv_{i}<3$. Then we have from
$e^{-y} + y -1 \le y^{2}/2$ for $-xv_{i} = y \in (0,3)$ that
\begin{equation*}
\mathbb{E}_{\Lambda} e^{xv_{i}(Y_{i}-\mathbb{E}_{\Lambda} Y_{i})} = e^{n
\Lambda (B_{i}) [e^{xv_{i}} -1 -xv_{i}]} \le
e^{n \Lambda (B_{i})(xv_{i})^{2}/2},
\end{equation*}
which is also upper bounded by the right-hand side~of the penultimate display.
In summary, by Markov's inequality and independence of the bin counts,
for any $z>0$ and $0<x<3$,
\begin{align*}
P_{\Lambda} \Biggl(\sum_{i=1}^{K}
v_{i}(Y_{i}-\mathbb{E}_{\Lambda}
Y_{i}) > z \Biggr) & \le \frac{\mathbb{E}_{\Lambda}[e^{\sum _{i=1}^{K} x v_{i}(Y_{i}-\mathbb{E}_{\Lambda} Y_{i})}]}{e^{xz}} \le
e^{-xz} \prod_{i=1}^{K}
\exp \biggl\{ \frac{n \Lambda (B_{i}) x^{2}}{2(1- \llvert x \rrvert /3)} \biggr\}
\\
&\le e^{-xz + \frac{n \Lambda (\mathbb O) x^{2}}{2-2 \llvert x \rrvert /3}} \le e^{-
\frac{z^{2}}{2n \Lambda (\mathbb O) + 2z}},
\end{align*}
after choosing $x= z/(n \Lambda (\mathbb O)+z) \in (0,1)$. Applying this
with $z=t \sqrt{n K}/2$ (and repeating the argument with $-v$ to deal with
$|\cdot |$) gives the following bound for (\ref{tobound2}):
\begin{equation*}
6^{K} \sup_{v\in B_{\infty}(1)} P_{\Lambda } \bigl(n
\bigl\llvert v^{T}( \hat \Lambda - \Lambda ) \bigr\rrvert \ge t
\sqrt{n K}/2 \bigr) \le e^{cK} e^{-
\frac{t^{2}nK/4}{2n \Lambda (\mathbb O) + t \sqrt{nK}}},
\end{equation*}
with universal constant $c>0$. Now since $K \le Cn$ we obtain that the
last tail is bounded by $ e^{-c'K}$ for all $t$ large enough, completing
the proof.
\end{proof}



\section{Additional proofs for Section \ref{pcont}} \label{sec:supp_stats}

\subsection{Proof of Lemma \ref{klbd}}

\begin{proof}
    
By independence and the definitions, it suffices to prove 
    $$- \mathbb{E}_{\Lambda_{0}}\left(\log \frac{p_{\Lambda}}{p_{\Lambda_0}}(Y)\right) \le 2 n\frac{|\Lambda(B)-\Lambda_{0}(B)|^2}{\Lambda_0(B)}$$ 
for every $Y=Y_i, B=B_i$ and where $p_\Lambda = p_{\Lambda, n}^1$ in this proof. From the definitions of Poisson densities we have
\begin{align}
    \frac{p_{\Lambda}}{p_{\Lambda_0}}(y) &=  e^{-n(\Lambda(B)-\Lambda_0(B))}\left[\frac{\Lambda(B)}{\Lambda_0(B)}\right]^y
\end{align}
and therefore
\begin{align*}
    \mathbb{E}_{\Lambda_0}\left[\log\frac{p_{\Lambda}}{p_{\Lambda_0}}(Y) \right]
    &= -n[\Lambda(B)-\Lambda_0(B)]+\log[\Lambda(B)/\Lambda_0(B)] \mathbb{E}_{\Lambda_0}[Y]\\
    &= -n[\Lambda(B)-\Lambda_0(B)]+n\Lambda_0(B) \log\left[1+\frac{\Lambda(B)-\Lambda_0(B)}{\Lambda_0(B)}\right].
  \end{align*}

Using a second order Taylor expansion about zero with mean value remainder 
$$\log(1+x)=x - \frac{1}{2}x^2\frac{1}{(1+\zeta x)^2} \hspace{1cm} \text{for some $\zeta \in (0,1)$}$$
at $x = \frac{\Lambda(B)-\Lambda_0(B)}{\Lambda_0(B)} \in (-1/2,1/2)$, we deduce

\begin{align*}
\log\left(1+\frac{\Lambda(B)-\Lambda_0(B)}{\Lambda_0(B)}\right)=\frac{\Lambda(B)-\Lambda_0(B)}{\Lambda_0(B)} - \frac{1}{2}\left(\frac{\Lambda(B)-\Lambda_0(B)}{\Lambda_0(B)}\right)^2 \frac{1}{(1+\zeta \frac{\Lambda(B)-\Lambda_0(B)}{\Lambda_0(B)})^2}.
\end{align*}

The last factor is bounded by $4$ under our hypothesis and we deduce

\begin{align*}
-\mathbb{E}_{P_{\Lambda_0}}\left[\log\frac{p_{\Lambda}}{p_{\Lambda_0}}(Y) \right]
    &= n \frac{(\Lambda(B)-\Lambda_0(B))^2}{2\Lambda_0(B)}\frac{1}{(1+\zeta \frac{\Lambda(B)-\Lambda_0(B)}{\Lambda_0(B)})^2} \le 2 n \frac{(\Lambda(B)-\Lambda_0(B))^2}{\Lambda_0(B)},
\end{align*}
completing the proof.

\end{proof}

\subsection{Proof of Lemma \ref{lemma:prelim-stats}}

\begin{proof}
We prove the result for $c\le2$ relevant below, in which case we will show that one can take $L_c=c^2/8$. The case $c>2$ follows from simple modifications and is left to the reader. By Jensen's inequality,
        \begin{align*}
            \log \int_{A_\epsilon} \frac{p_{\Lambda}}{p_{\Lambda_0}}(Y_1,...,Y_K)d\nu(\Lambda) \ge \int_{A_\epsilon} \log \frac{p_\Lambda}{p_{\Lambda_0}}(Y_1,...,Y_K)d\nu(\Lambda)
        \end{align*}
so using also Lemma \ref{klbd}, the probability in question is bounded, for $s>0$, by
    \begin{align*}
       P_{\Lambda_0} &\left(\int_{A_\epsilon} \log \frac{p_\Lambda}{p_{\Lambda_0}}(Y_1,...,Y_K)d\nu(\Lambda) \leq -  (2+c)n\epsilon^2  \right)\\
        &= P_{\Lambda_0}\Big(\int_{A_\epsilon} \big[\log \frac{p_\Lambda}{p_{\Lambda_0}}(Y_1,...,Y_K)-\mathbb{E}_{\Lambda_0}\log \frac{p_\Lambda}{p_{\Lambda_0}}(Y_1,...,Y_K)\big]d\nu(\Lambda)  \\
       & ~~~~~~~~~~~~~~~~~~~\leq  -(2+c)n\epsilon^2 - \int_{A_\epsilon} \mathbb{E}_{\Lambda_0}\log \frac{p_\Lambda}{p_{\Lambda_0}}(Y_1,...,Y_K) d\nu(\Lambda) \Big)\\
     &\le P_{\Lambda_0} \left(\int_{A_\epsilon} \big[\log \frac{p_\Lambda}{p_{\Lambda_0}}(Y_1,...,Y_K)-\mathbb{E}_{\Lambda_0}\log \frac{p_\Lambda}{p_{\Lambda_0}}(Y_1,...,Y_K)\big]d\nu(\Lambda)   \leq  -cn \epsilon^2 \right)\\
     &\le P_{\Lambda_0} \left(\exp{s\left |\mathbb{E}_{\nu} \left[\log \frac{p_\Lambda}{p_{\Lambda_0}}(Y_1,...,Y_K) - \mathbb{E}_{\Lambda_0}\log \frac{p_\Lambda}{p_{\Lambda_0}}(Y_1,...,Y_K)\right]\right|} \geq  e^{sc n\epsilon^2 }\right) \\
        &\leq e^{-snc\epsilon^2}\mathbb{E}_{\Lambda_0} \left[\exp{s\left|\mathbb{E}_{\nu} \left[\log \frac{p_\Lambda}{p_{\Lambda_0}}(Y_1,...,Y_K) - \mathbb{E}_{\Lambda_0}\log \frac{p_\Lambda}{p_{\Lambda_0}}(Y_1,...,Y_K)\right]\right|}\right]\\
        &\leq e^{-snc\epsilon^2}\mathbb{E}_{\nu} \left[\mathbb{E}_{\Lambda_0} \left[\exp{s\left|\log \frac{p_\Lambda}{p_{\Lambda_0}}(Y_1,...,Y_K) - \mathbb{E}_{\Lambda_0}\log \frac{p_\Lambda}{p_{\Lambda_0}}(Y_1,...,Y_K)\right|}\right] \right]
    \end{align*}
    using also Jensen's inequality and Fubini's theorem in the last step. We define random variables $$Z_i := \log \frac{p_\Lambda}{p_{\Lambda_0}}\left(Y_i\right) - \mathbb{E}_{\Lambda_0} \log \frac{p_\Lambda}{p_{\Lambda_0}}\left(Y_i\right)=Y_i \log \frac{\Lambda\left(B_i\right)}{\Lambda_0\left(B_i\right)} - n\Lambda_0(B_i)\log \frac{\Lambda\left(B_i\right)}{\Lambda_0\left(B_i\right)},~~i=1, \dots, K,$$ and estimate the expectation of $e^{s|\sum_i Z_i|} \le e^{s \sum_i Z_i}+e^{-s \sum_i Z_i}.$ For $e^{s \sum_i Z_i}$ we have by independence
\begin{align}
    \mathbb{E}_{\Lambda_0}[e^{s\sum_i Z_i}] &= \prod\limits_{i=1}^K \mathbb{E}_{\Lambda_0}\left[\exp\left(s\log \frac{\Lambda(B_i)}{\Lambda_0(B_i)} Y_i - s\log\frac{\Lambda(B_i)}{\Lambda_0(B_i)}  n\Lambda_0(B_i)\right)\right] \nonumber \\
    &= \prod\limits_{i=1}^K \exp\left(n\Lambda_0(B_i) \left(e^{s\log \frac{\Lambda(B_i)}{\Lambda_0(B_i)}} -1 - s\log \frac{\Lambda(B_i)}{\Lambda_0(B_i)}\right) \right). \label{eq:exp-before-taylor}
\end{align}

Using a second order Taylor expansion with mean value remainder, $0<\zeta<1$, and the hypothesis $\max_i|\frac{\Lambda(B_i)-\Lambda_0(B_i)}{\Lambda_0(B_i)}| <1/2$ we get for $0<s\le 1$ that
$$
e^{s \log \frac{\Lambda\left(B_i\right)}{\Lambda_0\left(B_i\right)}}=1+s \log \frac{\Lambda\left(B_i\right)}{\Lambda_0\left(B_i\right)}+\frac{s^2}{2} \left(\log \frac{\Lambda\left(B_i\right)}{\Lambda_0\left(B_i\right)}\right)^2 \underbrace{{e^{\zeta s \log \frac{\Lambda\left(B_i \right)}{\Lambda_0\left(B_i\right)}}}}_{C}
$$
with $C \le 3/2$. Similarly from the mean value theorem

$$\log \frac{\Lambda(B_i)}{\Lambda_0(B_i)}
= \log\left(1+ \frac{\Lambda(B_i)-\Lambda_0(B_i)}{\Lambda_0(B_i)} \right) 
= \frac{\Lambda(B_i)-\Lambda_0(B_i)}{\Lambda_0(B_i)} \underbrace{\cfrac{1}{1+\zeta \frac{\Lambda(B_i)-\Lambda_0(B_i)}{\Lambda_0(B_i)}}}_{C'}
$$ where $1/C' \le 2$. We conclude that (\ref{eq:exp-before-taylor}) is bounded by
\begin{align*}
    \mathbb{E}_{\Lambda_0}[e^{s\sum_i Z_i}]
    &\le \prod\limits_{i=1}^K \exp\left(n\Lambda_0(B_i) s^2 \big(\log \frac{\Lambda(B_i)}{\Lambda_0(B_i)}\big)^2  \right)\\
    &\le  \prod\limits_{i=1}^K \exp\left(2n s^2\frac{(\Lambda(B_i)-\Lambda_0(B_i))^2}{\Lambda_0(B_i)} \right)= e^{ 2s^2n \mathcal D^2_{2,K}(\Lambda, \Lambda_0)},~~0<s \le 1.
\end{align*}
and the same inequality can be established for $e^{-s\sum_i Z_i}$ by similar arguments. Combining what precedes with the support of the probability measure $\nu$ we obtain for all $0<s\le 1$ the bound
\begin{align*}
P_{\Lambda_0}\left(\int_{A_\epsilon} \frac{p_\Lambda}{p_{\Lambda_0}}(Y_1,...,Y_K)d\nu(\Lambda) \leq e^{-(2+c)n\epsilon^2} \right) 
&\leq  2e^{-snc\epsilon^2}\mathbb{E}_{\nu}  e^{ 2s^2n \mathcal D^2_{2,K}(\Lambda, \Lambda_0)} \\
&\leq  2e^{-snc\epsilon^2 + 2s^2n \epsilon^2} = 2e^{-n\epsilon^2 (sc- 2s^2)}.
\end{align*}
Choosing $s=c/4 <1$ (permitted since we assumed $c\le 2$) gives the overall bound
$$P_{\Lambda_0}\left(\int_{A_\epsilon} \frac{p_\Lambda}{p_{\Lambda_0}}(Y_1,...,Y_K)d\nu(\Lambda) \leq e^{-(2+c)n\epsilon^2} \right) \leq 2e^{-c^2n\epsilon^2/8},$$ completing the proof.

\end{proof}

\subsection{Proof of Theorem \ref{contract}}

\begin{proof}
    
  \textbf{Step 1.} Let us introduce the sets $$\mathcal{L}^n= \Theta_n \cap \mathcal R.$$ We can find tests (indicator functions) $\Psi_n \equiv \Psi(Y_1,...,Y_K,n)$ such that for every $n\in \mathbb{N}$ and $M>0$ large enough depending on $C, \Lambda_{max}$,
$$(i) ~ \mathbb E_{\Lambda_0}\Psi_n \to 0, \hspace{1cm} (ii) ~ \sup\limits_{\Lambda \in \mathcal L^n: ||\Lambda-\Lambda_0||_{\ell_1}\geq M\sqrt{K/n}} \mathbb{E}_{\Lambda}(1-\Psi_n) \leq Le^{-(C+4)n\varepsilon_n^2}$$

    \begin{proof}
    Define $\Psi_n = {1}\{||\hat{\Lambda}-\Lambda_0||_{\ell_1} \geq t\sqrt{K/n}\}$ for $\hat \Lambda$ as in (\ref{eq:estimator}). Part (i) follows from the concentration inequality Theorem \ref{thm:concentration-ineq} and so does (ii) in view of the following argument: note that
    $$\mathbb{E}_{\Lambda}[1-\Psi_n]=P_\Lambda(\Psi_n=0)=P_\Lambda(||\hat{\Lambda}-\Lambda_0||_{\ell_1} \leq t\sqrt{K/n}).$$
But 
    $$||\hat{\Lambda}-\Lambda_0||_{\ell_1} = ||\hat{\Lambda}-\Lambda+\Lambda -\Lambda_0||_{\ell_1} \geq ||\Lambda -\Lambda_0||_{\ell_1} - ||\hat{\Lambda}-\Lambda||_{\ell_1}$$
so that
    \begin{align*}
        \mathbb{E}_{\Lambda}[1-\Psi_n] &\leq P_{\Lambda} (||\Lambda -\Lambda_0||_{\ell_1} - ||\hat{\Lambda}-\Lambda||_{\ell_1} < t\sqrt{K/n})= P_{\Lambda} (||\hat{\Lambda}-\Lambda||_{\ell_1} > ||\Lambda -\Lambda_0||_{\ell_1} - t\sqrt{K/n})
    \end{align*}
    and hence, using Theorem \ref{thm:concentration-ineq}, we have for all $\Lambda \in \mathcal L^n$ such that $\|\Lambda-\Lambda_0\|_{\ell_1}\geq M\sqrt{K/n}$ the bound
    \begin{align*}
        \mathbb{E}_{\Lambda}(1-\Psi_n) &\leq   P_{\Lambda} (||\hat{\Lambda}-\Lambda||_{\ell_1} > ||\Lambda -\Lambda_0||_{\ell_1} - t\sqrt{K/n})\leq P_{\Lambda} (||\hat{\Lambda}-\Lambda||_{\ell_1} > M\sqrt{K/n} - t\sqrt{K/n})\\
        &\leq P_{\Lambda} (||\hat{\Lambda}-\Lambda||_{\ell_1} > (M-t)\sqrt{K/n}) \leq e^{-c'K} \le e^{-c'n\varepsilon_n^2}.
    \end{align*}
\end{proof}

  \textbf{Step 2.} Notice first

    $$\mathbb{E}_{\Lambda_0}\left[ \Pi(\{\Lambda\in (\mathcal{L}^n)^c \cup ||\Lambda-\Lambda_0||_{\ell_1} \geq M\sqrt{K/n} | Y_1,...,Y_K\}) \Psi_n\right] \leq \mathbb{E}_{\Lambda_0}\Psi_n \to 0$$
    by assumption on the tests. Let us write shorthand
    $$\bar L(K,n) \equiv  (\mathcal L^n)^c \cup  \{\Lambda: ||\Lambda-\Lambda_0||_{\ell_1} \geq M\sqrt{K/n}\}$$ so we only need to prove convergence in $P_{\Lambda_0}$ to $0$ of
    \begin{align*}
         \Pi(\bar L(K,n)| Y_1,...,Y_K)(1- \Psi_n)= \frac{\int_{\bar L(K,n)}(p_{\Lambda}/p_{\Lambda_0})(Y_1,...,Y_K)d\Pi(\Lambda)(1-\Psi_n)}{\int (p_{\Lambda}/p_{\Lambda_0})(Y_1,...,Y_K)d\Pi(\Lambda)}
    \end{align*}
    Lemma \ref{lemma:prelim-stats} shows that for all $c>0$ and probability measures $\nu$ with support in 
$$B_n:= \left\{ \Lambda \in \mathcal L^n: \mathcal{D}^2_{2,K}(\Lambda,\Lambda_0) \leq \varepsilon_n^2,~~ \mathcal D_{\infty, K}(\Lambda, \Lambda_0)\leq 1/2 \right\},$$
    we have 
    $$P_{\Lambda_0}\left(\int \frac{p_\Lambda}{p_{\Lambda_0}}(Y_1,...,Y_K)d\nu(\Lambda) \leq e^{-(2+c)n\varepsilon_n^2} \right) \leq 2e^{-L_c n\varepsilon_n^2}.$$
We apply this result with $c=1$ and $\nu$ equal to the normalised restriction of the prior $\Pi$ to $B_n$. If $A_n$ is the event
    $$A_n := \left\{ \int_{B_n} \frac{p_{\Lambda}}{p_{\Lambda_0}}(Y_1,...,Y_K)d\Pi(\Lambda) \geq \Pi(B_n)e^{-3n\varepsilon_n^2} \geq e^{-(3+C)n\varepsilon_n^2} \right\},$$
    then $P_{\Lambda_0}(A_n) \geq 1-2e^{-L_1 n\varepsilon_n^2} \to 1$, and we can write for every $\epsilon>0$,
\begin{align*}P_{\Lambda_0}&\left(\frac{\int_{\bar L(K,n)} (p_{\Lambda}/p_{\Lambda_0})(Y_1,...,Y_K)d\Pi(\Lambda)(1-\Psi_n)}{\int (p_\Lambda/p_{\Lambda_0})(Y_1,...,Y_K)d\Pi(\Lambda)} > \epsilon \right)\\
        &\leq P_{\Lambda_0}(A_n^c)+ P_{\Lambda_0}\left(e^{(3+C)n\varepsilon_n^2}(1-\Psi_n)\int_{\bar L(K,n)} \frac{p_{\Lambda}}{p_{\Lambda_0}}(Y_1,...,Y_K)d\Pi(\Lambda) \geq \epsilon\right).
    \end{align*}
    Now, using that
  $$\mathbb{E}_{\Lambda_0}\left[\frac{p_\Lambda}{p_{\Lambda_0}}(Y_1,...,Y_K)\right]=\int_{p_{\Lambda_0} >0} p_\Lambda d\mu \leq 1,\text{ ($\mu$ counting measure on $\mathbb{N}_0^K$)}$$
  $$\mathbb{E}_{\Lambda_0}\left[ \frac{p_\Lambda}{p_{\Lambda_0}}(Y_1,...,Y_K)(1-\Psi_n)\right] \leq \mathbb{E}_\Lambda(1-\Psi_n)$$
and that $0 \leq 1- \Psi_n \leq 1$, we obtain $$\mathbb{E}_{\Lambda_0}\left[(1-\Psi_n)\int_{\bar L(K,n)} \frac{p_\Lambda}{p_{\Lambda_0}}(Y_1,...,Y_K)d\Pi(\Lambda)\right] \leq 2\Pi(\mathcal{L}\text{\textbackslash}\mathcal{L}^n) + \sup\limits_{\Lambda\in \mathcal{L}^n:||\Lambda -\Lambda_0||_{\ell_1}\geq M\sqrt{K/n}} \mathbb{E}_{\Lambda}(1-\Psi_n).$$
Now the assumptions on $\mathcal{L}^n$ and concentration properties of the tests combined with Markov's inequality give, for every $\epsilon>0$,
$$P_{\Lambda_0}\left((1-\Psi_n)\int_{\bar L(K,n)} \frac{p_\Lambda}{p_{\Lambda_0}}(Y_1,...,Y_K)d\Pi(\Lambda) > \frac{\epsilon}{e^{(3+C)n\varepsilon_n^2}}\right) \leq (3L/\epsilon)e^{-n\varepsilon_n^2}$$ and the theorem follows by combining the preceding estimates and appropriate choice of $\epsilon$.
\end{proof}

\section{Proof of Theorem \ref{yetagain} (Stability Estimate)}\label{sec:stab-proof}

We first recall a general stability lemma for a transport operator, taken from \cite{nickl_consistent_2023} and inspired by earlier work in \cite{nickl_convergence_2020, nickl_bayesian_2023}.

\begin{condition}  \label{eq:condition}
    Let $u_0 \in H^2(\Omega)$ be a function such that $\sup_{x \in \Omega_{00}} |u_0(x)| \leq U < \infty$ and 
    \begin{equation}
        \frac{1}{2} \Delta u_0(x)+\mu |\nabla u_0(x)|^2 \geq c_0 >0, \hspace{0.5cm} a.e. \hspace{0.5cm} x \in \Omega_{00}
    \end{equation}
    for some subset $\Omega_{00}$ s.t. $dist(\Omega_{00},\partial\Omega)>0$.
\end{condition}

\begin{lemma} \label{transplem} (Lemma 2 in \cite{nickl_consistent_2023})
    Suppose $u_0$ (or $-u_0$) satisfies  Condition \ref{eq:condition}. Then there exists a constant $\underbar c = \underbar c (U, c_0, \mu)>0$ such that we have, for any $h\in C^1$ that vanishes on $\Omega \setminus \Omega_{00}$, $$\|\nabla \cdot (h \nabla u_0)\|_{L^2(\Omega)} \geq \underbar c \|h\|_{L^2(\Omega)}.$$
\end{lemma}

Given the preceding lemma, the proof of Theorem \ref{yetagain} proceeds in several steps by comparing the solution of (\ref{ellip}) to the solution of the standard Neumann problem with $q=0$, which exists if the `source' term $\phi$ integrates to zero. 

\smallskip

\textbf{Step 1:} For $\phi$ satisfying Conditions \ref{Dphicond} and \ref{ident} we set $$\phi_0 := \phi -1 =\phi - \int_\Omega \phi \text{ and }w_{\phi}= \mathcal{L}^{-1}_{D_1,0}[-\phi_0].$$ From Lemma \ref{regpde} combined with Proposition \ref{evergreen} we know that $w_\phi \in H^2_\nu \cap H^{\beta+2} \subset C^2$ for some $2 >\beta>d/2$. Let us show that $w_\phi$ further satisfies Condition \ref{eq:condition}
for some $\mu >0$ and subset $\Omega_{00} \subset \Omega$. 

\begin{enumerate}
    \item If $\phi = 1+\epsilon e_1$ with $|\nabla e_1|>0$ on $\Omega_{00}$,
    $$w_\phi = \mathcal{L}^{-1}_{D_1,0}[-\epsilon e_1] = \sum_{j=1}^{\infty} \frac{\epsilon \langle e_j, e_1 \rangle}{\lambda_j} e_j = \frac{\epsilon e_1}{\lambda_1}$$
    such that
    $|\nabla w_{\phi}|=|\frac{\epsilon \nabla e_1}{\lambda_1}|>0$ on $\Omega_{00}$. Since $\Delta w_\phi \in L^\infty$ we see that for $\mu$ large enough, Condition \ref{eq:condition} is satisfied with $u_0=w_\phi$.

    \item If $\phi \ge 1+\epsilon$ with $\epsilon >0$ on $\Omega_{00}$ then $\phi_0 \ge \epsilon$ on $\Omega_{00}$. Let us write $\bar w_\phi = - w_\phi$ then, 
    \begin{align}
        \phi_0 = \nabla \cdot (D_1 \nabla \bar w_\phi) = D_1 \Delta \bar w_\phi + \nabla D_1 \cdot \nabla \bar w_\phi.
    \end{align} This forces either
$D_1 \Delta \bar w_\phi (x) > \frac{\epsilon}{2} > 0$ and then $\Delta \bar w_\phi (x) > \frac{\epsilon}{2||D_1||_\infty} >0$, or 
$$\frac{\epsilon}{2} \leq \nabla D_1 (x) \cdot \nabla \bar w_\phi (x) \leq |\nabla \bar w_\phi (x)| |\nabla D_1 (x)|$$
so that  $|\nabla\bar  w_\phi|>\frac{\epsilon}{2||D_1||_{C^1}} $ on $\Omega_{00}$. Conclude that again, Condition \ref{eq:condition} is verified for $u_0=\bar w_\phi = - w_\phi$ and $\mu$ large enough, since $\Delta \bar w_\phi \in L^\infty$.
\end{enumerate}

\bigskip

\textbf{Step 2.} 
We next consider a constant potential $q=q_{\min}>0$ on the domain $\Omega$. In this case the spectral decomposition of $\mathcal{L}_{D_1, q_{\min}}$ is the same as the one of $\mathcal{L}_{D_1, 0}$ but with eigenvalues $$\lambda'_j = \lambda_j + q_{\min}, ~~ j \in \mathbb N \cup \{0\},$$ shifted by $q_{\min}$, including $\lambda_0=0$ corresponding to $e_0 =1$. The operator $\mathcal{L}_{D_1, q_{\min}}$ thus has inverse $\mathcal{L}^{-1}_{D_1, q_{\min}}$ on $L^2(\Omega)$ with spectral representation
    $$u=\mathcal{L}^{-1}_{D_1, q_{\min}}[\varphi] = -\sum\limits_{j=0}^\infty \frac{1}{\lambda'_j}e_j\langle e_j, \varphi \rangle,~~ \varphi \in L^2,$$
solving $\nabla \cdot (D \nabla u)- q_{\min}u=\varphi$ subject to Neumann boundary conditions. Now since $w_\phi$ solves $\mathcal{L}_{{D_1},0}w_\phi =-\phi_0$ with Neumann boundary conditions, the function
    $$v_{\phi} := w_{\phi} + \frac{\int_\Omega \phi}{q_{\min}}$$ also satisfies Neumann boundary conditions and solves
$$\mathcal{L}_{D_1, q_{\min}} v_{\phi} = \mathcal{L}_{D_1,0} v_\phi - q_{\min} v_\phi = -\phi_0 - \int_\Omega \phi - q_{\min} w_\phi = -\phi - q_{\min} w_\phi.$$

Therefore if $u=u_{D_1,q_{\min}} = \mathcal{L}^{-1}_{D_1,q_{\min}}[-\phi]$ is the actual solution of the steady state Schrödinger equation $$\nabla \cdot (D_1 \nabla u) - q_{\min} u = -\phi$$ with Neumann boundary conditions, then
\begin{align*}
    ||u_{D_1,q_{\min}} - v_\phi||_{C^2} &=||\mathcal{L}^{-1}_{D_1,q_{\min}}[-\phi]-\mathcal{L}^{-1}_{D_1,q_{\min}}[-\phi -q_{\min} w_\phi] ||_{C^2} \\
    &= q_{\min} ||\mathcal{L}^{-1}_{D_1,q_{\min}}[w_\phi]|| _{C^2}
    \lesssim q_{\min} ||\mathcal{L}^{-1}_{D_1,q_{\min}}[w_\phi]|| _{H^{2+\beta}}\\
    &\lesssim q_{\min} \|\mathcal{L}^{-1}_{D_1,q_{\min}}[w_\phi]\|_{\tilde H_{D_1, q_{\min}}^{2+\beta}} \lesssim q_{\min} 
\end{align*}
using the Sobolev imbedding ($d<4$), { {Proposition \ref{evergreen}} as well as the elliptic regularity Lemma \ref{regpde}. 
 {Note that since $w_\phi, \mathcal{L}^{-1}_{D_1,q_{\min}}[w_\phi]$ are orthogonal on constants, the preceding inequalities do not exhibit any further dependence on $q_{\min}$, in particular the preceding estimate implies 
$$q_{\min}\|u_{D_1, q_{\min}}\|_{H^2} \lesssim 1+\| w_\phi \|_{L^2}\lesssim 1,$$
to be used in the next step. Now }
when Condition \ref{eq:condition} holds for $w_{\phi}$ then it also holds for $v_\phi$ as the constant subtracted vanishes after taking the gradient and the Laplacian. 
 {Hence, by Condition (50), for $q_{min} \leq \delta$ small enough}
the function $u_{D_1,q_{\min}}$ satisfies Condition \ref{eq:condition} too since its $C^2$-norm distance to $v_\phi$ can be made as small as desired.
\smallskip

\textbf{Step 3.} Now, let $u_{D_1,q} = \mathcal{L}^{-1}_{D_1,q}[-\phi]$ be the solution of the steady state Schrödinger equation with non-constant potential $q$  {satisfying Condition (\ref{sergio})}. 
We then see that $$\mathcal L_{D_1,q}(u_{D_1,q}-u_{D_1,q_{\min}}) = (q-q_{\min}) u_{D_1, q_{\min}}$$ 
 {and further that
$$\mathcal L_{D_1,0}(u_{D_1,q}-u_{D_1,q_{\min}}) = (q-q_{\min}) u_{D_1, q_{\min}} + q(u_{D_1,q}-u_{D_1,q_{\min}})=: g,$$
where $\int_\Omega g=\int_\Omega (qu_{D,q}-q_{\min}u_{D,q_{\min}})=0$ in view of (\ref{probden}).}

We have that $u_{D_1,q_{\min}}$ lies in $\tilde H^2 = H^2_\nu$ and so do $q,q_{\min}$ as they belong to $H^2$ and are constant near the boundary. Therefore  $g$ belong to $H^2_\nu$ in view of (\ref{multi}) and $d<4$. Hence we obtain from the Sobolev imbedding with $d/2<\beta<2$, Proposition \ref{evergreen}, the definition of the $\tilde H^b$-norms, and again (\ref{multi}) 
 {
\begin{align}\label{eq:step3-corrected}
    \|u_{D_1,q} - u_{D_1,q_{\min}}\|_{C^2} 
    &\lesssim \|u_{D_1,q} - u_{D_1,q_{\min}}\|_{H^{2+\beta}}
    \lesssim \|\mathcal{L}^{-1}_{D_1,0}[g] \|_{\tilde H_{D_1, 0}^{2+\beta}} 
    \lesssim \|g\|_{\tilde H_{D_1,0}^2} \notag\\
    &\lesssim \underbrace{\|(q-q_{\min})u_{D_1,q_{\min}}\|_{ H^2}}_{(I)} + \underbrace{\|q(u_{D_1,q} - u_{D_1,q_{\min}})\|_{H^2}}_{(II)}.
\end{align}
For $(I)$, using Condition (\ref{sergio}) and the estimate from Step 2,
$$\|(q-q_{\min})u_{D_1,q_{\min}}\|_{H^2}\lesssim \displaystyle \left\|\frac{q-q_{\min}}{q_{\min}}\right\|_{H^2} 
    q_{\min} \|u_{D_1,q_{\min}}\|_{H^2} \lesssim \delta.$$
For $(II)$, by Proposition \ref{evergreen} (and its proof, combined with $\lambda_{1,D,q} \gtrsim q_{\min}^{-1}$ in view of the proof of Lemma \ref{qgap})  this time applied to the operator $\mathcal{L}_{D,q}$
\begin{align*}
    \|q(u_{D_1,q} &- u_{D_1,q_{\min}})\|_{H^2} 
    \lesssim \|q\|_{H^2} \|u_{D_1,q} - u_{D_1,q_{\min}}\|_{H^2}\\
    &\lesssim q_{\min}^{-1}\|q\|_{H^2} \|u_{D_1,q} - u_{D_1,q_{\min}}\|_{\tilde H^2_{D,q}}
    = q_{\min}^{-1}\|q\|_{H^2} \|\mathcal{L}_{D_1,q}(u_{D_1,q} - u_{D_1,q_{\min}})\|_{L^2}\\
    &= q_{\min}^{-1}\|q\|_{H^2} \|(q-q_{\min})u_{D_1,q_{\min}}\|_{L^2} \lesssim \delta \left\|1+\frac{q-q_{\min}}{q_{\min}}\right \|_{H^2} \lesssim (1+\delta)\delta
\end{align*}
by $(I)$ and Condition $(\ref{sergio})$.
}
 {Conclude that the r.h.s.~of $(\ref{eq:step3-corrected})$} can be made as small as desired for  {$\delta$}
small enough, so that again Condition \ref{eq:condition} is inherited by $u_{D_1,q}$. From what precedes Lemma \ref{transplem} applies to give for any $h \in C^1$ that vanishes on $\Omega $\textbackslash$ \Omega_{00}$
$$||\nabla \cdot (h \nabla u_{D_1,q})||_{L^2} \geq c ||h||_{L^2}$$
and applying this to $h=D_1-D_2$ gives 
\begin{equation}
    ||D_1 - D_2 ||_{L^2} \lesssim || \nabla \cdot ((D_1-D_2) \nabla u_{D_1,q})||_{L^2}.
\end{equation}

Now, since $u_{D_1,q} $ solves $\mathcal{L}_{D_1,q}u=-\phi$, and $u_{D_2,q}$ solves the same equation with $D_2$ in place of $D_1$, the r.h.s equals, with all $L^2$-norms over $\Omega_{00}$,
\begin{align*}
    || \nabla \cdot ((D_1-D_2) \nabla u_{D_1,q})||^2_{L^2} &= || \nabla \cdot (D_1 \nabla u_{D_1,q}) - \nabla \cdot (D_2 \nabla (u_{D_1,q}-u_{D_2,q}+u_{D_2,q}))||^2_{L^2}\\
    &\lesssim || - \phi + q u_{D_1,q} + \phi - q u_{D_2,q}||^2_{L^2} + ||\nabla \cdot (D_2 \nabla (u_{D_1,q}-u_{D_2,q}))||^2_{L^2} \\
    &= ||q (u_{D_1,q} - u_{D_2,q})||^2_{L^2} + ||\nabla \cdot (D_2 \nabla (u_{D_1,q}-u_{D_2,q}))||^2_{L^2}\\
    & \lesssim || u_{D_1,q} - u_{D_2,q}||^2_{H^2(\Omega_{00})}.
\end{align*}
This completes the proof of Theorem \ref{yetagain}.

\begin{acks}[Acknowledgments]
The authors would like to thank Aleksandra Jartseva, David Klenerman and
Ernest Laue for stimulating discussions and for drawing our attention to
the problem of inferring diffusivity in biochemistry. The authors would
also like to thank three anonymous referees, the associate editor,  {Judith Rousseau} as well
as Paula Horvat, for their constructive comments,  {and Sergio Bacallado for pointing out a mistake in an earlier version of the manuscript.}
\end{acks}
\begin{funding}
RN and FS were supported
by ERC Advanced Grant (Horizon Europe UKRI G116786) and RN was further
supported by EPSRC Grant EP/V026259.
\end{funding}




\begin{thebibliography}{41}

\bibitem{adams_sobolev_2003}
\begin{bbook}[mr]
\bauthor{\bsnm{Adams},~\bfnm{Robert~A.}\binits{R.~A.}} \AND
\bauthor{\bsnm{Fournier},~\bfnm{John~J.~F.}\binits{J.~J.~F.}}
(\byear{2003}).
\btitle{Sobolev Spaces},
\bedition{2nd} ed.
\bseries{Pure and Applied Mathematics (Amsterdam)}
\bvolume{140}.
\bpublisher{Elsevier/Academic Press},
\blocation{Amsterdam}.
\bid{mr={2424078}}
\end{bbook}
%
\OrigBibText
%
\begin{bbook}[author]
\bauthor{\bsnm{Adams},~\bfnm{Robert~A.}\binits{R.~A.}} \AND
\bauthor{\bsnm{Fournier},~\bfnm{John J.~F.}\binits{J.~J.~F.}} (\byear{2003}).
\btitle{Sobolev Spaces}. \bpublisher{Elsevier}.
\end{bbook}
%
\endOrigBibText
\bptok{imsref}%
\endbibitem

\bibitem{bass_diffusions_1998}
\begin{bbook}[mr]
\bauthor{\bsnm{Bass},~\bfnm{Richard~F.}\binits{R.~F.}}
(\byear{1998}).
\btitle{Diffusions and Elliptic Operators}.
\bseries{Probability and Its Applications (New York)}.
\bpublisher{Springer},
\blocation{New York}.
\bid{mr={1483890}}
\end{bbook}
%
\OrigBibText
%
\begin{bbook}[author]
\bauthor{\bsnm{Bass},~\bfnm{Richard}\binits{R.}} (\byear{1998}).
\btitle{Diffusions and Elliptic Operators}.
\bseries{Probability and its Applications}.
\bpublisher{Springer-Verlag}, \baddress{New York}.
\end{bbook}
%
\endOrigBibText
\bptok{imsref}%
\endbibitem

\bibitem{bass_stochastic_2011}
\begin{bbook}[mr]
\bauthor{\bsnm{Bass},~\bfnm{Richard~F.}\binits{R.~F.}}
(\byear{2011}).
\btitle{Stochastic Processes}.
\bseries{Cambridge Series in Statistical and Probabilistic Mathematics}
\bvolume{33}.
\bpublisher{Cambridge Univ. Press},
\blocation{Cambridge}.
\bid{doi={10.1017/CBO9780511997044}, doi={10.1017/CBO9780511997044}, mr={2856623}}
\end{bbook}
%
\OrigBibText
%
\begin{bbook}[author]
\bauthor{\bsnm{Bass},~\bfnm{Richard~F.}\binits{R.~F.}} (\byear{2011}).
\btitle{Stochastic Processes}.
\bseries{Cambridge Series in Statistical and Probabilistic Mathematics}.
\bpublisher{Cambridge University Press}, \baddress{Cambridge}.
\end{bbook}
%
\endOrigBibText
\bptok{imsref}%
\endbibitem

\bibitem{basu_live-cell_2021}
\begin{bmisc}[author]
\bauthor{\bsnm{Basu},~\bfnm{S.}\binits{S.}} \betal{et al.}
(\byear{2021}).
\btitle{Live-cell 3D single-molecule tracking reveals how NuRD modulates enhancer dynamics}.
\bid{doi={10.1101/2020.04.03.003178}}
\end{bmisc}
%
\OrigBibText
%
\begin{bmisc}[author]
\bauthor{\bsnm{Basu},~\bfnm{S.~et~al.}\binits{S.~e.~a.}} (\byear{2021}).
\btitle{Live-cell 3D single-molecule tracking reveals how NuRD modulates enhancer dynamics}.
\end{bmisc}
%
\endOrigBibText
\bptok{imsref}%
\endbibitem

\bibitem{belitser_rate-optimal_2015}
\begin{barticle}[mr]
\bauthor{\bsnm{Belitser},~\bfnm{Eduard}\binits{E.}},
\bauthor{\bsnm{Serra},~\bfnm{Paulo}\binits{P.}} \AND
\bauthor{\bparticle{van} \bsnm{Zanten},~\bfnm{Harry}\binits{H.}}
(\byear{2015}).
\btitle{Rate-optimal {B}ayesian intensity smoothing for inhomogeneous {P}oisson processes}.
\bjournal{J. Statist. Plann. Inference}
\bvolume{166}
\bpages{24--35}.
\bid{doi={10.1016/j.jspi.2014.03.009}, doi={10.1016/j.jspi.2014.03.009}, mr={3390131}, issn={0378-3758,1873-1171}}
\end{barticle}
%
\OrigBibText
%
\begin{barticle}[author]
\bauthor{\bsnm{Belitser},~\bfnm{Eduard}\binits{E.}},
\bauthor{\bsnm{Serra},~\bfnm{Paulo}\binits{P.}} \AND
\bauthor{\bparticle{van} \bsnm{Zanten},~\bfnm{Harry}\binits{H.}} (\byear{2015}).
\btitle{Rate-optimal Bayesian intensity smoothing for inhomogeneous Poisson processes}.
\bjournal{Journal of Statistical Planning and Inference}
\bvolume{166} \bpages{24--35}. 
\end{barticle}
%
\endOrigBibText
\bptok{imsref}%
\endbibitem

\bibitem{chung_brownian_1995}
\begin{bbook}[mr]
\bauthor{\bsnm{Chung},~\bfnm{Kai~Lai}\binits{K.~L.}} \AND
\bauthor{\bsnm{Zhao},~\bfnm{Zhong~Xin}\binits{Z.~X.}}
(\byear{1995}).
\btitle{From {B}rownian Motion to {S}chr\"{o}dinger's Equation}.
\bseries{Grundlehren der Mathematischen Wissenschaften [Fundamental Principles of Mathematical Sciences]}
\bvolume{312}.
\bpublisher{Springer},
\blocation{Berlin}.
\bid{doi={10.1007/978-3-642-57856-4}, doi={10.1007/978-3-642-57856-4}, mr={1329992}}
\end{bbook}
%
\OrigBibText
%
\begin{bbook}[author]
\bauthor{\bsnm{Chung},~\bfnm{Kai~Lai}\binits{K.~L.}} \AND
\bauthor{\bsnm{Zhao},~\bfnm{Zhongxin}\binits{Z.}} (\byear{1995}).
\btitle{From Brownian Motion to Schr\"{o}dinger's Equation}.
\bseries{Grundlehren der mathematischen Wissenschaften}
\bvolume{312}. \bpublisher{Springer}, \baddress{Berlin, Heidelberg}.
\end{bbook}
%
\endOrigBibText
\bptok{imsref}%
\endbibitem

\bibitem{cotter_mcmc_2013}
\begin{barticle}[mr]
\bauthor{\bsnm{Cotter},~\bfnm{S.~L.}\binits{S.~L.}},
\bauthor{\bsnm{Roberts},~\bfnm{G.~O.}\binits{G.~O.}},
\bauthor{\bsnm{Stuart},~\bfnm{A.~M.}\binits{A.~M.}} \AND
\bauthor{\bsnm{White},~\bfnm{D.}\binits{D.}}
(\byear{2013}).
\btitle{M{CMC} methods for functions: Modifying old algorithms to make them faster}.
\bjournal{Statist. Sci.}
\bvolume{28}
\bpages{424--446}.
\bid{doi={10.1214/13-STS421}, doi={10.1214/13-STS421}, mr={3135540}, issn={0883-4237,2168-8745}}
\end{barticle}
%
\OrigBibText
%
\begin{barticle}[author]
\bauthor{\bsnm{Cotter},~\bfnm{S.~L.}\binits{S.~L.}},
\bauthor{\bsnm{Roberts},~\bfnm{G.~O.}\binits{G.~O.}},
\bauthor{\bsnm{Stuart},~\bfnm{A.~M.}\binits{A.~M.}} \AND
\bauthor{\bsnm{White},~\bfnm{D.}\binits{D.}} (\byear{2013}).
\btitle{MCMC Methods for Functions: Modifying Old Algorithms to Make Them Faster}.
\bjournal{Statistical Science} \bvolume{28} \bpages{424--446}.
\bnote{Publisher: Institute of Mathematical Statistics}.
\end{barticle}
%
\endOrigBibText
\bptok{imsref}%
\endbibitem

\bibitem{davies_spectral_1995}
\begin{bbook}[mr]
\bauthor{\bsnm{Davies},~\bfnm{E.~B.}\binits{E.~B.}}
(\byear{1995}).
\btitle{Spectral Theory and Differential Operators}.
\bseries{Cambridge Studies in Advanced Mathematics}
\bvolume{42}.
\bpublisher{Cambridge Univ. Press},
\blocation{Cambridge}.
\bid{doi={10.1017/CBO9780511623721}, doi={10.1017/CBO9780511623721}, mr={1349825}}
\end{bbook}
%
\OrigBibText
%
\begin{bbook}[author]
\bauthor{\bsnm{Davies},~\bfnm{E.~Brian}\binits{E.~B.}} (\byear{1995}).
\btitle{Spectral Theory and Differential Operators}.
\bseries{Cambridge Studies in Advanced Mathematics}.
\bpublisher{Cambridge University Press}, \baddress{Cambridge}.
\end{bbook}
%
\endOrigBibText
\bptok{imsref}%
\endbibitem

\bibitem{donnet_posterior_2017}
\begin{barticle}[mr]
\bauthor{\bsnm{Donnet},~\bfnm{Sophie}\binits{S.}},
\bauthor{\bsnm{Rivoirard},~\bfnm{Vincent}\binits{V.}},
\bauthor{\bsnm{Rousseau},~\bfnm{Judith}\binits{J.}} \AND
\bauthor{\bsnm{Scricciolo},~\bfnm{Catia}\binits{C.}}
(\byear{2017}).
\btitle{Posterior concentration rates for counting processes with {A}alen multiplicative intensities}.
\bjournal{Bayesian Anal.}
\bvolume{12}
\bpages{53--87}.
\bid{doi={10.1214/15-BA986}, doi={10.1214/15-BA986}, mr={3597567}, issn={1936-0975,1931-6690}}
\end{barticle}
%
\OrigBibText
%
\begin{barticle}[author]
\bauthor{\bsnm{Donnet},~\bfnm{Sophie}\binits{S.}},
\bauthor{\bsnm{Rivoirard},~\bfnm{Vincent}\binits{V.}},
\bauthor{\bsnm{Rousseau},~\bfnm{Judith}\binits{J.}} \AND
\bauthor{\bsnm{Scricciolo},~\bfnm{Catia}\binits{C.}} (\byear{2017}).
\btitle{Posterior Concentration Rates for Counting Processes with Aalen Multiplicative Intensities}.
\bjournal{Bayesian Analysis} \bvolume{12}. 
\end{barticle}
%
\endOrigBibText
\bptok{imsref}%
\endbibitem

\bibitem{evans_partial_2010}
\begin{bbook}[mr]
\bauthor{\bsnm{Evans},~\bfnm{Lawrence~C.}\binits{L.~C.}}
(\byear{2010}).
\btitle{Partial Differ. Equ.},
\bedition{2nd} ed.
\bseries{Graduate Studies in Mathematics}
\bvolume{19}.
\bpublisher{Amer. Math. Soc.},
\blocation{Providence, RI}.
\bid{doi={10.1090/gsm/019}, doi={10.1090/gsm/019}, mr={2597943}}
\end{bbook}
%
\OrigBibText
%
\begin{bbook}[author]
\bauthor{\bsnm{Evans},~\bfnm{Lawrence~C.}\binits{L.~C.}} (\byear{2010}).
\btitle{Partial Differential Equations}.
\bpublisher{American Mathematical Soc.}
\end{bbook}
%
\endOrigBibText
\bptok{imsref}%
\endbibitem

\bibitem{folland_real_1999}
\begin{bbook}[mr]
\bauthor{\bsnm{Folland},~\bfnm{Gerald~B.}\binits{G.~B.}}
(\byear{1999}).
\btitle{Real Analysis: Modern Techniques and Their Applications},
\bedition{2nd} ed.
\bseries{Pure and Applied Mathematics (New York)}.
\bpublisher{Wiley},
\blocation{New York}.
\bid{mr={1681462}}
\end{bbook}
%
\OrigBibText
%
\begin{bbook}[author]
\bauthor{\bsnm{Folland},~\bfnm{G.~B.}\binits{G.~B.}} (\byear{1999}).
\btitle{Real Analysis: Modern Techniques and Their Applications, 2nd Edition {\textbar} Wiley}.
\end{bbook}
%
\endOrigBibText
\bptok{imsref}%
\endbibitem

\bibitem{F85}
\begin{bbook}[mr]
\bauthor{\bsnm{Freidlin},~\bfnm{Mark}\binits{M.}}
(\byear{1985}).
\btitle{Functional Integration and Partial Differential Equations}.
\bseries{Annals of Mathematics Studies}
\bvolume{109}.
\bpublisher{Princeton Univ. Press},
\blocation{Princeton, NJ}.
\bid{doi={10.1515/9781400881598}, doi={10.1515/9781400881598}, mr={0833742}}
\end{bbook}
%
\OrigBibText
%
\begin{bbook}[author]
\bauthor{\bsnm{Freidlin},~\bfnm{Mark}\binits{M.}} (\byear{1985}).
\btitle{Functional integration and partial differential equations}.
\bseries{Annals of Mathematics Studies} \bvolume{109}.
\bpublisher{Princeton University Press, Princeton, NJ}.
\end{bbook}
%
\endOrigBibText
\bptok{imsref}%
\endbibitem

\bibitem{ghosal_fundamentals_2017}
\begin{bbook}[mr]
\bauthor{\bsnm{Ghosal},~\bfnm{Subhashis}\binits{S.}} \AND
\bauthor{\bsnm{van~der Vaart},~\bfnm{Aad}\binits{A.}}
(\byear{2017}).
\btitle{Fundamentals of Nonparametric {B}ayesian Inference}.
\bseries{Cambridge Series in Statistical and Probabilistic Mathematics}
\bvolume{44}.
\bpublisher{Cambridge Univ. Press},
\blocation{Cambridge}.
\bid{doi={10.1017/9781139029834}, doi={10.1017/9781139029834}, mr={3587782}}
\end{bbook}
%
\OrigBibText
%
\begin{bbook}[author]
\bauthor{\bsnm{Ghosal},~\bfnm{Subhashis}\binits{S.}} \AND
\bauthor{\bparticle{van~der} \bsnm{Vaart},~\bfnm{Aad}\binits{A.}} (\byear{2017}).
\btitle{Fundamentals of Nonparametric Bayesian Inference}.
\bseries{Cambridge Series in Statistical and Probabilistic Mathematics}.
\bpublisher{Cambridge University Press}, \baddress{Cambridge}.
\end{bbook}
%
\endOrigBibText
\bptok{imsref}%
\endbibitem

\bibitem{gine_rates_2011}
\begin{barticle}[mr]
\bauthor{\bsnm{Gin{\'{e}}},~\bfnm{Evarist}\binits{E.}} \AND
\bauthor{\bsnm{Nickl},~\bfnm{Richard}\binits{R.}}
(\byear{2011}).
\btitle{Rates of contraction for posterior distributions in {$L^r$}-metrics, {$1\leq r\leq\infty$}}.
\bjournal{Ann. Statist.}
\bvolume{39}
\bpages{2883--2911}.
\bid{doi={10.1214/11-AOS924}, doi={10.1214/11-AOS924}, mr={3012395}, issn={0090-5364,2168-8966}}
\end{barticle}
%
\OrigBibText
%
\begin{barticle}[author]
\bauthor{\bsnm{Gin\'e},~\bfnm{Evarist}\binits{E.}} \AND
\bauthor{\bsnm{Nickl},~\bfnm{Richard}\binits{R.}} (\byear{2011}).
\btitle{Rates of contraction for posterior distributions in $L^{r}$-metrics, $1\leq r\leq \infty $}.
\bjournal{Ann. Statist.} \bvolume{39} \bpages{2883--2911}.
\end{barticle}
%
\endOrigBibText
\bptok{imsref}%
\endbibitem

\bibitem{gine_mathematical_2016}
\begin{bbook}[mr]
\bauthor{\bsnm{Gin{\'{e}}},~\bfnm{Evarist}\binits{E.}} \AND
\bauthor{\bsnm{Nickl},~\bfnm{Richard}\binits{R.}}
(\byear{2016}).
\btitle{Mathematical Foundations of Infinite-Dimensional Statistical Models}.
\bseries{Cambridge Series in Statistical and Probabilistic Mathematics}
\bvolume{40}.
\bpublisher{Cambridge Univ. Press},
\blocation{New York}.
\bid{doi={10.1017/CBO9781107337862}, doi={10.1017/CBO9781107337862}, mr={3588285}}
\end{bbook}
%
\OrigBibText
%
\begin{bbook}[author]
\bauthor{\bsnm{Gin\'e},~\bfnm{Evarist}\binits{E.}} \AND
\bauthor{\bsnm{Nickl},~\bfnm{Richard}\binits{R.}} (\byear{2016}).
\btitle{Mathematical foundations of infinite-dimensional statistical models}.
\bseries{Cambridge Series in Statistical and Probabilistic Mathematics}
\bvolume{[40]}. \bpublisher{Cambridge University Press, New York}.
\end{bbook}
%
\endOrigBibText
\bptok{imsref}%
\endbibitem

\bibitem{giordano2023bayesian}
\begin{bmisc}[author]
\bauthor{\bsnm{Giordano},~\bfnm{Matteo}\binits{M.}}
(\byear{2023}).
\btitle{Bayesian nonparametric inference in PDE models: Asymptotic theory and implementation}.
\bnote{Preprint. Available at \arxivurl{2311.18322}}.
\end{bmisc}
%
\OrigBibText
%
\begin{barticle}[author]
\bauthor{\bsnm{Giordano},~\bfnm{Matteo}\binits{M.}} (\byear{2023}).
\btitle{Bayesian nonparametric inference in PDE models: asymptotic theory and implementation}.
\bjournal{arXiv preprint arXiv:2311.18322}.
\end{barticle}
%
\endOrigBibText
\bptok{imsref}%
\endbibitem

\bibitem{giordano_nonparametric_2023}
\begin{bmisc}[author]
\bauthor{\bsnm{Giordano},~\bfnm{M.}\binits{M.}},
\bauthor{\bsnm{Kirichenko},~\bfnm{A.}\binits{A.}} \AND
\bauthor{\bsnm{Rousseau},~\bfnm{J.}\binits{J.}}
(\byear{2023}).
\btitle{Nonparametric Bayesian intensity estimation for covariate-driven inhomogeneous point processes}.
\bnote{Preprint. Available at \arxivurl{2312.14073} [math]}.
\end{bmisc}
%
\OrigBibText
%
\begin{bmisc}[author]
\bauthor{\bsnm{Giordano},~\bfnm{Matteo}\binits{M.}},
\bauthor{\bsnm{Kirichenko},~\bfnm{Alisa}\binits{A.}} \AND
\bauthor{\bsnm{Rousseau},~\bfnm{Judith}\binits{J.}} (\byear{2023}).
\btitle{Nonparametric Bayesian intensity estimation for covariate-driven inhomogeneous point processes}.
\bnote{arXiv:2312.14073 [math]}. 
\end{bmisc}
%
\endOrigBibText
\bptok{imsref}%
\endbibitem

\bibitem{GR22}
\begin{barticle}[mr]
\bauthor{\bsnm{Giordano},~\bfnm{Matteo}\binits{M.}} \AND
\bauthor{\bsnm{Ray},~\bfnm{Kolyan}\binits{K.}}
(\byear{2022}).
\btitle{Nonparametric {B}ayesian inference for reversible multidimensional diffusions}.
\bjournal{Ann. Statist.}
\bvolume{50}
\bpages{2872--2898}.
\bid{doi={10.1214/22-aos2213}, doi={10.1214/22-aos2213}, mr={4500628}, issn={0090-5364,2168-8966}}
\end{barticle}
%
\OrigBibText
%
\begin{barticle}[author]
\bauthor{\bsnm{Giordano},~\bfnm{Matteo}\binits{M.}} \AND
\bauthor{\bsnm{Ray},~\bfnm{Kolyan}\binits{K.}} (\byear{2022}).
\btitle{Nonparametric Bayesian inference for reversible multidimensional diffusions}.
\bjournal{Ann. Statist.} \bvolume{50} \bpages{2872--2898}.
\end{barticle}
%
\endOrigBibText
\bptok{imsref}%
\endbibitem

\bibitem{giordano_statistical_2025}
\begin{barticle}[mr]
\bauthor{\bsnm{Giordano},~\bfnm{Matteo}\binits{M.}} \AND
\bauthor{\bsnm{Wang},~\bfnm{Sven}\binits{S.}}
(\byear{2025}).
\btitle{Statistical algorithms for low-frequency diffusion data: A {PDE} approach}.
\bjournal{Ann. Statist.}
\bvolume{53}
\bpages{1150--1175}.
\bid{doi={10.1214/25-aos2496}, doi={10.1214/25-aos2496}, mr={4925119}, issn={0090-5364,2168-8966}}
\end{barticle}
%
\OrigBibText
%
\begin{bmisc}[author]
\bauthor{\bsnm{Giordano},~\bfnm{Matteo}\binits{M.}} \AND
\bauthor{\bsnm{Wang},~\bfnm{Sven}\binits{S.}} (\byear{2025}).
\btitle{Statistical algorithms for low-frequency diffusion data: A PDE approach}.
\bnote{Annals of Statistics, to appear}.
\end{bmisc}
%
\endOrigBibText
\bptok{imsref}%
\endbibitem

\bibitem{GHR04}
\begin{barticle}[mr]
\bauthor{\bsnm{Gobet},~\bfnm{Emmanuel}\binits{E.}},
\bauthor{\bsnm{Hoffmann},~\bfnm{Marc}\binits{M.}} \AND
\bauthor{\bsnm{Rei{\ss}},~\bfnm{Markus}\binits{M.}}
(\byear{2004}).
\btitle{Nonparametric estimation of scalar diffusions based on low frequency data}.
\bjournal{Ann. Statist.}
\bvolume{32}
\bpages{2223--2253}.
\bid{doi={10.1214/009053604000000797}, doi={10.1214/009053604000000797}, mr={2102509}, issn={0090-5364,2168-8966}}
\end{barticle}
%
\OrigBibText
%
\begin{barticle}[author]
\bauthor{\bsnm{Gobet},~\bfnm{Emmanuel}\binits{E.}},
\bauthor{\bsnm{Hoffmann},~\bfnm{Marc}\binits{M.}} \AND
\bauthor{\bsnm{Rei\ss},~\bfnm{Markus}\binits{M.}} (\byear{2004}).
\btitle{Nonparametric estimation of scalar diffusions based on low frequency data}.
\bjournal{Ann. Statist.} \bvolume{32} \bpages{2223--2253}.
\end{barticle}
%
\endOrigBibText
\bptok{imsref}%
\endbibitem

\bibitem{gugushvili_fast_2020}
\begin{bmisc}[author]
\bauthor{\bsnm{Gugushvili},~\bfnm{S.}\binits{S.}},
\bauthor{\bsnm{van~der Meulen},~\bfnm{F.}\binits{F.}},
\bauthor{\bsnm{Schauer},~\bfnm{M.}\binits{M.}} \AND
\bauthor{\bsnm{Spreij},~\bfnm{P.}\binits{P.}}
(\byear{2020}).
\btitle{Fast and scalable non-parametric Bayesian inference for Poisson point processes}.
\bnote{Preprint. Available at \arxivurl{1804.03616} [stat]}.
\end{bmisc}
%
\OrigBibText
%
\begin{bmisc}[author]
\bauthor{\bsnm{Gugushvili},~\bfnm{Shota}\binits{S.}},
\bauthor{\bparticle{van~der} \bsnm{Meulen},~\bfnm{Frank}\binits{F.}},
\bauthor{\bsnm{Schauer},~\bfnm{Moritz}\binits{M.}} \AND
\bauthor{\bsnm{Spreij},~\bfnm{Peter}\binits{P.}} (\byear{2020}).
\btitle{Fast and scalable non-parametric Bayesian inference for Poisson point processes}.
\bnote{arXiv:1804.03616 [stat]}.
\end{bmisc}
%
\endOrigBibText
\bptok{imsref}%
\endbibitem

\bibitem{heckert_recovering_2022}
\begin{barticle}[author]
\bauthor{\bsnm{Heckert},~\bfnm{A.}\binits{A.}},
\bauthor{\bsnm{Dahal},~\bfnm{L.}\binits{L.}},
\bauthor{\bsnm{Tjian},~\bfnm{R.}\binits{R.}} \AND
\bauthor{\bsnm{Darzacq},~\bfnm{X.}\binits{X.}}
(\byear{2022}).
\btitle{Recovering mixtures of fast-diffusing states from short single-particle trajectories}.
\bjournal{eLife}
\bvolume{11}
\bpages{e70169}.
\end{barticle}
%
\OrigBibText
%
\begin{barticle}[author]
\bauthor{\bsnm{Heckert},~\bfnm{Alec}\binits{A.}},
\bauthor{\bsnm{Dahal},~\bfnm{Liza}\binits{L.}},
\bauthor{\bsnm{Tjian},~\bfnm{Robert}\binits{R.}} \AND
\bauthor{\bsnm{Darzacq},~\bfnm{Xavier}\binits{X.}} (\byear{2022}).
\btitle{Recovering mixtures of fast-diffusing states from short single-particle trajectories}.
\bjournal{eLife} \bvolume{11} \bpages{e70169}.
\bnote{Publisher: eLife Sciences Publications, Ltd}.
\end{barticle}
%
\endOrigBibText
\bptok{imsref}%
\endbibitem

\bibitem{heltberg_physical_2021}
\begin{bmisc}[author]
\bauthor{\bsnm{Heltberg},~\bfnm{M.~L.}\binits{M.~L.}},
\bauthor{\bsnm{Min{\'{e}}-Hattab},~\bfnm{J.}\binits{J.}},
\bauthor{\bsnm{Taddei},~\bfnm{A.}\binits{A.}},
\bauthor{\bsnm{Walczak},~\bfnm{A.~M.}\binits{A.~M.}} \AND
\bauthor{\bsnm{Mora},~\bfnm{T.}\binits{T.}}
(\byear{2021}).
\btitle{Physical observables to determine the nature of membrane-less cellular sub-compartments}.
\bnote{Pages: 2021.04.01.438041 Section: New Results}.
\bid{doi={10.1101/2021.04.01.438041}}
\end{bmisc}
%
\OrigBibText
%
\begin{bmisc}[author]
\bauthor{\bsnm{Heltberg},~\bfnm{Mathias~Luidor}\binits{M.~L.}},
\bauthor{\bsnm{Min\'{e}-Hattab},~\bfnm{Judith}\binits{J.}},
\bauthor{\bsnm{Taddei},~\bfnm{Angela}\binits{A.}},
\bauthor{\bsnm{Walczak},~\bfnm{Aleksandra~M.}\binits{A.~M.}} \AND
\bauthor{\bsnm{Mora},~\bfnm{Thierry}\binits{T.}} (\byear{2021}).
\btitle{Physical observables to determine the nature of membrane-less cellular sub-compartments}.
\bnote{Pages: 2021.04.01.438041 Section: New Results}.
\end{bmisc}
%
\endOrigBibText
\bptok{imsref}%
\endbibitem

\bibitem{hoffmann_nonparametric_2024}
\begin{barticle}[mr]
\bauthor{\bsnm{Hoffmann},~\bfnm{Marc}\binits{M.}} \AND
\bauthor{\bsnm{Ray},~\bfnm{Kolyan}\binits{K.}}
(\byear{2025}).
\btitle{Nonparametric {B}ayesian estimation in a multidimensional diffusion model with high frequency data}.
\bjournal{Probab. Theory Related Fields}
\bvolume{191}
\bpages{103--180}.
\bid{doi={10.1007/s00440-024-01317-w}, doi={10.1007/s00440-024-01317-w}, mr={4869254}, issn={0178-8051,1432-2064}}
\end{barticle}
%
\OrigBibText
%
\begin{barticle}[author]
\bauthor{\bsnm{Hoffmann},~\bfnm{Marc}\binits{M.}} \AND
\bauthor{\bsnm{Ray},~\bfnm{Kolyan}\binits{K.}} (\byear{2024}).
\btitle{Nonparametric Bayesian estimation in a multidimensional diffusion model with high frequency data}.
\bjournal{Probability Theory and Related Fields}.
\end{barticle}
%
\endOrigBibText
\bptok{imsref}%
\endbibitem

\bibitem{kirichenko_optimality_2015}
\begin{barticle}[mr]
\bauthor{\bsnm{Kirichenko},~\bfnm{Alisa}\binits{A.}} \AND
\bauthor{\bparticle{van} \bsnm{Zanten},~\bfnm{Harry}\binits{H.}}
(\byear{2015}).
\btitle{Optimality of {P}oisson processes intensity learning with {G}aussian processes}.
\bjournal{J. Mach. Learn. Res.}
\bvolume{16}
\bpages{2909--2919}.
\bid{mr={3450529}, issn={1532-4435,1533-7928}}
\end{barticle}
%
\OrigBibText
%
\begin{barticle}[author]
\bauthor{\bsnm{Kirichenko},~\bfnm{Alisa}\binits{A.}} \AND
\bauthor{\bsnm{Zanten},~\bfnm{Harry~van}\binits{H.~v.}} (\byear{2015}).
\btitle{Optimality of Poisson Processes Intensity Learning with Gaussian Processes}.
\bjournal{Journal of Machine Learning Research} \bvolume{16}
\bpages{2909--2919}.
\end{barticle}
%
\endOrigBibText
\bptok{imsref}%
\endbibitem

\bibitem{last_lectures_2017}
\begin{bbook}[mr]
\bauthor{\bsnm{Last},~\bfnm{G{\"{u}}nter}\binits{G.}} \AND
\bauthor{\bsnm{Penrose},~\bfnm{Mathew}\binits{M.}}
(\byear{2018}).
\btitle{Lectures on the {P}oisson Process}.
\bseries{Institute of Mathematical Statistics Textbooks}
\bvolume{7}.
\bpublisher{Cambridge Univ. Press},
\blocation{Cambridge}.
\bid{mr={3791470}}
\end{bbook}
%
\OrigBibText
%
\begin{bbook}[author]
\bauthor{\bsnm{Last},~\bfnm{G\"{u}nter}\binits{G.}} \AND
\bauthor{\bsnm{Penrose},~\bfnm{Mathew}\binits{M.}} (\byear{2017}).
\btitle{Lectures on the Poisson Process}.
\bseries{Institute of Mathematical Statistics Textbooks}.
\bpublisher{Cambridge University Press}, \baddress{Cambridge}.
\end{bbook}
%
\endOrigBibText
\bptok{imsref}%
\endbibitem

\bibitem{lions_non-homogeneous_1972}
\begin{bbook}[mr]
\bauthor{\bsnm{Lions},~\bfnm{J.-L.}\binits{J.-L.}} \AND
\bauthor{\bsnm{Magenes},~\bfnm{E.}\binits{E.}}
(\byear{1972}).
\btitle{Non-homogeneous Boundary Value Problems and Applications. {V}ol. {I}}.
\bseries{Die Grundlehren der Mathematischen Wissenschaften}
\bvolume{181}.
\bpublisher{Springer},
\blocation{New York}.
\bid{mr={0350177}}
\end{bbook}
%
\OrigBibText
%
\begin{bbook}[author]
\bauthor{\bsnm{Lions},~\bfnm{J.~L.}\binits{J.~L.}} \AND
\bauthor{\bsnm{Magenes},~\bfnm{E.}\binits{E.}} (\byear{1972}).
\btitle{Non-Homogeneous Boundary Value Problems and Applications}.
\bpublisher{Springer Berlin Heidelberg},
\baddress{Berlin, Heidelberg}. 
\end{bbook}
%
\endOrigBibText
\bptok{imsref}%
\endbibitem

\bibitem{lions_stochastic_1984}
\begin{barticle}[mr]
\bauthor{\bsnm{Lions},~\bfnm{P.-L.}\binits{P.-L.}} \AND
\bauthor{\bsnm{Sznitman},~\bfnm{A.-S.}\binits{A.-S.}}
(\byear{1984}).
\btitle{Stochastic differential equations with reflecting boundary conditions}.
\bjournal{Comm. Pure Appl. Math.}
\bvolume{37}
\bpages{511--537}.
\bid{doi={10.1002/cpa.3160370408}, doi={10.1002/cpa.3160370408}, mr={0745330}, issn={0010-3640,1097-0312}}
\end{barticle}
%
\OrigBibText
%
\begin{barticle}[author]
\bauthor{\bsnm{Lions},~\bfnm{P.~L.}\binits{P.~L.}} \AND
\bauthor{\bsnm{Sznitman},~\bfnm{A.~S.}\binits{A.~S.}} (\byear{1984}).
\btitle{Stochastic differential equations with reflecting boundary conditions}.
\bjournal{Communications on Pure and Applied Mathematics}
\bvolume{37} \bpages{511--537}. 
\end{barticle}
%
\endOrigBibText
\bptok{imsref}%
\endbibitem

\bibitem{monard_efficient_2019}
\begin{barticle}[mr]
\bauthor{\bsnm{Monard},~\bfnm{Fran{\c{c}}ois}\binits{F.}},
\bauthor{\bsnm{Nickl},~\bfnm{Richard}\binits{R.}} \AND
\bauthor{\bsnm{Paternain},~\bfnm{Gabriel~P.}\binits{G.~P.}}
(\byear{2019}).
\btitle{Efficient nonparametric {B}ayesian inference for {$X$}-ray transforms}.
\bjournal{Ann. Statist.}
\bvolume{47}
\bpages{1113--1147}.
\bid{doi={10.1214/18-AOS1708}, doi={10.1214/18-AOS1708}, mr={3909962}, issn={0090-5364,2168-8966}}
\end{barticle}
%
\OrigBibText
%
\begin{barticle}[author]
\bauthor{\bsnm{Monard},~\bfnm{Fran\c{c}ois}\binits{F.}},
\bauthor{\bsnm{Nickl},~\bfnm{Richard}\binits{R.}} \AND
\bauthor{\bsnm{Paternain},~\bfnm{Gabriel~P.}\binits{G.~P.}} (\byear{2019}).
\btitle{Efficient nonparametric Bayesian inference for X-ray transforms}.
\bjournal{The Annals of Statistics} \bvolume{47} \bpages{1113--1147}.
\bnote{Publisher: Institute of Mathematical Statistics}.
\end{barticle}
%
\endOrigBibText
\bptok{imsref}%
\endbibitem

\bibitem{monard_consistent_2021}
\begin{barticle}[mr]
\bauthor{\bsnm{Monard},~\bfnm{Fran{\c{c}}ois}\binits{F.}},
\bauthor{\bsnm{Nickl},~\bfnm{Richard}\binits{R.}} \AND
\bauthor{\bsnm{Paternain},~\bfnm{Gabriel~P.}\binits{G.~P.}}
(\byear{2021}).
\btitle{Consistent inversion of noisy non-{A}belian {X}-ray transforms}.
\bjournal{Comm. Pure Appl. Math.}
\bvolume{74}
\bpages{1045--1099}.
\bid{doi={10.1002/cpa.21942}, doi={10.1002/cpa.21942}, mr={4230066}, issn={0010-3640,1097-0312}}
\end{barticle}
%
\OrigBibText
%
\begin{barticle}[author]
\bauthor{\bsnm{Monard},~\bfnm{Fran\c{c}ois}\binits{F.}},
\bauthor{\bsnm{Nickl},~\bfnm{Richard}\binits{R.}} \AND
\bauthor{\bsnm{Paternain},~\bfnm{Gabriel~P.}\binits{G.~P.}} (\byear{2021}).
\btitle{Consistent inversion of noisy non-Abelian X-Ray transforms}.
\bjournal{Communications on Pure and Applied Mathematics}
\bvolume{74} \bpages{1045--1099}. 
\end{barticle}
%
\endOrigBibText
\bptok{imsref}%
\endbibitem

\bibitem{nickl_bayesian_2023}
\begin{bbook}[mr]
\bauthor{\bsnm{Nickl},~\bfnm{Richard}\binits{R.}}
(\byear{2023}).
\btitle{Bayesian Non-linear Statistical Inverse Problems}.
\bseries{Zurich Lectures in Advanced Mathematics}.
\bpublisher{EMS Press},
\blocation{Berlin}.
\bid{doi={10.4171/zlam/30}, doi={10.4171/zlam/30}, mr={4604099}}
\end{bbook}
%
\OrigBibText
%
\begin{bbook}[author]
\bauthor{\bsnm{Nickl},~\bfnm{Richard}\binits{R.}} (\byear{2023}).
\btitle{Bayesian Non-linear Statistical Inverse Problems}.
\bnote{ISBN: 9783985470532 9783985475537 ISSN: 2943-4963, 2943-4971}.
\end{bbook}
%
\endOrigBibText
\bptok{imsref}%
\endbibitem

\bibitem{nickl_consistent_2023}
\begin{barticle}[mr]
\bauthor{\bsnm{Nickl},~\bfnm{Richard}\binits{R.}}
(\byear{2024}).
\btitle{Consistent inference for diffusions from low frequency measurements}.
\bjournal{Ann. Statist.}
\bvolume{52}
\bpages{519--549}.
\bid{doi={10.1214/24-aos2357}, doi={10.1214/24-aos2357}, mr={4744186}, issn={0090-5364,2168-8966}}
\end{barticle}
%
\OrigBibText
%
\begin{barticle}[author]
\bauthor{\bsnm{Nickl},~\bfnm{Richard}\binits{R.}} (\byear{2024}).
\btitle{Consistent inference for diffusions from low frequency measurements}.
\bjournal{Ann. Statist.} \bvolume{52} \bpages{519--549}.
\end{barticle}
%
\endOrigBibText
\bptok{imsref}%
\endbibitem

\bibitem{NR20}
\begin{barticle}[mr]
\bauthor{\bsnm{Nickl},~\bfnm{Richard}\binits{R.}} \AND
\bauthor{\bsnm{Ray},~\bfnm{Kolyan}\binits{K.}}
(\byear{2020}).
\btitle{Nonparametric statistical inference for drift vector fields of multi-dimensional diffusions}.
\bjournal{Ann. Statist.}
\bvolume{48}
\bpages{1383--1408}.
\bid{doi={10.1214/19-AOS1851}, doi={10.1214/19-AOS1851}, mr={4124327}, issn={0090-5364,2168-8966}}
\end{barticle}
%
\OrigBibText
%
\begin{barticle}[author]
\bauthor{\bsnm{Nickl},~\bfnm{Richard}\binits{R.}} \AND
\bauthor{\bsnm{Ray},~\bfnm{Kolyan}\binits{K.}} (\byear{2020}).
\btitle{Nonparametric statistical inference for drift vector fields of multi-dimensional diffusions}.
\bjournal{Ann. Statist.} \bvolume{48} \bpages{1383--1408}.
\end{barticle}
%
\endOrigBibText
\bptok{imsref}%
\endbibitem

\bibitem{nickl2025inferring_supplement}
\begin{bmisc}[author]
\bauthor{\bsnm{Nickl},~\binits{R.}} \AND
\bauthor{\bsnm{Seizilles},~\binits{F.}}
(\byear{2026}).
\bhowpublished{Supplement to ``Inferring diffusivity from killed diffusion.''
\url{https://doi.org/10.1214/26-AOS2633SUPP}}
\end{bmisc}
\bptok{imsref}%
\endbibitem

\bibitem{NS17}
\begin{barticle}[mr]
\bauthor{\bsnm{Nickl},~\bfnm{Richard}\binits{R.}} \AND
\bauthor{\bsnm{S{\"{o}}hl},~\bfnm{Jakob}\binits{J.}}
(\byear{2017}).
\btitle{Nonparametric {B}ayesian posterior contraction rates for discretely observed scalar diffusions}.
\bjournal{Ann. Statist.}
\bvolume{45}
\bpages{1664--1693}.
\bid{doi={10.1214/16-AOS1504}, doi={10.1214/16-AOS1504}, mr={3670192}, issn={0090-5364,2168-8966}}
\end{barticle}
%
\OrigBibText
%
\begin{barticle}[author]
\bauthor{\bsnm{Nickl},~\bfnm{Richard}\binits{R.}} \AND
\bauthor{\bsnm{S\"ohl},~\bfnm{Jakob}\binits{J.}} (\byear{2017}).
\btitle{Nonparametric Bayesian posterior contraction rates for discretely observed scalar diffusions}.
\bjournal{Ann. Statist.} \bvolume{45} \bpages{1664--1693}.
\end{barticle}
%
\endOrigBibText
\bptok{imsref}%
\endbibitem

\bibitem{nickl_convergence_2020}
\begin{barticle}[mr]
\bauthor{\bsnm{Nickl},~\bfnm{Richard}\binits{R.}},
\bauthor{\bsnm{van~de Geer},~\bfnm{Sara}\binits{S.}} \AND
\bauthor{\bsnm{Wang},~\bfnm{Sven}\binits{S.}}
(\byear{2020}).
\btitle{Convergence rates for penalized least squares estimators in {PDE} constrained regression problems}.
\bjournal{SIAM/ASA J. Uncertain. Quantificat.}
\bvolume{8}
\bpages{374--413}.
\bid{doi={10.1137/18M1236137}, doi={10.1137/18M1236137}, issn={2166-2525}, mr={4074017}}
\end{barticle}
%
\OrigBibText
%
\begin{barticle}[author]
\bauthor{\bsnm{Nickl},~\bfnm{Richard}\binits{R.}},
\bauthor{\bsnm{Van De~Geer},~\bfnm{Sara}\binits{S.}} \AND
\bauthor{\bsnm{Wang},~\bfnm{Sven}\binits{S.}} (\byear{2020}).
\btitle{Convergence Rates for Penalized Least Squares Estimators in PDE Constrained Regression Problems}.
\bjournal{SIAM/ASA Journal on Uncertainty Quantification} \bvolume{8}
\bpages{374--413}. 
\end{barticle}
%
\endOrigBibText
\bptok{imsref}%
\endbibitem

\bibitem{nickl_polynomial-time_2022}
\begin{barticle}[mr]
\bauthor{\bsnm{Nickl},~\bfnm{Richard}\binits{R.}} \AND
\bauthor{\bsnm{Wang},~\bfnm{Sven}\binits{S.}}
(\byear{2024}).
\btitle{On polynomial-time computation of high-dimensional posterior measures by {L}angevin-type algorithms}.
\bjournal{J. Eur. Math. Soc. (JEMS)}
\bvolume{26}
\bpages{1031--1112}.
\bid{doi={10.4171/jems/1304}, doi={10.4171/jems/1304}, mr={4721029}, issn={1435-9855,1435-9863}}
\end{barticle}
%
\OrigBibText
%
\begin{barticle}[author]
\bauthor{\bsnm{Nickl},~\bfnm{Richard}\binits{R.}} \AND
\bauthor{\bsnm{Wang},~\bfnm{Sven}\binits{S.}} (\byear{2024}).
\btitle{On polynomial-time computation of high-dimensional posterior measures by Langevin-type algorithms}.
\bjournal{J. Eur. Math. Soc. (JEMS)} \bvolume{26} \bpages{1031--1112}.
\end{barticle}
%
\endOrigBibText
\bptok{imsref}%
\endbibitem

\bibitem{williams_markov_2000}
\begin{bbook}[mr]
\bauthor{\bsnm{Rogers},~\bfnm{L.~C.~G.}\binits{L.~C.~G.}} \AND
\bauthor{\bsnm{Williams},~\bfnm{David}\binits{D.}}
(\byear{2000}).
\btitle{Diffusions, {M}arkov Processes, and Martingales. {V}ol. 2}.
\bseries{Cambridge Mathematical Library}.
\bpublisher{Cambridge Univ. Press},
\blocation{Cambridge}.
\bid{doi={10.1017/CBO9781107590120}, doi={10.1017/CBO9781107590120}, mr={1780932}}
\end{bbook}
%
\OrigBibText
%
\begin{bbook}[author]
\bauthor{\bsnm{Rogers},~\bfnm{L.~C.~G.}\binits{L.~C.~G.}} \AND
\bauthor{\bsnm{Williams},~\bfnm{David}\binits{D.}} (\byear{2000}).
\btitle{Diffusions, Markov processes, and martingales. Vol. 2}.
\bseries{Cambridge Mathematical Library}.
\bpublisher{Cambridge University Press, Cambridge}
\bnote{It\^o calculus, Reprint of the second (1994) edition}.
\end{bbook}
%
\endOrigBibText
\bptok{imsref}%
\endbibitem

\bibitem{stuart_inverse_2010}
\begin{barticle}[mr]
\bauthor{\bsnm{Stuart},~\bfnm{A.~M.}\binits{A.~M.}}
(\byear{2010}).
\btitle{Inverse problems: A {B}ayesian perspective}.
\bjournal{Acta Numer.}
\bvolume{19}
\bpages{451--559}.
\bid{doi={10.1017/S0962492910000061}, doi={10.1017/S0962492910000061}, mr={2652785}, issn={0962-4929,1474-0508}}
\end{barticle}
%
\OrigBibText
%
\begin{barticle}[author]
\bauthor{\bsnm{Stuart},~\bfnm{A.~M.}\binits{A.~M.}} (\byear{2010}).
\btitle{Inverse problems: A Bayesian perspective}.
\bjournal{Acta Numerica} \bvolume{19} \bpages{451--559}.
\end{barticle}
%
\endOrigBibText
\bptok{imsref}%
\endbibitem

\bibitem{tanaka_stochastic_1979}
\begin{barticle}[mr]
\bauthor{\bsnm{Tanaka},~\bfnm{Hiroshi}\binits{H.}}
(\byear{1979}).
\btitle{Stochastic differential equations with reflecting boundary condition in convex regions}.
\bjournal{Hiroshima Math. J.}
\bvolume{9}
\bpages{163--177}.
\bid{mr={0529332}, issn={0018-2079,2758-9641}}
\end{barticle}
%
\OrigBibText
%
\begin{barticle}[author]
\bauthor{\bsnm{Tanaka},~\bfnm{Hiroshi}\binits{H.}} (\byear{1979}).
\btitle{Stochastic differential equations with reflecting boundary condition in convex regions}.
\bjournal{Hiroshima Mathematical Journal} \bvolume{9}
\bpages{163--177}.
\bnote{Publisher: Hiroshima University, Mathematics Program}.
\end{barticle}
%
\endOrigBibText
\bptok{imsref}%
\endbibitem

\bibitem{taylor_partial_2010}
\begin{bbook}[mr]
\bauthor{\bsnm{Taylor},~\bfnm{Michael~E.}\binits{M.~E.}}
(\byear{2011}).
\btitle{Partial Differential Equations {I}. {B}asic Theory},
\bedition{2nd} ed.
\bseries{Applied Mathematical Sciences}
\bvolume{115}.
\bpublisher{Springer},
\blocation{New York}.
\bid{doi={10.1007/978-1-4419-7055-8}, doi={10.1007/978-1-4419-7055-8}, mr={2744150}}
\end{bbook}
%
\OrigBibText
%
\begin{bbook}[author]
\bauthor{\bsnm{Taylor},~\bfnm{Michael~E.}\binits{M.~E.}} (\byear{2010}).
\btitle{Partial Differential Equations I: Basic Theory}.
\bpublisher{Springer Science \& Business Media}.
\end{bbook}
%
\endOrigBibText
\bptok{imsref}%
\endbibitem

\end{thebibliography}
\end{document}